\documentclass[12pt]{article}

\usepackage{comment} 

\usepackage{tikz}

\usepackage{tikz-graph}

\usepackage{amsmath, amsthm, amssymb}
\usepackage{enumerate}
\newtheorem{theorem}{}[section]
\newtheorem{lemma}[theorem]{}

\newtheoremstyle{styleclaim}{}{}{\itshape}{}{}{}{ }{(\thmnumber{#2})}

\theoremstyle{styleclaim}

\newcommand{\cref}[1]{(\ref{#1})}

\begin{document}

\title{Three-coloring graphs with no induced seven-vertex path I : the triangle-free case}
\date{\today}
\author{Maria Chudnovsky\thanks{Columbia University, New York, NY 10027, USA. E-mail: mchudnov@columbia.edu. Partially supported by NSF grants IIS-1117631, DMS-1001091 and DMS-1265803.}, Peter Maceli\thanks{Columbia University, New York, NY 10027, USA. E-mail: plm2109@columbia.edu.}, and Mingxian Zhong\thanks{Columbia University, New York, NY 10027, USA. E-mail: mz2325@columbia.edu.}}

\maketitle

\begin{abstract}

In this paper, we give a polynomial time algorithm which determines if a given triangle-free graph with no induced seven-vertex path is $3$-colorable, and gives an explicit coloring if one exists.

\end{abstract}

\section{Introduction}

We start with some definitions.
All graphs in this paper are finite and simple. Let $G$ be a graph and $X$ be a subset of $V(G)$. We denote by $G[X]$ \textit{the subgraph of $G$ induced by $X$}, that is, the subgraph of $G$ with vertex set $X$ such that two vertices are adjacent in $G[X]$ if and only if they are adjacent in $G$. We denote by 
$G\setminus X$ the graph $G[V(G) \setminus X]$. If $X=\{v\}$ for some $v \in V(G)$, we write $G \setminus v$ instead of $G \setminus \{v\}$.
Let $H$ be a graph. If $G$ has no induced subgraph isomorphic to $H$, then we say that $G$ is \textit{$H$-free}. For a family $\mathcal{F}$ of graphs, we say that $G$ is \textit{$\mathcal{F}$-free} if $G$ is $F$-free for every $F\in \mathcal{F}$. If $G$ is not $H$-free, then \textit{$G$ contains $H$}. If $G[X]$ is isomorphic to $H$, then we say that \textit{$X$ is an $H$ in $G$}. 

For $n\geq 0$, we denote by \textit{$P_{n+1}$ the path with $n+1$ vertices and length $n$}, that is, the graph with distinct vertices $\{p_0,p_1,...,p_n\}$ such that $p_i$ is adjacent to $p_j$ if and only if $|i-j|=1$. For $n\geq 3$, we denote by \textit{$C_n$ the cycle of length $n$}, that is, the graph with distinct vertices $\{c_1,...,c_n\}$ such that $c_i$ is adjacent to $c_j$ if and only if $|i-j|=1$ or $n-1$. By convention, when explicitly describing a path or a cycle, we always list the vertices in order. Let $G$ be a graph. When $G[\{p_0,p_1,...,p_n\}]$ is the path $P_{n+1}$, we say that \textit{$p_0-p_1-...-p_n$ is a $P_{n+1}$ in $G$}. Similarly, when $G[\{c_1,c_2,...,c_n\}]$ is the cycle $C_n$, we say that \textit{$c_1-c_2-...-c_n-c_1$ is a $C_n$ in $G$}. For $n\geq 3$, an \textit{n-gon} in a graph $G$ is an induced subgraph of $G$ isomorphic to $C_n$. We also refer to a cycle of length three as a \textit{triangle}.  Lastly, suppose $C$ is a 6-gon in $G$ given by $v_0-v_1-v_2-v_3-v_4-v_5-v_0$. We say that $(C,p)$, or  $v_0-v_1-v_2-v_3-v_4-v_5-v_0$ with $p$, is a \textit{shell} in $G$, provided that $p\in V(G)\setminus\{v_0,...,v_5\}$ is such that $N(p)\cap\{v_0,...,v_5\}=\{v_\ell,v_{\ell+3}\}$ for some $\ell\in \{0,1,2\}$. The shell is drawn in Figure \ref{S7}.

\begin{figure}\label{S7}
\begin{center}

\begin{tikzpicture}

\GraphInit[vstyle=Simple]

 \tikzset{VertexStyle/.style = {shape = circle,fill = black,minimum size = 7pt,inner sep=0pt}}

\Vertex[LabelOut,Lpos=180,x=-1.732,y=-1]{$v_4$}

\Vertex[LabelOut,Lpos=270,x=0,y=-2]{$v_3$}

\Vertex[LabelOut,Lpos=0,x=1.732,y=-1]{$v_2$}

\Vertex[LabelOut,Lpos=0,x=1.732,y=1]{$v_1$}

\Vertex[LabelOut,Lpos=90,x=0,y=2]{$v_0$}

\Vertex[LabelOut,Lpos=180,x=-1.732,y=1]{$v_5$}

\Vertex[LabelOut,Lpos=180,x=0,y=0]{$p$}

\Edges($v_4$,$v_3$,$v_2$,$v_1$,$v_0$,$v_5$,$v_4$)
\Edges($v_3$,$p$,$v_0$)

\end{tikzpicture}

\end{center}

\caption{The shell.}
\end{figure}
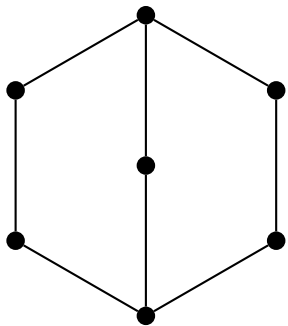

A \textit{$k$-coloring} of a graph $G$ is a mapping $c:V(G)\rightarrow \{1,...,k\}$ such that if $x,y\in V(G)$ are adjacent, then $c(x)\ne c(y)$. If a $k$-coloring exists for a graph $G$, we say that the $G$ is \textit{$k$-colorable}. The COLORING problem is determining the smallest integer $k$ such that a given graph is $k$-colorable, and was one of the initial problems R.M.Karp \cite{KARP} showed to be NP-complete. For fixed $k\geq 1$, the $k$-COLORING problem is deciding whether a given graph is $k$-colorable. Since Stockmeyer \cite{STOCK} showed that for any $k\geq 3$ the $k$-COLORING problem is NP-complete, there has been much interest in deciding for which classes of graphs coloring problems can be solved in polynomial time. In this paper, the general approach that we consider is to fix a graph $H$ and consider the $k$-COLORING problem restricted to the class of $H$-free graphs. 

We call a graph \textit{acyclic} if it is $C_n$-free for all $n\geq 3$. The \textit{girth} of a graph is the length of its shortest cycle, or infinity if the graph is acyclic. Kami\'nski and Lozin \cite{CYCLES} proved:

\begin{lemma}\label{HFREE}

For any fixed $k,g\geq 3$, the $k$-COLORING problem is NP-complete for the class of graphs with girth at least $g$.

\end{lemma} 

\noindent As a consequence of \ref{HFREE}, it follows that if the graph $H$ contains a cycle, then for any fixed $k\geq 3$, the $k$-COLORING problem is NP-complete for the class of $H$-free graphs. The \textit{claw} is the graph with vertex set $\{a_0, a_1, a_2, a_3\}$ and edge set $\{a_0a_1, a_0a_2, a_0a_3\}$. A theorem of Holyer \cite{LINE} together with an extension due to Leven and Galil \cite{LINE1} imply the following:

\begin{lemma}\label{CLAW}

If a graph $H$ contains the claw, then for every $k\geq 3$, the $k$-COLORING problem is NP-complete for the class of $H$-free graphs. 

\end{lemma}

\noindent Hence, the remaining problem of interest is deciding the $k$-COLORING problem for the class of $H$-free graphs where $H$ is a fixed acyclic claw-free graph. It is easily observed that every component of an acyclic claw-free graph is a path. And so, we focus on the $k$-COLORING problem for the class of $H$-free graphs where $H$ is a connected acyclic claw-free graph, that is, simply a path. Ho\`{a}ng, Kami\'nski, Lozin, Sawada, and Shu \cite{P5} proved the following:

\begin{lemma}

For every $k$, the $k$-COLORING problem can be solved in polynomial time for the class of $P_5$-free graphs. 

\end{lemma}

\noindent Additionally, Randerath and Schiermeyer \cite{P6} showed that:

\begin{lemma}

The 3-COLORING problem can be solved in polynomial time for the class of $P_6$-free graphs. 

\end{lemma}

\noindent While, Huang \cite{4P7} recently showed that:

\begin{lemma}

The following problems are NP-complete:
\begin{enumerate}

\item The $5$-COLORING problem is NP-complete for the class of $P_6$-free graphs. 

\item The $4$-COLORING problem is NP-complete for the class of $P_7$-free graphs. 

\end{enumerate}

\end{lemma}

\noindent Thus, the remaining open cases of the $k$-COLORING problem for $P_\ell$-free graphs are the following:

\begin{enumerate}

\item The $4$-COLORING problem for the class of $P_6$-free graphs. 

\item The 3-COLORING problem for the class of $P_\ell$-free graphs where $\ell\geq 7$. 

\end{enumerate}

Toward extending these polynomial results, it is convenient to consider the following more general coloring problem. A \textit{palette $L$} of a graph $G$ is a mapping which assigns each vertex $v\in V(G)$ a finite non-empty subset of $\mathbb{N}$, denoted by $L(v)$. A \textit{subpalette} of a palette $L$ of $G$ is a palette $L'$ of $G$ such that $L'(v)\subseteq L(v)$ for all $v\in V(G)$. We say a palette $L$ of the graph $G$ has \textit{order} $k$ if $L(v)\subseteq \{1,...,k\}$ for all $v\in V(G)$. Notationally, we write $(G,L)$ to represent a graph $G$ and a palette $L$ of $G$. We say that a $k$-coloring $c$ of $G$ is a \textit{coloring of $(G,L)$} provided $c(v)\in L(v)$ for all $v\in V(G)$. We say $(G,L)$ is \textit{colorable}, if there exists a coloring of $(G,L)$. 
We denote by $(G,\mathcal{L})$ a graph $G$ and a collection $\mathcal{L}$ of palettes of $G$. We say $(G,\mathcal{L})$ is \textit{colorable} if $(G,L)$ is colorable for some $L\in\mathcal{L}$, and $c$ is a {\em coloring} of $(G,\mathcal{L})$ if $c$ is a coloring of $(G,L)$ for some $L\in\mathcal{L}$.

Let $G$ be a graph. A subset $D$ of $V(G)$ is called a \textit{dominating set}, if every vertex in $V(G)\setminus D$ is adjacent to at least one vertex in $D$. Given $(G,L)$, consider a subset $X\subseteq V(G)$ such that $|L(x)|=1$ for all $x\in X$. For a subset $Y\subseteq V(G)\setminus X$, we say that we \textit{update the palettes of the vertices in $Y$ with respect to $X$}, if for all $y\in Y$ we set $$L(y)=L(y)\setminus (\bigcup_{u\in N(y)\cap X \text{ with } |L(u)|=1}L(u)).$$ Note that updating can be carried out in time $O(|V(G)|^2)$. 

By reducing to an instance of $2$-SAT, which Aspvall, Plass and Tarjan \cite{PSAT} showed can be solved in linear time, Edwards \cite{2SAT} proved the following:

\begin{lemma}\label{check} There is an algorithm with the following specifications:

\bigskip

{\bf Input:} A palette $L$ of a graph $G$ such that $|L(v)|\leq 2$ for all $v\in V(G)$.

\bigskip

{\bf Output:} A coloring of $(G,L)$, or a determination that none exists.

\bigskip

{\bf Running time:} $O(|V(G)|^2)$.

\end{lemma}

\noindent Let $G$ be a graph. A subset $S$ of $V(G)$ is called \textit{monochromatic} with respect to a given coloring $c$ of $G$ if $c(u)=c(v)$ for all $u,v\in S$. For a palette $L$, and a set $X$ of  subsets of $V(G)$,
we say that $(G,L,X)$  is {\em colorable} if there is a coloring $c$ of 
$(G,L)$ such that $S$ is monochromatic with respect to $c$ for all
$S \in X$. The proof of \ref{check} is easily modified to obtain the following generalization \cite{2SAT+}:

\begin{lemma}\label{checkSubsets} There is an algorithm with the following specifications:

\bigskip

{\bf Input:} A palette $L$ of a graph $G$ such that $|L(v)|\leq 2$ for all $v\in V(G)$, together with a set $X$ of subsets of $V(G)$.

\bigskip

{\bf Output:} A coloring of $(G,L,X)$, or a determination that none exists.

\bigskip

{\bf Running time:} $O(|X||V(G)|^2)$.

\end{lemma}

\noindent Applying \ref{check} yields the following general approach for $3$-coloring a graph. Let $G$ be a graph, and suppose $D\subseteq V(G)$ is a dominating set. Initialize the order 3 palette $L$ of $G$ by setting $L(v)=\{1,2,3\}$ for all $v\in V(G)$. Consider a fixed $3$-coloring $c$ of $G[D]$, and let $L_c$ be the subpalette of $L$ obtained by updating the palettes of the vertices in $V(G)\setminus D$ with respect to $D$. By construction, $(G,L_c)$ is colorable if and only if the coloring $c$ of $G[D]$ can be extended to a $3$-coloring of $G$.  Since $|L_c(v)|\leq 2$ for all $v\in V(G)$,  \ref{check} allows us to efficiently test
if $(G,L_c)$ is colorable. Let $\mathcal{L}$ to be the set of all such palettes $L_c$ where $c$ is a $3$-coloring of $G[D]$. 
It follows that $G$ is 3-colorable if and only if $(G,\mathcal{L})$ is colorable. Assuming we can efficiently produce a dominating set $D$ of bounded size, since there are at most $3^{|D|}$ ways to 3-color $G[D]$, it follows that we can efficiently test if  $(G,\mathcal{L})$ is colorable, and so we can decide if $G$ is 3-colorable in polynomial time. This method figures prominently in the polynomial time algorithms for the 3-COLORING problem for the class of $P_\ell$-free graphs where $\ell\leq 5$. However, this approach needs to be modified when considering the class of $P_\ell$-free graphs when $\ell\geq 6$, since a dominating set of bounded size may not exist. Very roughly, the techniques used in this paper may be described as such a modification.  

In this paper and \cite{PART2}, we prove that the 3-COLORING problem can be solved in polynomial time for the class of $P_7$-free graphs. Here we consider the triangle-free case and prove the following:

\begin{lemma}\label{half1}There is an algorithm with the following specifications:

\bigskip

{\bf Input:} A $\{P_7,C_3\}$-free graph $G$.

\bigskip

{\bf Output:} A $3$-coloring of $G$, or a determination that none exists.

\bigskip

{\bf Running time:} $O(|V(G)|^{18})$.

\end{lemma}

\bigskip

\noindent Here is a brief outline of the algorithm. Consider a $\{P_7,C_3\}$-free graph $G$. We begin by establishing two polynomial time procedures \ref{lemma1} and \ref{lemma2} which determine if a $3$-coloring of a specific induced subgraph of $G$ extends to a coloring of $G$, and gives an explicit $3$-coloring if one exists. More specifically, given an order 3 palette $L$ of $G$, and a set $X$ of subsets of $V(G)$, \ref{lemma1} and \ref{lemma2} allow us to reduce determining if $(G,L,X)$ is colorable to determining if one of polynomially many triples $(G',L',X')$ is colorable, where each of $(G,L',X)$ is ``closer'' than $(G,L,X)$ to being of the form required by~\ref{checkSubsets}. Next, we introduce a polynomial time ``cleaning" procedure \ref{cleaning}, which preprocesses the graph $G$ so that we can apply \ref{lemma1} and \ref{lemma2}. Next, we use \ref{lemma2} to show that if $G$ contains a 7-gon, then in polynomial time we can either produce a $3$-coloring of $G$, or determine that none exists. And so, we may assume $G$ is a $\{P_7,C_3,C_7\}$-free graph. Next, we use \ref{lemma2} to show that if $G$ contains a shell, then in polynomial time we can either produce a $3$-coloring of $G$, or determine that none exists. And so, we may assume $G$ is a $\{P_7,C_3,C_7,shell\}$-free graph. Finally, we use \ref{lemma1} to show that if $G$ contains a 5-gon, then in polynomial time we can either produce a $3$-coloring of $G$, or determine that none exists. And so, we may assume $G$ is a $\{P_7,C_3,C_5,C_7\}$-free graph. Since $G$ is $P_7$-free, it follows that $G$ is $C_k$-free for all $k>7$. And so, $G$ is bipartite, and we can easily produce a 2-coloring of $G$, thus, establishing \ref{half1}. 

\bigskip

\noindent In \cite{PART2}, using different techniques, we prove the following:

\begin{lemma}\label{half2} There is a polynomial time algorithm with the following specifications:

\bigskip

{\bf Input:} A $P_7$-free graph $G$ which contains a triangle.

\bigskip

{\bf Output:} A $3$-coloring of $G$, or a determination that none exists.

\end{lemma}

\noindent Together, \ref{half1} and \ref{half2}, imply the following:

\begin{lemma}\label{both} There is a polynomial time algorithm with the following specifications:

\bigskip

{\bf Input:} A $P_7$-free graph $G$.

\bigskip

{\bf Output:} A $3$-coloring of $G$, or a determination that none exists.

\end{lemma}

\noindent This paper is organized as follows. In section 2 we prove \ref{lemma1}
and in section 3 we prove \ref{lemma2}.
In section 4, we give a preprocessing procedure \ref{cleaning}, so that we can apply \ref{lemma1} and \ref{lemma2} to a given $\{P_7,C_3\}$-free graph. In section 5 we prove a lemma that allows us to identify more easily situations where 
\ref{lemma1} and \ref{lemma2} are applicable.
In section 6, we prove \ref{7gon}, which shows that if a $\{P_7,C_3\}$-free graph contains a 7-gon, then 3-COLORING can be solved in polynomial time. In section 7, we prove \ref{shell}, which shows that if a $\{P_7,C_3,C_7\}$-free graph contains a shell, then  3-COLORING can be solved in polynomial time. In section 8, we prove \ref{5gon}, which shows that if a $\{P_7,C_3,C_7,shell\}$-free graph contains a 5-gon, then  $3$-COLORING can be solved in polynomial time. Finally, in section 9, we tie everything together and give a formal proof of \ref{half1}.

\section{Reducing the Palettes: Part I}

In this section, we give a polynomial time procedure \ref{lemma1} which, 
given a $\{P_7,C_3\}$-free graph $G$ with  palette $L$, and a set of subsets $X$ of $V(G)$, under certain circumstances, allows us  to  reduce determining if 
$(G,L,X)$ is colorable to determining if one
of polynomially many triples $(G,L',X)$ is colorable, where
each $(G,L',X)$ is ``closer'' than $(G,L,X)$ to the form required by
\ref{checkSubsets}. More precisely, more vertices have lists of size at most 
two in the palette $L'$ than in the palette $L$. We begin with some 
definitions,  and easy lemma and algorithm.

Let $G$ be a graph. A \textit{clique} in $G$ is a set of vertices all pairwise adjacent. A \textit{stable set} in $G$ is a set of vertices all pairwise non-adjacent. The \textit{neighborhood} of a vertex $v\in V(G)$ is the set of all vertices adjacent to $v$, and is denoted $N(v)$. The \textit{degree} of a vertex $v\in V(G)$ is $|N(v)|$, and is denoted $deg(v)$. A \textit{partition} of a set $S$ is a collection of disjoint subsets of $S$ whose union is $S$. Let $A$ and $B$ be disjoint subsets of $V(G)$. For a vertex $b\in V(G)\setminus A$, we say that \textit{$b$ is complete to $A$} if $b$ is adjacent to every vertex of $A$, and that \textit{$b$ is anticomplete to $A$} if $b$ is non-adjacent to every vertex of $A$. If every vertex of $A$ is complete to $B$, we say \textit{$A$ is complete to $B$}, and if every vertex of $A$ is anticomplete to $B$, we say that \textit{$A$ is anticomplete to $B$}. If $b\in V(G)\setminus A$ is neither complete nor anticomplete to $A$, we say that \textit{$b$ is mixed on $A$}. We say $G$ is \textit{connected} if $V(G)$ cannot be partitioned into two disjoint non-empty sets anticomplete to each other. The \textit{complement} $\overline{G}$ of $G$ is the graph with vertex set $V(G)$ such that two vertices are adjacent in $\overline{G}$ if and only if they are non-adjacent in $G$. If $\overline{G}$ is connected we say that $G$ is \textit{anticonnected}. For $X \subseteq V(G)$, we say that
$X$ is {\em connected} if $G[X]$ is connected, and that $X$ is anticonnected
if $G[X]$ is anticonnected. A \textit{component} of $X \subseteq V(G)$ is a maximal connected subset of $X$, and an \textit{anticomponent} of $X$ is a maximal anticonnected subset of $X$.

\begin{lemma}\label{half'} Let $G$ be a bipartite $\overline{C_4}$-free graph
with bipartition $(A,B)$. If $a,a'\in A$ are such that $deg(a)\leq deg(a')$, 
then $N(a)\subseteq N(a')$.

\end{lemma}

\begin{proof}

Suppose not, and so there exists $b\in N(a)\setminus N(a')$. Since $|N(a)|\leq|N(a')|$, it follows that there exists $b'\in N(a')\setminus N(a)$. However, then $\{a,b,a',b'\}$ is a $\overline{C_4}$ in $G$, a contradiction. This proves \ref{half'}.

\end{proof}

\begin{lemma}\label{half} There is an algorithm with the following specifications:

\bigskip

{\bf Input:} A bipartite $\overline{C_4}$-free graph $G$ together with a bipartition $V(G)=A\cup B$.

\bigskip

{\bf Output:} A partition  $A_1\cup ...\cup A_q$ of $A$ and an ordering $\{b_1,...,b_{|B|}\}$ of the vertices in $B$ such that for every $i\in \{1,...,q\}$ and $j\in \{1,...,|B|\}$ the following hold:

\bigskip

\begin{enumerate}
\item If $a,a'\in A_i$, then $N(a)=N(a')$, and 
\item If $b_j$ is complete to $A_i$, then $A_i\cup ... \cup A_q$ is complete 
to $\{b_j,...,b_{|B|}\}$.
\end{enumerate}

\bigskip

{\bf Running time:} $O(|V(G)|^2)$.

\end{lemma}

\begin{proof}

In time $O(|V(G)|^2)$ we can compute the degree of each vertex in $G$, and sort the vertices of $A$ and $B$ by degree, thus obtaining a labeling $a_1,...,a_{|A|}$ of $A$ such that $deg(a_1)\leq ...\leq deg(a_{|A|})$, and a labeling $b_1,...,b_{|B|}$ of $B$ such that $deg(b_1)\leq ...\leq deg(b_{|B|})$. Now, let $q=deg(a_{|A|})$, and for each $i\in \{1,...,q\}$ define $A_i=\{a\in A: deg(a)=i\}$. By applying \ref{half'} twice, it follows that if $a,a'\in A_i$, then $N(a)=N(a')$. Next, suppose $b_j$ is complete to $A_i$ for some $i\in \{1,...,q\}$ and $j\in \{1,...,|B|\}$, which implies $A_i\subseteq N(b_j)$ and $b_j\in N(a)$ for all $a\in A_i$. Since $deg(b_j)\leq...\leq deg(b_{|B|})$, by \ref{half'}, it follows that $A_i$ is complete to $\{b_j,...,b_{|B|}\}$. And, since $deg(a)\geq i$ for all $a\in A_{i}\cup...\cup A_q$, by \ref{half'}, it follows that $\{b_j,...,b_{|B|}\}$ is complete to $A_i\cup ...\cup A_q$. This proves \ref{half}.

\end{proof}

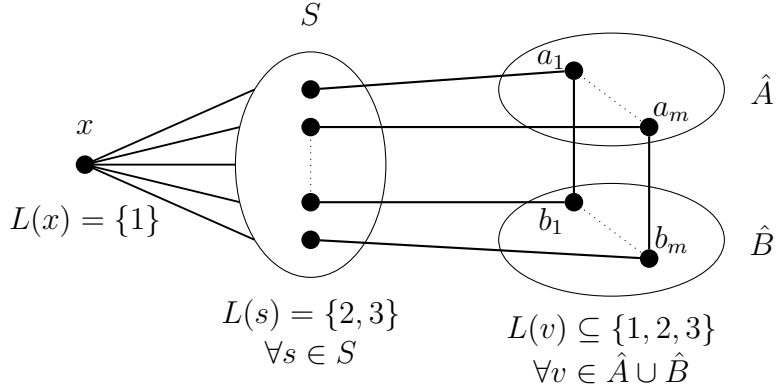
\begin{figure}[here]
\begin{center}

 \begin{tikzpicture}[scale=1]
 \GraphInit[vstyle=Simple]

 \tikzset{VertexStyle/.style = {shape = circle,fill = black,minimum size = 7pt,inner sep=0pt}}
 
 \Vertex[x=1,y=0]{1}

\draw [black, line width=.025cm](1,0)--(3.254,1);
\draw [black, line width=.025cm](1,0)--(3.06,.5);
\draw [black, line width=.025cm](1,0)--(3.254,-1);
\draw [black, line width=.025cm](1,0)--(3.06,-.5);
\draw [black, line width=.025cm](1,0)--(3,0);

\draw (4,0) ellipse (1cm and 1.5cm);

\draw (8,-1) ellipse (1.5cm and .75cm);
\draw (8,1) ellipse (1.5cm and .75cm);




 \Vertex[x=7.5,y=1.25]{2}
 \Vertex[x=7.5,y=-.5]{3}
 \Vertex[x=8.5,y=-1.25]{5}
 \Vertex[x=8.5,y=.5]{4}

\node at (7.2,1.4){$a_1$};
\node at (7.2,-.75){$b_1$};

\node at (8.8,.7){$a_m$};
\node at (8.8,-1){$b_m$};

\draw[dotted] (7.5,1.25)--(8.5,.5);
\draw[dotted] (7.5,-.5)--(8.5,-1.25);
\draw[dotted] (4,.5)--(4,-.5);

 \Vertex[x=4,y=1]{6}
 \Vertex[x=4,y=-.5]{7}
 \Vertex[x=4,y=-1]{9}
 \Vertex[x=4,y=.5]{8}

\Edges(6,2)
\Edges(7,3)

\Edges(8,4)
\Edges(9,5)


\Edges(2,3)
\Edges(4,5)

\node at (1,-.75) {$L(x)=\{1\}$};

\node at (1,.5) {$x$};

\node at (4,2) {$S$};

\node at (4,-2) {$L(s)=\{2,3\}$};
\node at (4,-2.5){$\forall s\in S$};

\node at (10,1) {$\hat{A}$};
\node at (10,-1) {$\hat{B}$};

\node at (8,-2.2) {$L(v)\subseteq \{1,2,3\}$};
\node at (8,-2.7){$\forall v\in \hat{A}\cup\hat{B}$};
 
\end{tikzpicture}

\caption{By \ref{lemma1}, when we encounter the above situation we can reduce determining if $(G,L,X)$ is colorable to determining if one of the triples 
$(G,L',X)$ is colorable for some $L'\in \mathcal{L}$, where each of $(G,L',X)$ is ``closer'' to the form required by~\ref{checkSubsets} (in particular, $L'(v) \leq 2$ for all $v \in \hat{A} \cup \hat{B}$).}

\end{center}

\end{figure}

\noindent The following is the main result of the section.

\begin{lemma}\label{lemma1}
 
Let $G$ be a $\{P_7,C_3\}$-free graph with $V(G)=\{x\}\cup S\cup \hat{A}\cup \hat{B}\cup Y$, where

\begin{itemize}
\item $x$ is complete to $S$ and anticomplete to $\hat{A}\cup \hat{B}\cup Y$,
\item $\hat{A}=\{a_1,...,a_t\}$ and $\hat{B}=\{b_1,...,b_t\}$ are stable,
\item for $i,j\in \{1,...,t\}$, $a_i$ is adjacent to $b_j$ if and only if $i=j$, and
\item each vertex of $\hat{A}\cup \hat{B}$ has a neighbor in $S$.

\end{itemize}

\noindent Let $L$ be an order 3 palette of $G$ such that 
$L(v) \subseteq \{2,3\}$ for every $v\in S$.

\noindent Let $X$ be a set of subsets of $V(G)$.
\bigskip

\noindent Then there exists a set $\mathcal{L}$ of $O(|V(G)|^2)$ subpalettes of $L$ such that

\bigskip

\noindent (a) For each $L'\in\mathcal{L}$, $L'(v)=L(v)$ for every $v \in \{x\}\cup S\cup Y$, and $|L'(v)| \leq 2$ for every $v \in \hat{A}\cup \hat{B}$, and

\bigskip
 
\noindent (b) $(G,L,X)$ is colorable if and only if $(G,L',X)$ is colorable for
at least one $L' \in \mathcal{L}$; and for every $L' \in \mathcal{L}$, every coloring of $(G,L',X)$ is a coloring of $(G,L,X)$.

\bigskip

\noindent Moreover, if the partition $\{x\}\cup S\cup \hat{A}\cup \hat{B}\cup Y$ of $V(G)$ is given, then $\mathcal{L}$ can be computed in time $O(|V(G)|^4)$.

\end{lemma}

\begin{proof}

Since $G$ is triangle-free, it follows that $S$ is stable, and that every vertex of $S$ is either anticomplete to or mixed on $\{a_i,b_i\}$ for every $i\in\{1,...,t\}$. Let $H$ be a bipartite graph with bipartition $V(H)=S\cup \{c_1,...,c_t\}$, where $s \in S$ is adjacent to $c_i$ in $H$ if and only if $s$ is mixed on $\{a_i,b_i\}$ in $G$. Note, $H$ can be constructed in time $O(|V(G)|^2)$.

\bigskip

\noindent \textit{(1) $H$ is a $\overline{C_4}$-free graph.}

\bigskip

\noindent Proof: Suppose not. Then there exist $s,s'\in S$ such that in $G$ for $i\ne j$, $s$ is mixed on $\{a_i,b_i\}$ and anticomplete to $\{a_j,b_j\}$, and $s'$ is mixed on $\{a_j,b_j\}$ and anticomplete to $\{a_i,b_i\}$. By symmetry, we may assume that $s$ is adjacent to $a_i$, and $s'$ is adjacent to $a_j$. However, then $b_i-a_i-s-x-s'-a_j-b_j$ is a $P_7$ in $G$, a contradiction. This proves \textit{(1)}.

\bigskip

\noindent Write $C=\{c_1,...,c_t\}$.
By \textit{(1)}, applying \ref{half} in time $O(|V(G)|^2)$ we obtain a partition $S_1\cup ...\cup S_q$ of $S$ and an ordering $\{d_1,...,d_t\}$ of the vertices of $C$. Renumber the vertices of $\hat{A}$ and $\hat{B}$ so that $d_k$ corresponds to the edge $a_kb_k$ for every $k\in \{1,...,t\}$. 

\bigskip

\noindent \textit{(2) For every $i\in \{1,...,q\}$ and $j\in \{1,...,t\}$ the following hold:} 

\bigskip

\noindent \textit{(2a) The vertices in $S_i$ are either all anticomplete to or all mixed on $\{a_j,b_j\}$.}

\bigskip

\noindent \textit{(2b) If the vertices in $S_i$ are all mixed on $\{a_j,b_j\}$, then every vertex in $S_i\cup ...\cup S_q$ is mixed on $\{a_k,b_k\}$ for all 
$k\in\{j,...,t\}$.}

\bigskip

\noindent Proof: By \ref{half}.1, it follows that in $H$ every vertex $d_j$ is either complete or anticomplete to $S_i$. Hence, by the construction of $H$, in $G$ the vertices in $S_i$ are either all anticomplete to or all mixed on $\{a_j,b_j\}$. This proves \textit{(2a)}. By \ref{half}.2, it follows that in $H$ if $d_j$ is complete to $S_i$, then $\{d_j,...,d_t\}$ is complete to $S_i\cup ...\cup S_q$, and \textit{(2b)} follows. This proves \textit{(2)}.

\bigskip

For $j\in \{1,...,t\}$, we define the \textit{height} of the edge $a_jb_j$ to be the maximum $\ell$ such that both $a_j$ and $b_j$ have neighbors in $S_\ell\cup ...\cup S_q$. Since every vertex in $\hat{A}\cup \hat{B}$ has a neighbor in $S$, the height of an edge is well defined. If the height of the edge $\{a_j,b_j\}$ is $\ell<q$, then \textit{(2)} implies that one of the vertices in $\{a_j,b_j\}$ is anticomplete to $S_{\ell+1}\cup ...\cup S_q$, we call this the \textit{small} vertex in $\{a_j,b_j\}$ and denote it by $s_j$. We call the vertex of $\{a_j,b_j\}\setminus \{s_j\}$ the \textit{large} vertex in $\{a_j,b_j\}$ and denote it by $l_j$. Then $l_j$ is complete to $S_{\ell+1}\cup ...\cup S_q$. If the edge $a_jb_j$ has height $q$, then we arbitrary assign $\{l_j,s_j\}=\{a_j,b_j\}$. Next, let $N_j$ be the set of vertices in $S_\ell\cup...\cup S_q$ adjacent to $l_j$, and let $M_j$ be the set of vertices in $S_\ell\cup...\cup S_q$ adjacent to $s_j$.
Clearly, computing the height of $a_jb_j$, determining the small and large 
vertices, and  computing the $N_j$ and $M_j$  can be done in time $O(|V(G)|^2)$.
\bigskip

\noindent \textit{(3) For $j\in\{1,...,t\}$, suppose the edge $a_jb_j$ has height $\ell$. Then the following hold:} 

\bigskip

\noindent \textit{(3a) $N_j\cup M_j=S_\ell\cup...\cup S_q$, where $N_j,M_j$ are disjoint, both non-empty, and $M_j\subseteq S_\ell$.}

\bigskip

\noindent \textit{(3b) Let  $k\in\{1,...,t\}\setminus \{j\}$, and let $\{y,z\}=\{a_k,b_k\}$. If $y$ is anticomplete to $S_\ell\cup...\cup S_q$, then the height of $a_kb_k$ is strictly less than $\ell$, $y=s_k$, and both $N_j,M_j$ are proper subsets of $N_k$ .}

\bigskip

\noindent Proof: Since $G$ is triangle-free, it follows that $N_j,M_j$ are disjoint. By the definition of height, both $N_j,M_j$ are non-empty and, by \textit{(2a)}, it follows that every vertex in $S_\ell$ is mixed on $\{a_j,b_j\}$. Hence, by \textit{(2b)}, it follows that every vertex in $S_\ell\cup ...\cup S_q$ is mixed on $\{a_j,b_j\}$, and so $N_j\cup M_j=S_\ell\cup...\cup S_q$. Finally, by our choice of $s_j$, it follows that $M_j\subseteq S_\ell$. This proves \textit{(3a)}. Next, we prove \textit{(3b)}. Since $y$ is anticomplete to $S_\ell\cup...\cup S_q$, it follows, by the definition of height, that the height of $a_kb_k$ is strictly less than $\ell$, and that $y=s_k$. Hence, by \textit{(3a)}, it follows that $l_k$ is complete to $S_\ell\cup...\cup S_q$, and so both $N_j,M_j$ are proper subsets of $N_k$. This proves \textit{(3b)}.

\bigskip

We say that $(G,L,X)$ has a \textit{type I coloring} if there exists a coloring $c$ of $(G,L,X)$ such that $\{c(a_i),c(b_i)\}=\{2,3\}$ for some $i \in \{1,...,t\}$. We now prove the following:

\bigskip

\noindent \textit{(4) There exists a set $\mathcal{L}_1$ of $O(|V(G)|)$ of subpalettes of $L$ such that}

\bigskip

\noindent \textit{(4a) For each $L_1\in\mathcal{L}_1$, $L_1(v)=L(v)$ for every $v \in \{x\} \cup S\cup Y$, and $|L_1(v)| \leq 2$ for every $v \in \hat{A}\cup \hat{B}$, and}

\bigskip
 
\noindent \textit{(4b) $(G,L,X)$ has a type I coloring if and only if $(G,L_1,X)$ is colorable for some $L_1 \in \mathcal{L}_1$; and for every $L_1 \in \mathcal{L}_1$, every coloring of $(G,L_1,X)$ is a type I
coloring of $(G,L,X)$.}

\bigskip

\noindent \textit{Moreover, $\mathcal{L}_1$ can be constructed in time $O(|V(G)|^3)$.}

\bigskip

\noindent Proof: Let $i\in\{1,...,t\}$ and $\ell$ be the height of the edge $a_ib_i$. First, set 

\begin{itemize}

\item $L_i(l_i)=\{2\}$,

\item $L_i(s_i)=\{3\}$, and 

\item $L_i(v)=L(v)$ for all $v \in V(G) \setminus (\hat{A}\cup \hat{B}$).

\end{itemize}

\noindent Next, for each $j \in \{1,...,t\} \setminus \{i\}$, let $y \in \{a_j,b_j\}$. If $N(y)\cap (S_\ell\cup ...\cup S_q)\ne \emptyset$, then set 
$$L_i(y) = \begin{cases}
 L(y)\setminus \{3\} & \text{, \hspace{2ex}if $N(y)\cap (S_\ell\cup ...\cup S_q)\subseteq N_i$} \\
 L(y)\setminus \{2\} & \text{, \hspace{2ex}if $N(y)\cap (S_\ell\cup ...\cup S_q)\subseteq M_i$} \\
 \{1\} & \text{, \hspace{2ex}otherwise}\\
\end{cases}$$
Otherwise, if $y$ is anticomplete to $S_\ell\cup ...\cup S_q$, then set $L_i(y)=\{2,3\}$. As above, construct the subpalette $L_i'$ of $L$, but with the roles of the colors $2$ and 3 exchanged.

Clearly, if one of $(G, L_i,X)$ and $(G,L_i',X)$ is colorable, then there exists a type I coloring of $(G,L,X)$. Now, suppose $c$ is a type I coloring of $(G,L,X)$ with $\{c(a_i),c(b_i)\}=\{2,3\}$. By symmetry, we may assume $c(l_i)=2$ and $c(s_i)=3$. We claim that $c(v) \in L_i(v)$ for all $v\in V(G)$. By definition, this is the case for $a_i,b_i$, and for all $v \in V(G) \setminus (\hat{A}\cup \hat{B})$. Since $L_i(v)=L(v) \subseteq \{2,3\}$ for every $v\in S$, as $c(l_i)=2$, it follows that every vertex in $N_i$ is colored 3. Similarly, as $c(s_i)=3$, it follows that every vertex in $M_i$ is colored 2. Hence, by \textit{(3a)}, the colors of all the vertices in $S_\ell\cup...\cup S_q$ are forced by $c(l_i)$ and $c(s_i)$. Next, consider $j \in \{1,...,t\} \setminus \{i\}$, and let $\{y,z\}= \{a_j,b_j\}$. If $y$ has a neighbor $y' \in S_\ell \cup ...\cup S_q$, then, as $c(y) \neq c(y')$, it follows that $c(y) \in L_i(y)$ by construction. So we may assume that $y$ is anticomplete to $S_\ell \cup ...\cup S_q$, and so, by construction, $L_i(y)=\{2,3\}$. By \textit{(3b)}, it follows that the height of $a_jb_j$ is strictly less than $\ell$, $y=s_j$, and that both $N_i,M_i$ are proper subsets of $N_j$. Hence, by construction, $L_i(l_j)=\{1\}$. By \textit{(3a)}, both $N_i,M_i$ are non-empty, and so $c(l_j)=1$, which implies $c(s_j) \in \{2,3\}$, and the claim holds. 

For every $i\in \{1,...,t\}$, construct the subpalettes $L_i,L_i'$ of $L$ as above. Then $\mathcal{L}_1=\{L_1,...,L_t,L_1',...,L_t'\}$ satisfies \textit{(4a)} and \textit{(4b)}. For a fixed $i\in \{1,...,t\}$, the subpalettes $L_i,L_i'$ of $L$ can be constructed in time $O(|V(G)|^2)$, and so $\mathcal{L}_1$ can be constructed in time $O(|V(G)|^3)$. This proves \textit{(4)}.

\bigskip

We say that $(G,L,X)$ has a \textit{type II coloring} if there exist distinct $i,j \in \{1,...,t\}$, $z \in \{a_j,b_j\}$  and a coloring $c$ of $(G,L,X)$ such that 

\bigskip

\noindent \textit{(II.1) $a_ib_i$ and $a_jb_j$ have the same height $\ell$, and}

\bigskip
 
\noindent \textit{(II.2) $N_i$ is not a proper subset of $N(z)\cap (S_{\ell}\cup ...\cup S_q)$, and}

\bigskip

\noindent \textit{(II.3) writing $\{y\}=\{a_j,b_j\} \setminus \{z\}$, we have 
$c(s_i)=c(z)=1$, and $\{c(l_i),c(y)\}=\{2,3\}$.}

\bigskip

\noindent We now prove the following:

\bigskip

\noindent \textit{(5) There exists a set $\mathcal{L}_2$ of $O(|V(G)|^2)$ subpalettes of $L$ such that}

\bigskip

\noindent \textit{(5a) For each $L_2\in\mathcal{L}_2$, $L_2(v)=L(v)$ for every $v \in \{x\} \cup S\cup Y$, and $|L_2(v)| \leq 2$ for every $v \in \hat{A}\cup \hat{B}$, and}
 
\bigskip
 
\noindent \textit{(5b) $(G,L,X)$ has a type II coloring if and only if $(G,L_2,X)$ is colorable for some $L_2 \in \mathcal{L}_2$; and for every $L_2 \in \mathcal{L}_2$, every coloring of  
$(G,L_2,X)$ is a type II coloring of $(G,L,X)$.}

\bigskip

\noindent \textit{Moreover, $\mathcal{L}_2$ can be constructed in time $O(|V(G)|^4)$.}

\bigskip

\noindent Proof: Let $i,j\in\{1,...,t\}$ be distinct, $\{y,z\}=\{a_j,b_j\}$, and suppose \textit{(II.1)} and \textit{(II.2)} are satisfied. First, set 

\begin{itemize}
\item $L_{\{i,j\}}(l_i)=\{2\}$,
\item $L_{\{i,j\}}(y)=\{3\}$,
\item $L_{\{i,j\}}(s_i)=L_{\{i,j\}}(z)=\{1\}$,
 and
\item $L_{\{i,j\}}(v)=L(v)$ for all $v \in V(G) \setminus (\hat{A}\cup \hat{B}$).

\end{itemize}

\noindent Next, for every $k \in \{1,...,t\} \setminus \{i,j\}$, consider each $w\in \{a_k,b_k\}$. If $N(w)\cap (S_\ell\cup ...\cup S_q)\ne \emptyset$, then set 

$$L_{\{i,j\}}(w) = \begin{cases}
 L(w)\setminus \{3\} & \text{, \hspace{2ex}if $N(w)\cap (S_\ell\cup ...\cup S_q)\subseteq N_i$} \\
 L(w)\setminus \{2\} & \text{, \hspace{2ex}if $N(w)\cap (S_\ell\cup ...\cup S_q)\subseteq M_i$} \\
 \{1\} & \text{, \hspace{2ex}otherwise}\\
\end{cases}$$

\noindent Otherwise, if $w$ is anticomplete to $S_\ell\cup ...\cup S_q$, then set $L_{\{i,j\}}(w)=\{2,3\}$. As above, construct the subpalette $L_{\{i,j\}}'$ of $L$, but with the roles of the colors $2$ and 3 exchanged. 

Clearly, if one of $(G,L_{\{i,j\}},X)$ and $(G,L_{\{i,j\}}',X)$ is colorable, then there exists a type II coloring of $(G,L,X)$. Now, suppose $c$ is a type II coloring of $(G,L,X)$ with $c(s_i)=c(z)=1$, and $\{c(l_i),c(y)\}=\{2,3\}$. By symmetry, we may assume $c(l_i)=2$ and $c(y)=3$. We claim that $c(v) \in L_{\{i,j\}}(v)$ for all $v\in V(G)$. By definition, this is the case for $a_i,b_i,a_j,b_j$, and for all $v \in V(G) \setminus (\hat{A}\cup \hat{B})$. Since $L_{\{i,j\}}(v)=L(v) \subseteq \{2,3\}$ for every $v\in S$, as $c(l_i)=2$, it follows, that every vertex in $N_i$ is colored 3. Let $M'=N(y)\cap (S_{\ell}\cup ...\cup S_q)$ and $N'=N(z)\cap (S_{\ell}\cup ...\cup S_q)$. Similarly, as $c(y)=3$, it follows that every vertex in $M'$ is colored 2. Hence, $N_i\cap M'=\emptyset$. By \textit{(3a)}, $N_i\cup M_i=N'\cup M'$.
Since $N_i$ is not a proper subset of $N'$, it follows that $N_i=N'$ and
$M_i=M'$. And so, it follows that the colors of all the vertices in $S_\ell\cup...\cup S_q$ are forced; namely $c(v)=3$ for every $v \in N_i$, and $c(v)=2$ for every $v \in M_i$. Next, consider $k \in \{1,...,t\} \setminus \{i,j\}$, and let  $u \in \{a_k,b_k\}$. If $u$ has a neighbor $u' \in S_\ell \cup ...\cup S_q$, then, as $c(u) \neq c(u')$, it follows that $c(u) \in L_{\{i,j\}}(u)$ by construction. So we may assume that $u$ is anticomplete to $S_\ell \cup ...\cup S_q$, and so, by construction, $L_{\{i,j\}}(u)=\{2,3\}$. By \textit{(3b)}, it follows that the height of $a_kb_k$ is strictly less than $\ell$, $u=s_k$, and that both $N_i,M_i$ are proper subsets of $N_k$. Hence, by construction, $L_i(l_k)=\{1\}$. By \textit{(3a)}, both $N_i,M_i$ are non-empty, and so $c(l_k)=1$, which implies $c(s_k) \in \{2,3\}$, and the claim holds. 

For every distinct pair $i,j\in\{1,...,t\}$ satisfying \textit{(II.1)} and \textit{(II.2)}, construct the subpalettes $L_{\{i,j\}},L_{\{i,j\}}'$ of $L$ as above. Let $\mathcal{L}_2$ be the set of all the subpalettes thus constructed, and observe that $|\mathcal{L}_2| \leq 4n^2$. Then $\mathcal{L}_2$ satisfies \textit{(5a)} and \textit{(5b)}. For distinct $i,j\in \{1,...,t\}$, the corresponding subpalettes can be constructed in time $O(|V(G)|^2)$, and so $\mathcal{L}_2$ can be constructed in time $O(|V(G)|^4)$. This proves \textit{(5)}.

\bigskip

We say that $(G,L,X)$ has a \textit{type III coloring} if for some $i\in \{1,...,t\}$ where $a_ib_i$ has height $\ell$, there exists a coloring $c$ of $(G,L,X)$ such that 

\bigskip
 
\noindent \textit{(III.1) $c(l_i)\in\{2,3\}$ and $c(s_i)=1$,}

\bigskip 
 
\noindent \textit{(III.2) let $j \in \{1,...,t\} \setminus \{i\}$ such that the height of $a_jb_j$ is at most $\ell$; write $\{y,z\}= \{a_j,b_j\}$. If $N_i$ is a proper subset of $N(z)\cap(S_\ell\cup...\cup S_q)$, then $c(z)=1$.}

\bigskip
 

\bigskip

We now prove the following:

\bigskip

\noindent \textit{(6) Suppose $(G,L,X)$ has no type I coloring, and no type 
II coloring. Then  there exists a set $\mathcal{L}_3$ of $O(|V(G)|)$ subpalettes of $L$ such that}

\bigskip

\noindent \textit{(6a) For each $L_3 \in\mathcal{L}_3$, $L_3(v)=L(v)$ for every $v \in \{x\}\cup S\cup Y$, and $|L_3(v)| \leq 2$ for every $v \in \hat{A}\cup \hat{B}$, and}
 
\bigskip
 
\noindent \textit{(6b) $(G,L,X)$ has a type III coloring if and only if $(G, L_3,X)$ is colorable for some $L_3 \in \mathcal{L}_3$; and for every $L_3 \in \mathcal{L}_3$, every coloring of $(G, L_3,X)$ is a type III coloring of $(G,L,X)$.}

\bigskip

\noindent \textit{Moreover, $\mathcal{L}_3$ can be constructed in time $O(|V(G)|^3)$.}

\bigskip

\noindent Proof: Let $i\in\{1,...,t\}$ and $\ell$ be the height of the edge $a_ib_i$. First, set 

\begin{itemize}

\item $L_i(l_i)=\{2\}$,

\item $L_i(s_i)=\{1\}$, and 

\item $L_i(v)=L(v)$ for all $v \in V(G) \setminus (\hat{A}\cup \hat{B}$).

\end{itemize}

\noindent Next, for every $j \in \{1,...,t\} \setminus \{i\}$, consider each $y \in \{a_j,b_j\}$. If $N(y)\cap N_i\ne \emptyset$, then set $L_i(y)=L(y) \setminus \{3\}$. Otherwise, if $y$ is anticomplete to $N_i$, then, taking $z\in \{a_j,b_j\}\setminus \{y\}$, set

$$L_i(y) = \begin{cases}
L(y)\setminus \{1\} & \text{, \hspace{2ex}if $N_i$ is a proper subset of $N(z)\cap(S_\ell\cup...\cup S_q)$} \\
L(y)\setminus \{3\} & \text{, \hspace{2ex}if $N_i$ is not a proper subset of $N(z)\cap(S_\ell\cup...\cup S_q)$} \\
\end{cases}$$

\noindent As above, construct the subpalette $L_i'$ of $L$, but with the roles of the colors $2$ and 3 exchanged.

First, we argue that every coloring of $(G,L_i,X)$ and $(G,L_i',X)$ is a 
type III coloring of $(G,L,X)$.
Suppose that one of $(G,L_i,X)$ and $(G,L_i',X)$ is colorable. By symmetry, we may assume that $c$ is a coloring of $(G,L_i,X)$, and so $c(l_i)=2$, $c(s_i)=1$, and \textit{(III.1)} holds. Now, we show \textit{(III.2)} holds. Suppose for $j\in \{1,...,t\}\setminus \{i\}$ the height $\ell'$ of $a_jb_j$ is at most $\ell$. Fix $\{y,z\}=\{a_j,b_j\}$. Since $c(l_i)=2$, every vertex in $N_i$ is colored 3. 
Suppose $N_i$ is a proper subset of $N(z) \cap (S_\ell \cup ... \cup S_q)$, and so $c(z)\ne 3$. By \textit{(3a)}, it follows that $y$ is anticomplete to $N_i$, and so, by construction, $c(y)\ne 1$. Thus, since $(G,L,X)$ does not have a type I coloring, it follows that $c(z)=1$, and so \textit{(III.2)} holds. Hence, $c$ is a type III coloring of $(G,L,X)$.

Next, we argue that if $(G,L,X)$ has a type III coloring, then 
$(G,L_3,X)$ is colorable for some $L_3 \in \mathcal{L}_3$. Suppose that $c$ is a type III coloring of $(G,L,X)$ with $c(l_i)\in \{2,3\}$ and $c(s_i)=1$. By symmetry, we may assume $c(l_i)=2$. We claim that $c(v) \in L_i(v)$ for all $v\in V(G)$. By definition, this is the case for $a_i,b_i$, and for all $v \in V(G) \setminus (\hat{A}\cup \hat{B})$. Since $L_i(v)=L(v) \subseteq \{2,3\}$ for every $v\in S$, as $c(l_i)=2$, it follows that every vertex in $N_i$ is colored 3. Next, consider $j \in \{1,...,t\} \setminus \{i\}$, and let $\{y,z\} = \{a_j,b_j\}$. If $y$ has a neighbor in $N_i$, then $c(y)\ne3$, and it follows that $c(y) \in L_i(y)$ by construction. So we may assume that $y$ is anticomplete to $N_i$. By \textit{(3a)}, it follows that the height of $a_jb_j$ is at most $\ell$, and that $z$ is complete to $N_i$. Hence, $c(z)\in \{1,2\}$. If $N_i$ is a proper subset of $N(z)\cap (S_\ell\cup....\cup S_q)$, then $c(z)=1$, by \textit{(III.2)}, and so $c(y)\ne 1$. Thus, we may assume that $N_i$ is not a proper subset of $N(z)\cap (S_\ell\cup....\cup S_q)$. Since $z$ is complete to $N_i$, this implies that $N_i=N(z)\cap (S_\ell\cup....\cup S_q)$. Hence, by \textit{(3a)}, it follows that $M_i=N(y)\cap (S_\ell\cup....\cup S_q)$, and so $a_jb_j$ also has height $\ell$. Thus, since $(G,L,X)$ does not admit a type II coloring,
it follows that $c(y) \neq 3$, and the claim holds.

For every $i\in \{1,...,t\}$, construct the subpalettes $L_i,L_i'$ of $L$ as above. Then $\mathcal{L}_3=\{L_1,...,L_t,L_1',...,L_t'\}$ satisfies \textit{(6a)} and \textit{(6b)}. For a fixed $i\in \{1,...,t\}$, the subpalettes $L_i,L_i'$ of $L$ can be constructed in time $O(|V(G)|^2)$, and so $\mathcal{L}_3$ can be constructed in time $O(|V(G)|^3)$. This proves \textit{(6)}.

\bigskip

Finally, define the additional subpalette $\hat{L}$ of $L$ such that for $v\in V(G)$

$$\hat{L}(v) = \begin{cases}
L(v) & \text{, \hspace{2ex}if $v\in V(G)\setminus (\hat{A}\cup \hat{B})$} \\
 \{1\} & \text{, \hspace{2ex}if $v=l_i$ for some $i\in\{1,...,t\}$} \\
 \{2,3\} & \text{, \hspace{2ex}if $v=s_i$ for some $i\in\{1,...,t\}$} \\
\end{cases}$$

\bigskip

Define $\mathcal{L}=\mathcal{L}_1\cup \mathcal{L}_2\cup \mathcal{L}_3\cup \{\hat{L}\}$. Note, $\mathcal{L}$ has size $O(|V(G)|^2)$, as it is dominated by $\mathcal{L}_2$, and can be constructed in time $O(|V(G)|^4)$, by \textit{(4)},\textit{(5)} and \textit{(6)}. By \textit{(4a)}, \textit{(5a)}, and \textit{(6a)}, it follows that $\mathcal{L}$ satisfies \textit{(a)}. We now argue that $\mathcal{L}$ satisfies \textit{(b)}. Since $\mathcal{L}$ contains only subpalettes of $L$, it clearly follows that if $(G,L',X)$ is colorable for some $L' \in \mathcal{L}$, then $(G,L,X)$ is colorable; and for every $L' \in \mathcal{L}$, every
coloring of $(G,L',X)$ is a coloring of $(G,L,X)$. Now, suppose that $c$ is a coloring of $(G,L,X)$. If $c$ is a type I or II coloring of $(G,L,X)$, then, by \textit{(4b)} and \textit{(5b)}, it follows that $(G,L',X)$ is colorable for
some $L' \in \mathcal{L}_1\cup \mathcal{L}_2$. Hence, we may assume that $(G,L,X)$ admits no coloring of type I or II. If $c(l_i)=1$ for every $i\in\{1,...,t\}$, then $c(s_i)\in \{2,3\}$ for every $i\in \{1,...,t\}$, and it follows that $c$ is a coloring of $(G,\hat{L},X)$, so we may assume not.

We claim that $c$ is a type III coloring of $(G,L,X)$. Let $a_ib_i$ be an edge with minimum height $\ell_{min}$ such that $c(l_i)\in\{2,3\}$, and subject to that with $N_i$ maximal. By symmetry, we may assume $c(l_i)=2$. Since $c$ is not a type I coloring of $(G,L,X)$, it follows that $c(s_i)=1$, and thus \textit{(III.1)} is satisfied. By our choice of $\ell_{min}$ and $i$, for all $j\in \{1,...,t\}\setminus \{i\}$ if the height of $a_jb_j$ is at most $\ell_{min}$ and $N_i$ is a proper subset of $N(z) \cap (S_\ell \cup ... \cup S_q)$ for some $z \in \{a_j,b_j\}$, then $c(z)=1$, and so \textit{(III.2)} is satisfied. This proves that
 $c$ is a type III coloring of $(G,L,X)$.
Now, by \textit{(6b)}, it follows that $(G,L',X)$ is colorable for some 
$L \in \mathcal{L}_3$. This proves \ref{lemma1}.

\end{proof}

\section{Reducing the Palettes: Part II}

Let $G$ be a graph, $L$ a palette of $G$, and $X$ a set of subsets of $V(G)$. 
Let $\mathcal{P}$ be a collection of triples
$(G',L',X')$, where $G'$ is an induced subgraph of $G$,
$L'$ is a subpalette of  $L|_{V(G')}$, and $X'$ is a set of subsets
of $V(G')$. We say that $\mathcal{P}$ is 
{\em colorable} if at least one element of $\mathcal{P}$ is colorable. 
We say that
$\mathcal{P}$ is a {\em restriction} of $(G,L,X)$ if $\mathcal{P}$ being
colorable implies that $(G,L,X)$ is colorable, and we can extend a coloring of 
a colorable element of $\mathcal{P}$ to a coloring of $(G,L,X)$ in time 
$O(|V(G)|)$.

\bigskip

First, we make the following easy observation:

\begin{lemma}\label{nbrs}
Let $G$ be a graph, and let $v \in V(G)$.
Let $L$ be an order $k$ palette of a graph $G$, and let $X$ be a set of subsets
of $V(G) \setminus \{v\}$. If $|L(v)|=k$ and in every coloring of $(G\setminus v,L,X)$ at most $k-1$ colors are used to color $N(v)$, then $(G,L,X)$ is colorable if and only if $(G\setminus v,L|_{V(G)\setminus \{v\}},X)$ is colorable. Furthermore, we can extend a coloring of $(G\setminus v,L|_{V(G)\setminus \{v\}},X)$ to a coloring of $(G,L,X)$ in constant time.

\end{lemma}

\noindent Next, we prove the following:

\begin{lemma}\label{edge}
Let $G$ be a triangle-free graph, and let $u,v \in V(G)$ be adjacent.
Let $X$ be a set of subsets of $V(G) \setminus \{u,v\}$.
Let $L$ be an order $3$ palette of $G$. Assume that:

\begin{itemize}

\item $|L(v)|=3$, and in every coloring of $(G\setminus \{u,v\},L|_{V(G)\setminus \{u,v\}},X)$, at most two colors are used to color $N(v)\setminus \{u\}$, and

\item $|L(u)|=2$ with $L(y) \subseteq \{1,2,3\}\setminus L(u)$ for all $y\in N(u)\setminus \{v\}$.

\end{itemize}

\noindent Then $(G,L,X)$ is colorable if and only if $(G\setminus \{u,v\},L|_{V(G)\setminus \{u,v\}},X)$ is colorable. Furthermore, we can extend a coloring of $(G\setminus \{u,v\},L|_{V(G)\setminus \{u,v\}},X)$ to a coloring of $(G,L,X)$ in constant time.

\end{lemma}

\begin{proof}

Clearly, if $(G,L,X)$ is colorable, then $(G\setminus \{u,v\},L|_{V(G)\setminus \{u,v\}},X)$ is colorable. Now, suppose $c$ is a coloring of $(G\setminus \{u,v\},L|_{V(G)\setminus \{u,v\}},X)$. We may assume that only colors $1$ and $2$ are used on
$N(v) \setminus \{u\}$. Assigning $c(v)=3$ and $c(u)\in L(u)\setminus \{3\}$, we obtain a coloring of $(G,L,X)$. This proves \ref{edge}.

\end{proof}

\bigskip

\noindent Let $Z$ be a set of subsets of $V(G)$.
Given a coloring $c$ of $(G,L,Z)$ and a subset $X\subseteq V(G)$ define the subpalette $L^X_c$ of $L$ as follows: for every $v\in V(G)$, set $$L^X_c(v) = \begin{cases}
 c(v) & \text{, \hspace{2ex}if $v\in X$} \\ 
 L(v) & \text{, \hspace{2ex}otherwise} \\
\end{cases}.$$

\noindent For a subset $Y\subseteq V(G)\setminus X$, let $L^{X,Y}_c$ be the subpalette of $L^X_c$ obtained by updating the palettes of the vertices in $Y$ with respect to $X$.




\begin{lemma}\label{update}

If $c$ is a coloring of $(G,L,Z)$, then $c$ is a coloring of $(G,L^{X,Y}_c,Z)$.

\end{lemma}

\begin{proof}

Clearly, $c$ is a coloring of $(G,L^{X,Y}_c,Z)$, since, by the definition of updating, $c(v)\in L^{X,Y}_c(v)$ for all $v\in V(G)$. This proves \ref{update}.

\end{proof}

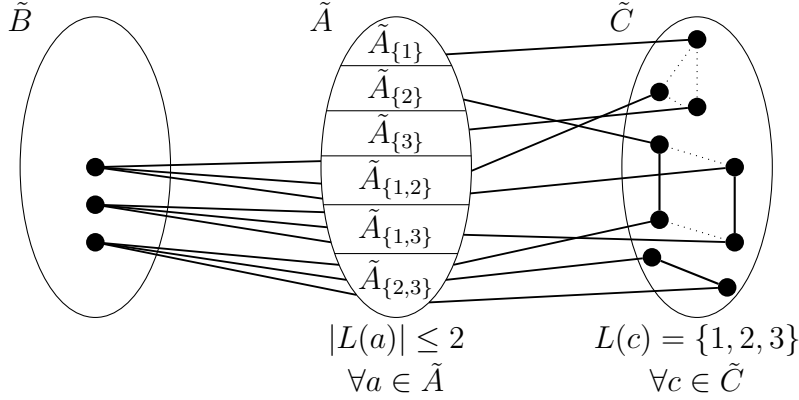
\begin{figure}[here]
\begin{center}

 \begin{tikzpicture}[scale=1]
 \GraphInit[vstyle=Simple]

 \tikzset{VertexStyle/.style = {shape = circle,fill = black,minimum size = 7pt,inner sep=0pt}}

\draw (0,0) ellipse (1cm and 2cm);
\node at (-1,2) {$\tilde{B}$};

 \Vertex[x=0,y=0]{1}
 \Vertex[x=0,y=-.5]{2}
 \Vertex[x=0,y=-1]{3}

\draw [black, line width=.025cm](0,0)--(3.0,.08);
\draw [black, line width=.025cm](0,0)--(3.012,-.22);
\draw [black, line width=.025cm](0,0)--(3.025,-.42);

\draw [black, line width=.025cm](0,-.5)--(3.05,-.6);
\draw [black, line width=.025cm](0,-.5)--(3.08,-.8);
\draw [black, line width=.025cm](0,-.5)--(3.135,-1);

\draw [black, line width=.025cm](0,-1)--(3.24,-1.3);
\draw [black, line width=.025cm](0,-1)--(3.3385,-1.5);
\draw [black, line width=.025cm](0,-1)--(3.472,-1.7);

\draw (4,0) ellipse (1cm and 2cm);
\node at (3,2) {$\tilde{A}$};

\draw(3+.26,1.35)--(5-.26,1.35);

\draw(3+.078,.75)--(5-.078,.75);

\draw(3,.15)--(5,.15);

\draw(3+.03,-.5)--(5-.03,-.5);

\draw(3+.18,-1.15)--(5-.18,-1.15);

\node at (4,1.65) {$\tilde{A}_{\{1\}}$};
\node at (4,1.05) {$\tilde{A}_{\{2\}}$};
\node at (4,.45) {$\tilde{A}_{\{3\}}$};
\node at (4,-.17) {$\tilde{A}_{\{1,2\}}$};
\node at (4,-.82) {$\tilde{A}_{\{1,3\}}$};
\node at (4,-1.5) {$\tilde{A}_{\{2,3\}}$};

\draw (8,0) ellipse (1cm and 2cm);
\node at (7,2) {$\tilde{C}$};

 \Vertex[x=8,y=1.7]{4}
 \Vertex[x=8,y=.8]{5}
 
 \draw[dotted](8,1.7)--(8,.7);

 \Vertex[x=7.5,y=.3]{6}
 \Vertex[x=7.5,y=-.7]{7}
 \Vertex[x=8.5,y=-.2+.2]{8}
 \Vertex[x=8.5,y=-1.2+.2]{9}
 
 \draw[dotted](7.5,.3)--(8.5,-.2+.2);
 \draw[dotted](7.5,-.7)--(8.5,-1.2+.2);

 \Vertex[x=7.4,y=-1.2]{11}
 \Vertex[x=8.4,y=-1.6]{12}

\Edges(11,12);
 
 \Edges(6,7);
 \Edges(8,9);
 
 \draw[dotted](8,1.7)--(7.5,1)--(8,.7);

 \draw [black, line width=.025cm](11)--(4.6+.061,-1.5);
 \draw [black, line width=.025cm](12)--(4.4+.0355,-1.8);

 \Vertex[x=7.5,y=1]{10}

 \draw [black, line width=.025cm](10)--(4+1,-.05);

 \draw [black, line width=.025cm](8,1.7)--(4+.66,1.5);
 \draw [black, line width=.025cm](5)--(4+.97,.5);
 \draw [black, line width=.025cm](7.5,.3)--(4+.895,.9);

 \draw [black, line width=.025cm](8)--(4+1-.019,-.35);

 \draw [black, line width=.025cm](7.5,-.7)--(4.76,-1.3);

 \draw [black, line width=.025cm](9)--(4+.895,-.9);

 \node at (4,-2.3) {$|L(a)|\leq 2$};
\node at (4,-2.8){$\forall a\in \tilde{A}$};

 \node at (8,-2.3) {$L(c)=\{1,2,3\}$};
\node at (8,-2.8){$\forall c\in \tilde{C}$};

\end{tikzpicture}

\caption{By \ref{lemma2}, when we encounter the above situation we can reduce determining if $(G,L,Z)$ is colorable to determining if a restriction 
$\mathcal{P}$ of $(G,L,Z)$ is colorable. The elements of $\mathcal{P}$ are ``closer'' to being of the form required by \ref{checkSubsets} than $(G,L,Z)$ is.}
\end{center}
\end{figure}

A vertex $u$ in a graph $G$ is {\em dominated} if there is 
$v \in V(G) \setminus \{u\}$ such that $u$ is non-adjacent to $v$, and 
$N(u) \subseteq N(v)$. In this case we say that $u$ is {\em dominated by $v$}.
The following is the main result of this section:

\begin{lemma}\label{lemma2}
 
Let $L$ be an order 3 palette of a $\{P_7,C_3\}$-free graph $G$,
such that $V(G)=\tilde{A}\cup \tilde{B}\cup \tilde{C}$, where

\begin{itemize}

\item $\tilde{B}$ is anticomplete to $\tilde{C}$,
\item every component of $\tilde{C}$ has at most 2 vertices,
\item every vertex of $\tilde{C}$ has a neighbor in $\tilde{A}$,
\item every vertex of $\tilde{A}$ has a neighbor in $\tilde{C}$,
\item $|L(a)|\leq 2$ for every $a \in \tilde{A}$, 
\item $|L(c)|= 3$ for every $c \in \tilde{C}$, and
\item $G$ contains no dominated vertices.


\end{itemize}

\noindent For every non-empty subset $X \subseteq \{1,2,3\}$, let $\tilde{A}_X=\{a\in \tilde{A}$ with $L(a)=X\}$. For every $c\in \tilde{C}$ and distinct $i,j\in\{1,2,3\}$, let $N_{\{i,j\}}(c)=N(c) \cap \tilde{A}_{\{i,j\}}$, and $M_{\{i,j\}}(c)=\tilde{A}_{\{i,j\}} \setminus N(c)$. 

\bigskip

\noindent Assume also that: 
\begin{itemize}

\item For every $c_1,c_2 \in \tilde{C}$ and $\{i,j,k\}= \{1,2,3\}$, $N_{\{i,j\}}(c_1) \cap M_{\{i,j\}}(c_2)$ is complete to $M_{\{i,k\}}(c_1) \cap N_{\{i,k\}}(c_2)$.

\item For distinct $i,j \in \{1,2,3\}$, there exists a vertex in $\tilde{B}$ complete to $\tilde{A}_{\{i,j\}}$.


\end{itemize}

\bigskip

Let $Z$ be a set of subsets of $A$.
\noindent Then there exists a restriction  $\mathcal{P}$ of $(G,L,Z)$,
of size  $O(|V(G)|^7)$, such that

\bigskip

\noindent (a) $\tilde{A}\cup \tilde{B}\subseteq V(G')$ for every
$(G',L',X') \in \mathcal{P}$, and
 
\bigskip

\noindent (b) Every $(G',L',X') \in \mathcal{P}$  is such that $L'(v)=L(v)$ for every $v \in \tilde{A}\cup \tilde{B}$, and $|L'(v)| \leq 2$ for every $v \in V(G')\cap \tilde{C}$, and $|X'|$ has size $O(|Z|+|V(G)|)$, and

\bigskip
 
\noindent (c) If $(G,L,Z)$ is colorable, then  $\mathcal{P}$ is colorable. 

\bigskip

\noindent Moreover, if the partition $\tilde{A}\cup \tilde{B}\cup \tilde{C}$ of $V(G)$ is given, then $\mathcal{P}$ can be computed in time $O(|V(G)|^7)$. 
\end{lemma}

\begin{proof}
 
 Since $|L(a)|\leq 2$ for all $a\in \tilde{A}$, we obtain the partition $\tilde{A}_{\{1\}}\cup \tilde{A}_{\{2\}}\cup \tilde{A}_{\{3\}}\cup \tilde{A}_{\{1,2\}}\cup \tilde{A}_{\{1,3\}}\cup \tilde{A}_{\{2,3\}}$ of $\tilde{A}$. Let $i,j\in\{1,2,3\}$ be distinct. Since $G$ is triangle-free and there exists a vertex in $\tilde{B}$ complete to $\tilde{A}_{\{i,j\}}$, it follows that $\tilde{A}_{\{i,j\}}$ is stable. Further, there is no vertex $c\in \tilde{C}$ such that $N(c)\subseteq \tilde{A}_{\{i,j\}}$, as such a vertex would be dominated by any vertex in $\tilde{B}$ complete to $\tilde{A}_{\{i,j\}}$. Thus, it follows that:


\bigskip

\noindent \textit{(1) Every vertex $c\in \tilde{C}$ such that $N(c)\cap \tilde{A}\subseteq \tilde{A}_{\{i,j\}}$ for some distinct $i,j\in\{1,2,3\}$ belongs to a component of $\tilde{C}$ of size $2$.}

\bigskip

A coloring $c$ of $(G,L,Z)$ is a \textit{type I coloring} if for distinct $i,j \in \{1,2,3\}$ there exists $z\in \tilde{C}$ and $x,y\in N_{\{i,j\}}(z)$ with $c(x)=i$ and $c(y)=j$.

\bigskip

\noindent \textit{(2) There exists a restriction $\mathcal{P}_1$ of $(G,L,Z)$,
of size  $O(|V(G)|^5)$, such that:}

\bigskip

\noindent \textit{(2a)  $\tilde{A}\cup \tilde{B}\subseteq V(G')$ and $X'=Z$ 
for every $(G',L',X') \in \mathcal{P}_1$,}
 
\bigskip

\noindent \textit{(2b) Every $(G',L',X') \in \mathcal{P}_1$  is such that $L'(v)=L(v)$ for every $v \in \tilde{A}\cup \tilde{B}$, and $|L'(v)| \leq 2$ for every $v \in V(G')\cap \tilde{C}$, and}

\bigskip
 
\noindent \textit{(2c) If $(G,L,Z)$ has a type I coloring, then $\mathcal{P}_1$ is colorable.}

\bigskip

\noindent \textit{Moreover, $\mathcal{P}_1$ can be computed in time $O(|V(G)|^7)$.}

\bigskip

\noindent Proof: Let $z\in \tilde{C}$, $\{i,j,k\}=\{1,2,3\}$, and $x,y \in N_{\{i,j\}}(z)$. Let $U_{(x,y,z)}$ be the set of all vertices $u \in  M_{\{i,k\}}(z)\cup M_{\{j,k\}}(z)$ for which there exists $c \in \tilde{C}$ such that $c$ is adjacent to $u$ and anticomplete to $\{x,y\}$. Since $x,y \in N_{\{i,j\}}(z)$, by assumption, it follows that $\{x,y\}$ is complete to $U_{(x,y,z)}$. Set

\begin{itemize}
 
\item $L'(v)=\{i\}$, for all $v\in \{x\}\cup N_{\{i,k\}}(z)$,
\item $L'(v)=\{j\}$, for all $v\in \{y\}\cup N_{\{j,k\}}(z)$,
\item $L'(v)=\{k\}$ for $v\in \{z\}\cup U_{(x,y,z)}$, and
\item $L'(v)=L(v)$, for every $v\in V(G)\setminus(N_{\{i,k\}}(z)\cup N_{\{j,k\}}(z)\cup U_{(x,y,z)}\cup \{x,y,z\})$.

\end{itemize}

\noindent 
Let $A=\{x,y\} \cup (\tilde{A}\setminus \tilde{A}_{\{i,j\}})$. Next, update the palettes of the vertices in $\tilde{C}$ with respect to $A$. If $c\in \tilde{C}$ has a neighbor in $(M_{\{i,k\}}(z)\cup M_{\{j,k\}}(z))\setminus U_{(x,y,z)}$, then it follows, by the definition of $U_{(x,y,z)}$, that $c$ is not anticomplete to $\{x,y\}$. 
And so after updating, for every $c\in \tilde{C}$ if $N(c)\cap (\tilde{A}\setminus \tilde{A}_{\{i,j\}})$ is non-empty, then $|L'(c)|\leq 2$. Let $F$ be the vertex set of the 2-vertex components of $\tilde{C}$ such that both vertices of the component have a neighbor in $\tilde{A}_{\{i,j\}}$. Initialize $D^{ij}_{(x,y,z)}=\emptyset$. Consider a vertex $c_1\in \tilde{C}\setminus F$ with $N(c_1)\cap \tilde{A}\subseteq \tilde{A}_{\{i,j\}}$. Then $|L'(c_1)|=3$, and, by (1), $c_1$ is adjacent to some $c_2\in \tilde{C}$ which is anticomplete to $\tilde{A}_{\{i,j\}}$. 
Since every vertex of $\tilde{C}$ has a neighbor in $\tilde{A}$, it follows that $|L'(c_2)|\leq 2$. If $|L'(c_2)|=1$, set $L'(c_1)=L'(c_1)\setminus L'(c_2)$. Otherwise, $|L'(c_2)|=2$. Since $c_2$ is anticomplete to $\tilde{A}_{\{i,j\}}$, it
follows that  that $$N(c_2) \cap \tilde{A} \subseteq \tilde{A}_{\{1\}} \cup \tilde{A}_{\{2\}} \cup \tilde{A}_{\{3\}} \cup  N_{\{i,k\}} (z) \cup N_{\{j,k\}}(z) \cup U_{(x,y,z)}.$$
In particular $|L'(v)|=1$ for every $v \in N(c_2) \cap \tilde{A}$, and
so, by the definition of updating, $L'(v) \subseteq \{1,2,3\}\setminus L(c_2)$ for all $v\in N(c_2)\setminus \{c_1\}$. Therefore, by \ref{edge}, it follows that $(G\setminus D^{ij}_{(x,y,z)},L'|_{V(G)\setminus D^{ij}_{(x,y,z)}},Z)$ is colorable if and only if $(G\setminus (D^{ij}_{(x,y,z)}\cup \{c_1,c_2\}),L'|_{V(G)\setminus (D^{ij}_{(x,y,z)}\cup \{c_1,c_2\})},Z)$ is colorable. In this case, set $D^{ij}_{(x,y,z)}=D^{ij}_{(x,y,z)}\cup \{c_1,c_2\}$. Carry out this procedure for every $c_1\in \tilde{C}\setminus F$ with $N(c_1)\cap \tilde{A}\subseteq \tilde{A}_{\{i,j\}}$. Let $G^{ij}_{(x,y,z)}=G\setminus D^{ij}_{(x,y,z)}$. Repeatedly applying the previous argument, it follows that $(G,L',Z)$ is colorable if and only if $(G^{ij}_{(x,y,z)},L'|_{V(G^{ij}_{(x,y,z)})},Z)$ is colorable. 
By assumption, there exists $b\in \tilde{B}$ complete to $\tilde{A}_{\{i,j\}}$. Let $\mathcal{L}'$ be the set of $O(|V(G)|^2)$ subpalettes of $L'|_{V(G^{ij}_{(x,y,z)})}$ obtained from \ref{lemma1} applied with

\begin{itemize}
\item $x=b$,
\item $S=\tilde{A}_{\{i,j\}}$, 
\item $\hat{A}\cup \hat{B}=F$, 
\item $Y=V(G^{ij}_{(x,y,z)})\setminus (\{b\}\cup \tilde{A}_{\{i,j\}}\cup F)$, and
\item $X=Z$.

\end{itemize}

\noindent For each $\hat{L}\in \mathcal{L}'$, define the subpalette $L^{ij}_{(x,y,z)}(\hat{L})$ of $L|_{V(G^{ij}_{(x,y,z)})}$ as follows: For $v\in V(G^{ij}_{(x,y,z)})$ set
$$L^{ij}_{(x,y,z)}(\hat{L})(v) = \begin{cases}
 \hat{L}(v) & \text{, \hspace{2ex}if $v\in \tilde{C}\setminus D^{ij}_{(x,y,z)}$} \\ 
 L'(v) & \text{, \hspace{2ex}otherwise} \\
\end{cases}$$

\noindent Let 
$$\mathcal{P}^{ij}_{(x,y,z)}=\{(G^{ij}_{(x,y,z)},L^{ij}_{(x,y,z)}(\hat{L}),Z) :\hat{L}\in \mathcal{L}'\} .$$ 

Let $\mathcal{P}_1$ be the union of all $\mathcal{P}^{ij}_{(x,y,z)}$. Since there are at most $3|V(G)|^3$ choices of $x,y,z$ and $\{i,j\}$, and each set $F$ can be found in time $O(|V(G)|^2)$, it follows that building $\mathcal{P}_1$ requires $O(|V(G)|^3)$ applications of \ref{lemma1}, and so $\mathcal{P}_1$ can be constructed in time $O(|V(G)|^7)$. Now, we argue that $\mathcal{P}_1$ is indeed a restriction $(G,L,Z)$. Suppose $\mathcal{P}^{ij}_{(x,y,z)}$ is colorable. By \ref{lemma1}, it follows that $(G^{ij}_{(x,y,z)},L'|_{V(G^{ij}_{(x,y,z)})},Z)$ is colorable, and so, as observed above, by construction and \ref{edge}, it follows that $(G,L',Z)$ is colorable. Since $L'$ is a subpalette of $L$, we deduce that $(G,L,Z)$ is colorable.
This proves that  $\mathcal{P}_1$ is indeed a restriction $(G,L,Z)$.

By construction and \ref{lemma1}, \textit{(2a)} and \textit{(2b)} hold. Next, we show that \textit{(2c)} holds. We need to prove that if $(G,L,Z)$ admits a
type I coloring, then $\mathcal{P}_1$ is colorable.
 Suppose $c$ is a type I coloring of $(G,L,Z)$. Let $z \in \tilde{C}$ with $c(z)=k$ and $x,y \in N_{\{i,j\}}(z)$ with $c(x)=i$ and $c(y)=j$. We claim that the restriction $\mathcal{P}^{ij}_{(x,y,z)}$ is colorable. Since $z$ is complete to $N_{\{i,k\}}(z)\cup N_{\{j,k\}}(z)$, it follows that $c(v)=i$ for every $v \in N_{\{i,k\}}(z)$ and $c(v) = j$ for every $v \in N_{\{j,k\}}(z)$. By assumption, for every $c_1,c_2 \in \tilde{C}$ we have that $N_{\{i,j\}}(c_1) \cap M_{\{i,j\}}(c_2)$ is complete to $M_{\{i,k\}}(c_1) \cap N_{\{i,k\}}(c_2)$. Taking $c_1=z$, it follows that $\{x,y\}$ is complete to every $v\in M_{\{i,k\}}(z)\cup M_{\{j,k\}}(z)$ for which there exists $c \in \tilde{C}$ such that $c$ is adjacent to $v$ and anticomplete to $\{x,y\}$, that is, $\{x,y\}$ is complete to $U_{(x,y,z)}$. Consequently, $c(v)=k$ for all $v\in U_{(x,y,z)}$. 
Also $c(v)=k$ for every $v\in \tilde{B}$ that is complete to $\tilde{A}_{\{i,j\}}$.
By \ref{update}, it follows that $c$ is a coloring of $(G,L^{A,\tilde{C}}_c,Z)$. 
Let $F$ be the vertex set of the 2-vertex components of $\tilde{C}$ such that both vertices of the component have a neighbor in $\tilde{A}_{\{i,j\}}$.
Let $C'$ be the set of vertices $v\in \tilde{C} \setminus F$ with $N(v) \cap\tilde{A} \subseteq \tilde{A}_{\{i,j\}}$. By construction, $c(v)\in L''(v)$ for all $L''\in \mathcal{L}^{ij}_{(x,y,z)}$ and $v\in V(G^{ij}_{(x,y,z)})\setminus (C' \cup F)$. Next, we show that the claim holds for every vertex in $C'\setminus D^{ij}_{(x,y,z)}$. By (1) and construction, every vertex $c_1\in C'\setminus D^{ij}_{(x,y,z)}$ is adjacent to some $c_2\in \tilde{C}\setminus D^{ij}_{(x,y,z)}$, and  $|L^{A,\tilde{C}}_c(c_2)|=1$. Since $c(c_1)\ne c(c_2)$, it follows that $c(c_1)\in L''(c_1)$ for all $L''\in\mathcal{L}^{ij}_{(x,y,z)}$. Now $\mathcal{P}^{ij}_{(x,y,z)}$ is colorable by \ref{lemma1} and \textit{(2c)} follows. This proves \textit{(2)}.

\bigskip

Let $\mathcal{X}=\{N(v) \cap \tilde{A}_{\{i,j\}} : v \in \tilde{C}, 1 \leq i <j \leq 3 \}$, and let $\mathcal{Z}=Z \cup \mathcal{X}$.
 
\noindent From \textit{(2)} it follows that:

\bigskip

\noindent \textit{(3) If $(G,L,Z)$ does not have a type I coloring, then 
$(G,L,Z)$ is colorable if and only if $(G,L,\mathcal{Z})$ is colorable.}

\bigskip

\noindent \textit{(4) Assume that $(G,L,Z)$ does not have a type I coloring. If there exists a vertex in $\tilde{C}$ with neighbors in all three of $\tilde{A}_{\{1,2\}}$, $\tilde{A}_{\{1,3\}}$ and $\tilde{A}_{\{2,3\}}$, then there exists a restriction $\mathcal{P}_2$ of size  $O(|V(G)|^2)$  of $(G,L,Z)$ such that:}

\bigskip

\noindent \textit{(4a)  $\tilde{A}\cup \tilde{B}\subseteq V(G')$ 
and $X'=\mathcal{Z}$ for every $(G',L',X') \in \mathcal{P}_2$,}
 
\bigskip

\noindent \textit{(4b) Every $(G',L',X') \in \mathcal{P}_2$  is such that $L'(v)=L(v)$ for every $v \in \tilde{A}\cup \tilde{B}$, and $|L'(v)| \leq 2$ for every $v \in V(G')\cap \tilde{C}$, and}

\bigskip

\noindent \textit{(4c) If $(G,L,Z)$ is colorable, then $\mathcal{P}_2$ is colorable.}

\bigskip

\noindent \textit{Moreover, $\mathcal{P}_2$ can be computed in time $O(|V(G)|^4)$.}

\bigskip 

\noindent \textit{If there does not exist a vertex in $\tilde{C}$ with neighbors in all three of $\tilde{A}_{\{1,2\}}$, $\tilde{A}_{\{1,3\}}$ and $\tilde{A}_{\{2,3\}}$, then $\mathcal{P}_2=\emptyset$.}

\bigskip

\noindent Proof:
Let $w \in \tilde{C}$. For $\{i,j,k\}= \{1,2,3\}$, let $P_{\{j,k\}}(w)$ be the set of vertices $v \in M_{\{j,k\}}(w)$ for which there exists $c \in \tilde{C}$ adjacent to $v$ and anticomplete to $N_{\{i,j\}}(w) \cup N_{\{i,k\}}(w)$. By assumption, it follows that $P_{\{j,k\}}(w)$ is complete to $N_{\{i,j\}}(w) \cup N_{\{i,k\}}(w)$. Thus, if $v\in M_{\{j,k\}}(w)\setminus P_{\{j,k\}}(w)$, then 
every vertex in $N(v)\cap \tilde{C}$ has a neighbor in $N_{\{i,j\}}(w) \cup N_{\{i,k\}}(w)$.

If there does not exist a vertex in $\tilde{C}$ with neighbors in all three of $\tilde{A}_{\{1,2\}}$, $\tilde{A}_{\{1,3\}}$ and $\tilde{A}_{\{2,3\}}$, set $\mathcal{P}_2=\emptyset$, and halt.
Otherwise, suppose $w \in \tilde{C}$ has neighbors in all three of $\tilde{A}_{\{1,2\}}$, $\tilde{A}_{\{1,3\}}$ and $\tilde{A}_{\{2,3\}}$. Let $\{i,j,k\}=\{1,2,3\}$
be such that $P_{\{j,k\}}(w)=\emptyset$ (we will show later that if $(G,L,Z)$ is
colorable, then such  $P_{\{j,k\}}(w)$ exists); if no such $\{i,j,k\}$ exists, 
set  $\mathcal{P}_2=\emptyset$, and halt. Set

\begin{itemize}

\item $L'(v)=\{i\}$, for all $v\in \{w\}\cup P_{\{i,j\}}(w)$,
\item $L'(v)=\{j\}$, for all $v\in N_{\{i,j\}}(w)\cup N_{\{j,k\}}(w)$,
\item $L'(v)=\{k\}$, for all $v\in N_{\{i,k\}}(w)$, and
\item $L'(v)=L(v)$, for every $v\in V(G)\setminus(N_{\{i,j\}}(w)\cup N_{\{j,k\}}(w)\cup N_{\{i,k\}}(w)\cup P_{\{i,j\}}(w) \cup \{w\})$.

\end{itemize}

\noindent Let $A=\tilde{A}_{\{1\}} \cup \tilde{A}_{\{2\}} \cup \tilde{A}_{\{3\}} \cup N_{\{i,j\}}(w)\cup N_{\{j,k\}}(w)\cup N_{\{i,k\}}(w)\cup P_{\{i,j\}}(w)$. Next, update the palettes of all the vertices in $\tilde{C}$ with respect to $A$. If $c\in \tilde{C}$ has a neighbor in $(M_{\{i,j\}}(w)\cup M_{\{j,k\}}(w))\setminus P_{\{i,j\}}$, then $c$ is not anticomplete to $N_{\{i,j\}}(w)\cup N_{\{j,k\}}(w)\cup N_{\{i,k\}}(w)$,
and so after updating, for every $c\in \tilde{C}$, if $N(c)\cap (\tilde{A}\setminus M_{\{i,k\}}(w))\neq \emptyset$, then $|L'(v)| \leq 2$. 
Let $F$ to be the vertex set of the 2-vertex components of $\tilde{C}$ such that both vertices of the component have a neighbor in $M_{\{i,k\}}(w)$. Initialize $D_{(i,j,k)}=\emptyset$. Consider a vertex $c_1\in \tilde{C}\setminus F$ with $N(c_1)\cap \tilde{A}\subseteq M_{\{i,k\}}(w)$. Then $|L'(c_1)|=3$, and, by (1), $c_1$ is adjacent to some $c_2\in \tilde{C}$. 
Since in every coloring of $(G \setminus c_1,L', \mathcal{Z})$ at most two colors appear in $N(c_1)\subseteq \{c_2\}\cup M_{\{i,k\}}(w)$, by \ref{nbrs}, it follows that $(G\setminus D_{(i,j,k)},L'|_{V(G)\setminus D_{(i,j,k)}}, \mathcal{Z})$ is colorable if and only if $(G\setminus (D_{(i,j,k)}\cup \{c_1\}),L'|_{V(G)\setminus (D_{(i,j,k)}\cup \{c_1\})}, \mathcal{Z})$ is colorable. In this case, set $D_{(i,j,k)}=D_{(i,j,k)}\cup \{c_1\}$. Carry out this procedure for every $c_1\in \tilde{C}\setminus F$ with $N(c_1)\cap \tilde{A}\subseteq M_{\{i,k\}}(w)$. Let $G_{(i,j,k)}=G\setminus D_{(i,j,k)}$. Repeatedly applying the previous argument, it follows that $(G,L',\mathcal{Z})$ is colorable if and only if $(G_{(i,j,k)},L'|_{V(G_{(i,j,k)})}, \mathcal{Z})$ is colorable. 
By assumption, there exists $b\in \tilde{B}$ complete to $\tilde{A}_{\{i,k\}}$; therefore $b$ is complete to $M_{\{i,k\}}(w)$.  Let $\mathcal{L}'$ be the set of $O(|V(G)|^2)$ subpalettes of $L'|_{V(G_{(i,j,k)})}$ obtained from \ref{lemma1} applied with

\begin{itemize}

\item $x=b$, 
\item $S=M_{\{i,k\}}(w)$, 
\item $\hat{A}\cup \hat{B}=F$, 
\item $Y=V(G_{(i,j,k)})\setminus (\{b\}\cup M_{\{i,k\}}(w)\cup F)$, and
\item $X=\mathcal{Z}$.

\end{itemize}

\noindent For each $\hat{L}\in\mathcal{L}'$, define the subpalette $L_{(i,j,k)}(\hat{L})$ of $L|_{V(G_{(i,j,k)})}$ as follows: For $v\in V(G_{(i,j,k)})$ set

$$L_{(i,j,k)}(\hat{L})(v) = \begin{cases}
 \hat{L}(v) & \text{, \hspace{2ex}if $v\in \tilde{C}\setminus D_{(i,j,k)}$} \\ 
L'(v) & \text{, \hspace{2ex}otherwise} \\
\end{cases}$$

\noindent Let 
$$\mathcal{P}_{(i,j,k)}=\{(G_{(i,j,k)},L_{(i,j,k)}(\hat{L}), \mathcal{Z}) : \hat{L}\in \mathcal{L}'\}.$$ 

Let $\mathcal{P}_2$ be union of all $\mathcal{P}_{(i,j,k)}$. Since $w$ can be found in time $O(|V(G)|^2)$, there are $6$ choices of $(i,j,k)$, and each set $F$ can be found in time $O(|V(G)|^2)$, it follows that building $\mathcal{P}_2$ requires $6$ applications of \ref{lemma1}, and so $\mathcal{P}_2$ can be constructed in time $O(|V(G)|^4)$. Now, we argue that $\mathcal{P}_2$ is indeed a restriction of $(G,L,Z)$. Suppose $\mathcal{P}_{(i,j,k)}$ is colorable. By \ref{lemma1}, it follows that $(G_{(i,j,k)},L'|_{V(G_{(i,j,k)})}, \mathcal{Z})$ is colorable, and so, as observed above, by construction and \ref{nbrs}, it follows that $(G,L',\mathcal{Z})$ is colorable. Since $L'$ is a subpalette of $L$, and $Z \subseteq \mathcal{Z}$,  we deduce that $(G,L,Z)$ is colorable.

By construction and \ref{lemma1}, \textit{(4a)} and \textit{(4b)} hold. Next, we show that \textit{(4c)} holds. Suppose $c$ is a coloring of $(G,L,Z)$ with $c(w)=i$ where $w \in \tilde{C}$ has neighbors in all three of $\tilde{A}_{\{1,2\}}$, $\tilde{A}_{\{1,3\}}$ and $\tilde{A}_{\{2,3\}}$. By \textit{(3)},
$c$ is a coloring of $(G,L, \mathcal{Z})$,  and by symmetry, we may assume that $j,k \in\{1,2,3\}\setminus \{i\}$ are such that $c(v)=j$ for all $v \in N_{\{j,k\}}(w)$. We claim that the restriction $\mathcal{P}_{(i,j,k)}$ is colorable. 
Clearly, every set in $\mathcal{Z}$ is monochromatic in $c$.
It follows that $c(v)=j$ for every $v \in N_{\{i,j\}}(w)$, and $c(v)=k$ for every $v \in N_{\{i,k\}}(w)$. Since $N_{\{j,k\}}(w)$ is complete to $P_{\{i,j\}}(w)$, it follows that $c(v)=i$ for every $v \in P_{\{i,j\}}(w)$. Since $N_{\{i,j\}}(w)\cup N_{\{i,k\}}(w)$ is complete to $P_{\{j,k\}}(w)$, it follows that $P_{\{j,k\}}(w)$ is empty.  Let $A=\tilde{A}_{\{1\}} \cup \tilde{A}_{\{2\}} \cup \tilde{A}_{\{3\}} \cup N_{\{i,j\}}(w)\cup N_{\{j,k\}}(w)\cup N_{\{i,k\}}(w)\cup P_{\{i,j\}}(w)$.

By \ref{update}, it follows that $c$ is a coloring of $(G,L^{A,\tilde{C}}_c)$. 
Let $F$ to be the vertex set of the 2-vertex components of $\tilde{C}$ such that both vertices of the component have a neighbor in $M_{\{i,k\}}(w)$.
Let $C'$ be the set of vertices $v\in \tilde{C} \setminus F$ with $N(v) \cap \tilde{A} \subseteq M_{\{i,k\}}(w)$. By construction, $c(v)\in L''(v)$ for all $L''\in \mathcal{L}_{(i,j,k)}$ and $v\in V(G_{(i,j,k)})\setminus (C' \cup F)$. Next, we show that the claim holds for every vertex in $C'\setminus D_{(i,j,k)}$. By (1) and construction, every vertex $c_1\in C'\setminus D_{(i,j,k)}$ is adjacent to some vertex $c_2\in \tilde{C}\setminus D_{(i,j,k)}$. By construction, both $c_1$ and $c_2$ have a neighbor in $M_{\{i,k\}}(w)$, hence belong to $F$, and so \textit{(4c)} follows from \ref{lemma1}. This proves \textit{(4)}.

\bigskip

A coloring $c$ of $(G,L,Z)$ is a \textit{type II coloring} if there exists $w \in \tilde{C}$ with neighbors $x,y \in \tilde{A}_{\{1,2\}}\cup \tilde{A}_{\{1,3\}}\cup \tilde{A}_{\{2,3\}}$ such that $c(x) \neq c(y)$.

\bigskip

\noindent \textit{(5) Assume that $(G,L,Z)$ does not have a type I coloring, and that no vertex of $\tilde{C}$ has neighbors in all three of $\tilde{A}_{\{1,2\}}$, $\tilde{A}_{\{1,3\}}$ and $\tilde{A}_{\{2,3\}}$. If $G$ has a type II coloring, then there exists a restriction  $\mathcal{P}_3$ of $(G,L,Z)$ of size $O(|V(G)|^7)$ such that:}

\bigskip

\noindent \textit{(5a)  $\tilde{A}\cup \tilde{B}\subseteq V(G')$ and 
$X'=\mathcal{Z}$ for every $(G',L',X') \in \mathcal{P}_3$,}
 
\bigskip

\noindent \textit{(5b) Every $(G',L',X') \in \mathcal{P}_3$  is such that $L'(v)=L(v)$ for every $v \in \tilde{A}\cup \tilde{B}$, and $|L'(v)| \leq 2$ for every $v \in V(G')\cap \tilde{C}$, and}

\bigskip

 \noindent \textit{(5c) $\mathcal{P}_3$ is colorable.}

\bigskip

\noindent \textit{Moreover, $\mathcal{P}_3$ can be computed in time $O(|V(G)|^7)$.}

\bigskip

\noindent \textit{If $G$ does not have a type II coloring, then 
$\mathcal{P}_3=\emptyset$.}

\bigskip

\noindent Proof: Let $\{i,j,k\}=\{1,2,3\}$. If there does not exist a vertex
 $w\in \tilde{C}$ anticomplete to $\tilde{A}_{\{k\}}$ with $N_{\{i,k\}}(w)\ne \emptyset$ and $N_{\{i,j\}}(w)\cup N_{\{j,k\}}(w)\ne\emptyset$, set $\mathcal{P}_3=\emptyset$ and halt. Otherwise, let $w\in \tilde{C}$ be anticomplete to $\tilde{A}_{\{k\}}$ with $N_{\{i,k\}}(w)\ne \emptyset$ and $N_{\{i,j\}}(w)\cup N_{\{j,k\}}(w)\ne\emptyset$. Since no vertex of $\tilde{C}$ has neighbors in all three of $\tilde{A}_{\{1,2\}}$, $\tilde{A}_{\{1,3\}}$ and $\tilde{A}_{\{2,3\}}$, it follows that either $N_{\{i,j\}}(w)=\emptyset$ or $N_{\{j,k\}}(w)=\emptyset$. Set

\begin{itemize}
\item $L'(w)=\{k\}$,
\item $L'(v)=\{i\}$, for every $v \in N_{\{i,k\}}(w)$, 
\item $L'(v)=\{j\}$, for every $v \in N_{\{i,j\}}(w)\cup N_{\{j,k\}}(w)$, and
\item $L'(v)=L(v)$, for every $v\in V(G)\setminus(N_{\{i,j\}}(w)\cup N_{\{i,k\}}(w)\cup N_{\{j,k\}}(w)\cup \{w\})$.
\end{itemize}

\noindent 
Let $A=\tilde{A}_{\{1\}} \cup \tilde{A}_{\{2\}} \cup \tilde{A}_{\{3\}} \cup N_{\{i,j\}}(w)\cup N_{\{i,k\}}(w)\cup N_{\{j,k\}}(w)\cup \{w\}$. First, update the palettes of the vertices in $\tilde{A}$ with respect to $A$. 
Then, update the palettes of all the vertices in $\tilde{C}$ with respect to $\tilde{A}$. And so after updating, for every $c\in \tilde{C}$ if $N(c)\cap A$ is non-empty, then $|L'(c)|\leq 2$. 
Let $v\in \tilde{C}$ with $|L'(v)|=3$. Then, by the definition of updating,  $N(v)\cap \tilde{A} \subseteq M_{\{i,k\}}(w)\cup M_{\{j,k\}}(w)$. Let $F_{\{i,k\}}$ be  the vertex set of the 2-vertex components of $\tilde{C}$ such that both vertices of the component have a neighbor in $\tilde{A}_{\{i,k\}}$, and let $F_{\{j,k\}}$ be  the vertex set of the 2-vertex components of $\tilde{C}$ such that both vertices of the component have a neighbor in $\tilde{A}_{\{j,k\}}$. Initialize $D^w_{(i,j,k)}=\emptyset$. Consider a vertex $c_1\in \tilde{C}\setminus (F_{\{i,k\}} \cup F_{\{j,k\}})$ with $N(c_1)\cap \tilde{A}\subseteq M_{\{i,k\}}(w)\cup M_{\{j,k\}}(w)$ and $|L'(c_1)|=3$. Recall  that 
in every coloring of $(G \setminus c_1,L', \mathcal{Z})$ at most two colors  appear in $N(c_1) \cap (M_{\{i,k\}}(w)\cup M_{\{j,k\}}(w))$. Therefore,
if $c_1$ is anticomplete to $\tilde{C}\setminus \{c_1\}$, then, by \ref{nbrs}, it follows that $(G\setminus D^w_{(i,j,k)},L'|_{V(G)\setminus D^w_{(i,j,k)}}, \mathcal{Z})$ is colorable if and only if $(G\setminus (D^w_{(i,j,k)}\cup \{c_1\}),L'|_{V(G)\setminus (D^w_{(i,j,k)}\cup \{c_1\})}, \mathcal{Z})$ is colorable. In this case, set $D^w_{(i,j,k)}=D^w_{(i,j,k)}\cup \{c_1\}$. So we may assume $c_1$ is adjacent to some $c_2\in \tilde{C}$. Suppose that  $c_1$ is
anticomplete to at least one of $M_{\{i,k\}}(w)$ and $M_{\{j,k\}}(w)$.
Then in every coloring of $(G \setminus c_1,L', \mathcal{Z})$ at most one color appears in $N(c_1) \cap (M_{\{i,k\}}(w)\cup M_{\{j,k\}}(w))$, and so,
since $N(c_1) \subseteq  \{c_2\} \cup M_{\{i,k\}}(w) \cup M_{\{j,k\}}(w)$,
we deduce that   at most two colors appear in $N(c_1)$. Thus, again by \ref{nbrs}, $(G\setminus D^w_{(i,j,k)},L'|_{V(G)\setminus D^w_{(i,j,k)}}, \mathcal{Z})$ is colorable if and only if $(G\setminus (D^w_{(i,j,k)}\cup \{c_1\}),L'|_{V(G)\setminus (D^w_{(i,j,k)}\cup \{c_1\})}, \mathcal{Z})$ is colorable. In this case, set $D^w_{(i,j,k)}=D^w_{(i,j,k)}\cup \{c_1\}$. Therefore we may assume that $c_1$ has both a neighbor in
$M_{\{i,k\}}(w)$ and a neighbor in  $M_{\{j,k\}}(w)$. Since 
$c_1 \not \in F_{\{i,k\}} \cup F_{\{j,k\}}$, it follows that $c_2$ is anticomplete
to $\tilde{A}_{\{i,k\}} \cup \tilde{A}_{\{j,k\}}$. This implies that every neighbor
of $c_2$ in $\tilde{A}_{\{i,j\}}$ either belongs to $N_{\{i,j\}}(w)$ or is complete to 
$N_{\{i,k\}}(w)$. This implies that $|L'(y)|=1$ for every 
$y \in N(c_2) \setminus \{c_1\}$, and so, by the definition of updating,   $L'(y) \subseteq \{1,2,3\}\setminus 
L(c_2)$ for all $y\in N(c_2)\setminus \{c_1\}$.
If $|L'(c_2)|=1$, set $L'(c_1)=L'(c_1)\setminus L'(c_2)$. Otherwise, $|L'(c_2)|=2$ and so, by \ref{edge}, it follows that $(G\setminus D^w_{(i,j,k)},L'|_{V(G)\setminus D^w_{(i,j,k)}},\mathcal{Z})$ is colorable if and only if $(G\setminus (D^w_{(i,j,k)}\cup \{c_1,c_2\}),L'|_{V(G)\setminus (D^w_{(i,j,k)}\cup \{c_1,c_2\})}, \mathcal{Z})$ is colorable. In this case, set $D^w_{(i,j,k)}=D^w_{(i,j,k)}\cup \{c_1,c_2\}$. 
Carry out this procedure for every $c_1\in \tilde{C}\setminus (F_{\{i,k\}}\cup F_{\{j,k\}})$ with $N(c_1)\cap \tilde{A}\subseteq M_{\{i,k\}}(w)\cup M_{\{j,k\}}(w)$ and $|L'(c_1)|=3$. Let $G^w_{(i,j,k)}=G\setminus D^w_{(i,j,k)}$. Repeatedly applying the previous argument, it follows that $(G,L',\mathcal{Z})$ is colorable if and only if $(G^w_{(i,j,k)},L'|_{V(G^w_{(i,j,k)})},\mathcal{Z})$ is colorable. 
By assumption, there exists $b\in \tilde{B}$ complete to $\tilde{A}_{\{i,k\}}$.
Let $\mathcal{L}'$ be the set of $O(|V(G)|^2)$ subpalettes of $L'|_{V(G^w_{(i,j,k)})}$ obtained from \ref{lemma1} applied with 

\begin{itemize}

\item $x=b$, 
\item $S=\tilde{A}_{\{i,k\}}$,
\item $\hat{A}\cup \hat{B}=F_{\{i,k\}}$, 
\item $Y=V(G^w_{\{i,j,k\}})\setminus (\{b\}\cup \tilde{A}_{\{i,k\}}(w)\cup F_{\{i,k\}})$, and
\item $X=\mathcal{Z}$.

\end{itemize}

\noindent For each $\hat{L}\in\mathcal{L}'$, define the subpalette $L^w_{(i,j,k)}(\hat{L})$ of $L|_{V(G_{(i,j,k)})}$ as follows: For $v\in V(G_{(i,j,k)})$ set

$$L^w_{(i,j,k)}(\hat{L})(v) = \begin{cases}
 \hat{L}(v) & \text{, \hspace{2ex}if $v\in \tilde{C}\setminus D^w_{(i,j,k)}$} \\
 L'(v) & \text{, \hspace{2ex}otherwise} \\
\end{cases}$$


$$\mathcal{P}^w_{(i,j,k)}=\{(G^w_{(i,j,k)},L^w_{(i,j,k)}(\hat{L}), \mathcal{Z}) : \hat{L}\in \mathcal{L}'\}.$$ 

By assumption, there exists $b'\in \tilde{B}$ complete to $\tilde{A}_{\{j,k\}}$. 
For each $\hat{L}\in\mathcal{L}'$, let $\mathcal{L}''(\hat{L})$ be the set of $O(|V(G)|^2)$ subpalettes of $\hat{L}$ obtained from \ref{lemma1} applied with

\begin{itemize}

\item $x=b'$, 
\item $S=A_{\{j,k\}}$,
\item $\hat{A}\cup \hat{B}=F_{\{j,k\}}$, 
\item $Y=V(G^w_{\{i,j,k\}})\setminus (\{b'\}\cup A_{\{j,k\}}(w)\cup F_{\{j,k\}})$, and
\item $X=\mathcal{Z}$.

\end{itemize}

\noindent Finally, for every $\hat{L}\in\mathcal{L}'$ and $L''\in\mathcal{L}''(\hat{L})$, define the subpalette $L^w_{(i,j,k)}(\hat{L},L'')$ of $L|_{V(G_{(i,j,k)})}$ as follows: For every $v\in V(G^w_{(i,j,k)})$ set

$$L^w_{\{i,j,k\}}(\hat{L},L'')(v) = \begin{cases}
L''(v) &\text{, \hspace{2ex}if $v\in \tilde{C}\setminus D^w_{(i,j,k)}$} \\ 
\hat{L}(v) & \text{, \hspace{2ex}otherwise} \\
\end{cases}$$

\noindent Let 

$$\mathcal{P}^w_{(i,j,k)}=\{(G^w_{(i,j,k)}, L^w_{(i,j,k)}(\hat{L},L''), \mathcal{Z}):\hat{L}\in \mathcal{L}',L''\in\mathcal{L}''(\hat{L})\}.$$

Let $\mathcal{P}_3$ be the union of all $\mathcal{P}^w_{(i,j,k)}$. Since there are at most $6n$ choices of $w$ and $\{i,j,k\}$, and $F_{\{i,k\}}$, $F_{\{j,k\}}$ can be found in time $O(|V(G)|^2)$, it follows that building each $\mathcal{P}^w_{(i,j,k)}$ requires $O(|V(G)|^2)$ applications of \ref{lemma1}, and so $\mathcal{P}_3$ can be constructed in time $O(|V(G)|^7)$. Now, we argue that $\mathcal{P}_3$ is indeed a restriction of $(G,L,Z)$. Suppose some $\mathcal{P}^w_{(i,j,k)}$ is colorable. By \ref{lemma1}, it follows that $(G^w_{(i,j,k)},L'|_{V(G^w_{(i,j,k)})}, \mathcal{Z})$ is colorable, and so, as argued above, by \ref{nbrs} and \ref{edge}, it follows that $(G,L',\mathcal{Z})$ is colorable. Since $L'$ is a subpalette of $L$, and $Z \subseteq \mathcal{Z}$, we deduce that $(G,L,Z)$ is colorable, and a coloring of $(G,L,Z)$ can be reconstructed in linear time.


Suppose $c$ is a type II coloring of $(G,L,Z)$. By~(3), $c$ is a coloring of 
$(G,L, \mathcal{Z})$.
By construction and \ref{lemma1}, \textit{(5a)} and \textit{(5b)} hold. Next, we show that \textit{(5c)} holds. Let $w \in \tilde{C}$ have neighbors of two different colors (under $c$)  in $\tilde{A}_{\{1,2\}}\cup \tilde{A}_{\{1,3\}}\cup \tilde{A}_{\{2,3\}}$. By symmetry, we may assume $c(w)=k$ with $N_{\{i,k\}}(w)\ne \emptyset$. Then $w$ is anticomplete to $\tilde{A}_{\{k\}}$. Since $G$ admits a type II coloring and no type I coloring, 
we deduce that $N_{\{i,j\}}(w) \cup N_{\{j,k\}}(w) \ne \emptyset$. We claim that 
$\mathcal{P}^w_{(i,j,k)}$ is colorable. It follows that $c(v)=i$ for every $v \in N_{\{i,k\}}(w)$ and that $c(u)=j$ for every $u \in N_{\{j,k\}}(w)$. 
We claim that  $c(v)=j$ for every $v \in N_{\{i,j\}}(w)$. Suppose not.
Then $N_{\{i,j\}}(w) \neq \emptyset$ and  $N_{\{j,k\}}(w) = \emptyset$ . Since $c$ is a type II coloring, it follows
that there exists $y \in N_{\{i,j\}}(w)$ with $c(y)=j$. But since $(G,L,Z)$ has
no type I coloring, it follows that   $c(u)=j$ for every $u \in N_{\{j,k\}}(w)$.
This proves the claim.

Let $A=\tilde{A}_{\{1\}} \cup \tilde{A}_{\{2\}} \cup \tilde{A}_{\{3\}} \cup N_{\{i,j\}}(w)\cup N_{\{i,k\}}(w)\cup N_{\{j,k\}}(w)\cup \{w\}$. Let $M$ be the 
palette of $G$ obtained from $L$ by first updating the palettes of the vertices in $\tilde{A}$ with respect to $A$, and then updating the palettes of the
vertices of $\tilde{C}$ with respect to $\tilde{A}$. It follows that $c$
is a coloring of $(G,M,\mathcal{Z})$.

 Let $F_{\{i,k\}}$ be  the vertex set of the 2-vertex components of $\tilde{C}$ such that both vertices of the component have a neighbor in $\tilde{A}_{\{i,k\}}$, and let $F_{\{j,k\}}$ be  the vertex set of the 2-vertex components of $\tilde{C}$ such that both vertices of the component have a neighbor in $\tilde{A}_{\{j,k\}}$.
Let $C'$ be the set of vertices 
$v\in \tilde{C} \setminus (F_{\{i,k\}} \cup F_{\{j,k\}})$
 with $N(v)\cap \tilde{A}\subseteq M_{\{i,k\}}(w)\cup M_{\{j,k\}}(w)$. By construction, $c(v)\in L''(v)$ for all $L''\in\mathcal{L}^w_{\{i,j,k\}}$ and $v\in V(G^w_{\{i,j,k\}})\setminus (C' \cup F_{\{i,k\}} \cup F_{\{j,k\}}) $. Next, we consider vertices in $C'\setminus D^w_{(i,j,k)}$. By construction, every vertex $c_1\in C'\setminus D^w_{(i,j,k)}$ has both a neighbor in $M_{\{i,k\}}(w)$, and 
a neighbor in $M_{\{j,k\}}(w)$, is adjacent to some vertex $c_2\in \tilde{C}$,
and $c_2$ has a neighbor in at least one of $\tilde{A}_{\{i,k\}}$ and
$\tilde{A}_{\{j,k\}}$. Moreover, 
$|L^{\tilde{A},\tilde{C}}_c(c_2)|=1$, and so, since $c(c_1)\ne c(c_2)$, it follows that $c(c_1)\in L''(c_1)$ for all $L''\in\mathcal{L}^w_{(i,j,k)}$. 
Now $\mathcal{P}^w_{(i,j,k)}$ is colorable by \ref{lemma1}, and  \textit{(5c)} follows. This proves \textit{(5)}.

\bigskip

\noindent \textit{(6) Assume $(G,L,Z)$ does not have a type I or a type II coloring. Then there exists a restriction $\{(G',L',Z')\}$ of $(G,L,Z)$ such that the following hold:}

\bigskip

\noindent \textit{(6a) $\tilde{A}\cup \tilde{B}\subseteq V(G')$},

\bigskip

\noindent \textit{(6b) $L'(v)=L(v)$ for every $v \in \tilde{A}\cup \tilde{B}$, and $|L'(v)| \leq 2$ for every $v \in V(G')\cap \tilde{C}$, and}

\bigskip
 
\noindent \textit{(6c) If $(G,L,Z)$ is colorable, then $(G',L',Z')$ is colorable.}
\bigskip

\noindent \textit{Moreover, $(G',L',Z')$ can be computed in time $O(|V(G)|^2)$.}

\bigskip

\noindent Proof: Let $C'$ be the set of vertices in $\tilde{C}$ anticomplete to $\tilde{A}_{\{1\}}\cup \tilde{A}_{\{2\}}\cup \tilde{A}_{\{3\}}$. Let $G'=G\setminus C'$, let $W=\{N(v) \cap \tilde{A} : v \in C'\}$, and let $Z'=Z \cup W$. 
Update the palettes of all the vertices in $\tilde{C} \setminus C'$ with respect to $\tilde{A}_{\{1\}}\cup \tilde{A}_{\{2\}}\cup \tilde{A}_{\{3\}}$, and let $L'$ be the palette $L|_{V(G')}$. Clearly, \textit{(6a)} and \textit{(6b)} hold. Suppose
that $(G,L,Z)$ is colorable. Since $(G,L,Z)$ has no type I and no type II coloring, it
follows that every set in $Z'$ is monochromatic in every coloring of $(G,L,Z)$,
and therefore $(G',L',Z')$ is colorable. Thus \textit{(6c)} holds.

Now, we argue that $\{(G',L',Z')\}$ is a restriction of $(G,L,Z)$. Let 
$c$ be a coloring of $(G',L',Z')$; we need to show that $c$ can
be extended to a coloring of $(G,L,Z)$ in time $O(|V(G)|)$.
Let $c_1\in C'$. Then $c_1$ is anticomplete to $\tilde{A}_{\{1\}}\cup \tilde{A}_{\{2\}}\cup \tilde{A}_{\{3\}}$, and $|L(c_1)|=3$.
It follows that in every coloring of $(G,L,Z)$ only one color appears in $N(c_1) \setminus \tilde{C}$. 
Recall also that $c_1$ has at most one neighbor in $C$. Now, repeatedly applying
\ref{nbrs} to vertices of $C'$,  it follows that $c$ can be extended to 
a coloring of $(G,L,Z)$ in time $O(|V(G)|)$. This prove \textit{(6)}.

\bigskip

Let $\mathcal{P}=\mathcal{P}_1\cup \mathcal{P}_2\cup \mathcal{P}_3\cup \{(G',L',Z')\}$ be the union of the sets of restrictions from \textit{(2)},\textit{(4)}, \textit{(5)} and \textit{(6)}. Then $\mathcal{P}$ is a restriction of $(G,L,Z)$.
We claim that $\mathcal{P}$ satisfies the conclusion of the theorem. By \textit{(2)},\textit{(4)}, \textit{(5)} and \textit{(6)}, it follows that $\mathcal{P}$ has size $O(|V(G)|^7)$, and that \textit{(a)} and \textit{(b)} hold. Next, we show that \textit{(c)} holds. Since $\mathcal{P}$ is a restriction  of $(G,L,Z)$, by definition, it follows that if $\mathcal{P}$ is colorable, then $(G,L,Z)$ is colorable. By \textit{(2c)}, if $(G,L,Z)$ has a type I coloring, then $\mathcal{P}_1$ is colorable. So we may assume $(G,L,Z)$ does not have a type I coloring. By \textit{(4c)}, if some vertex of $\tilde{C}$ has neighbors in all three of $\tilde{A}_{\{1,2\}}$, $\tilde{A}_{\{1,3\}}$ and $\tilde{A}_{\{2,3\}}$, then $(G,L,Z)$ is colorable if and only if $\mathcal{P}_2$ is colorable. So we may assume no such vertex exists. By \textit{(5c)}, if $(G,L,Z)$ has a type II coloring, then $\mathcal{P}_3$ is colorable. So we may assume $(G,L,Z)$ does not have a type II coloring. By \textit{(6)}, it follows that $(G,L,Z)$ is colorable if and only if $(G',L',Z')$ is colorable. This proves \ref{lemma2}.

\end{proof}

\section{Cleaning}

In this section, we identify two configurations whose presence in a graph $G$ allows us to delete some vertices and obtain an induced subgraph $G'$ of $G$,
 such that $G'$ is $k$-colorable if and only $G$ is $k$-colorable. Furthermore, these two configurations can be efficiently recognized, and so at the expense of carrying out a polynomial time procedure we may assume a given graph does not contain either configuration.






Let $G$ be a graph. Recall that a vertex $v \in V(G)$ is \textit{dominated} by a vertex $u \in V(G) \setminus \{v\}$ if $u$ is non-adjacent to $v$ and $N(v) \subseteq N(u)$. In terms of coloring, dominated vertices are useful because of the following: 

\begin{lemma}\label{dominated}
Let $\mathcal{F}$ be a set of graphs and $G$ be an $\mathcal{F}$-free graph. If $v\in V(G)$ is dominated by $u\in V(G)\setminus \{v\}$, then $G\setminus v$ is an $\mathcal{F}$-free graph which is $k$-colorable if and only if $G$ is $k$-colorable. Furthermore, we can extend any $k$-coloring of $G\setminus v$ to a $k$-coloring of $G$ in constant time.

\end{lemma}

\begin{proof}
Clearly, $G\setminus v$ is $\mathcal{F}$-free, and if $G$ is $k$-colorable, then $G\setminus v$ is $k$-colorable. Now, suppose $c$ is a $k$-coloring of $G\setminus v$. Since $v$ is non-adjacent to $u$ and $N(v)\subseteq N(u)$, it follows, assigning $c(v)=c(u)$, that $c$ extends to a coloring of $G$. This proves \ref{dominated}.
\end{proof}

Let $G$ be a graph. Recall that for a subset $X\subseteq V(G)$, we say that a vertex $v\in V(G)\setminus X$ is \textit{mixed} on $X$, if $v$ is not complete and
not anticomplete to $X$.
We say that a pair of disjoint non-empty stable sets $(A, B)$ form a \textit{non-trivial homogeneous pair of stable sets} if $2<|A|+|B|< |V(G)|$ and no vertex $v \in V(G)\setminus (A\cup B)$ is mixed on $A$ or mixed on $B$. In terms of coloring, non-trivial homogeneous pairs of stable sets are useful because of the following:


\begin{lemma}\label{pair}
Let $\mathcal{F}$ be a set  of graphs and let $G$ be an $\mathcal{F}$-free graph. Let $(A,B)$ be a non-trivial homogeneous pair of stable sets and choose $a\in A$ and $b\in B$, adjacent if possible. Then $G\setminus ((A\cup B) \setminus \{a,b\})$ is an $\mathcal{F}$-free graph which is $k$-colorable if and only if $G$ is $k$-colorable. Furthermore, we can extend any $k$-coloring of $G\setminus ((A\cup B) \setminus \{a,b\})$ to a $k$-coloring of $G$ in linear time.

\end{lemma}

\begin{proof}
Clearly, $G\setminus ((A\cup B) \setminus \{a,b\})$ is $\mathcal{F}$-free, and if $G$ is $k$-colorable, then $G\setminus ((A\cup B) \setminus \{a,b\})$ is $k$-colorable. Now, suppose $c$ is a $k$-coloring of $G\setminus ((A\cup B) \setminus \{a,b\})$. By our choice of $a\in A$ and $b\in B$, if $A$ is not anticomplete to $B$, then $c(a)\ne c(b)$. Since both $A$ and $B$ are stable and no vertex $v \in V(G)\setminus (A\cup B)$ is mixed on $A$ or mixed on $B$, it follows, assigning $c(a')=c(a)$ for all $a'\in A$ and $c(b')=c(b)$ for all $b'\in B$, that $c$ extends to a coloring of $G$. This proves \ref{pair}.
\end{proof}

We say a graph $G$ is \textit{clean} if $G$ has no dominated vertices and no non-trivial homogeneous pair of stable sets. 

\begin{lemma}\label{cleaning} There is an algorithm with the following specifications:

\bigskip

{\bf Input:} A graph $G$.

\bigskip

{\bf Output:} A clean induced subgraph $G'$ of $G$ such that, for
every integer $k$,  $G'$ is $k$-colorable if and only if $G$ is $k$-colorable.

\bigskip

{\bf Running time:} $O(|V(G)|^5)$.

\bigskip

\noindent Furthermore, we can extend any $k$-coloring of $G'$ to a $k$-coloring of $G$ in linear time.



\end{lemma}

\begin{proof}

\noindent Since there are $O(|V(G)|^2)$ potential pairs of non-adjacent vertices $u,v\in V(G)$ and we can verify in time $O(|V(G)|)$ if $N(v) \subseteq N(u)$, it follows that in time $O(|V(G)|^3)$ we can find a dominated vertex in $G$, if one exists. In \cite{PAIR}, King and Reed give an algorithm, that runs is time
$O(|V(G)|^4)$, which is easily modified to produce a non-trivial homogeneous pair of stable sets, if one exists. And so, we can find a dominated vertex or a non-trivial homogeneous pair of stable sets in time $O(|V(G)|^4)$, if one exists. Hence, iteratively, \ref{cleaning} follows from the observations made in \ref{dominated} and \ref{pair}.

\end{proof}

Thus, \ref{cleaning} implies that given a graph $G$, we may assume $G$ is clean at the expense of carrying out a time $O(|V(G)|^5)$ procedure. It is also clear from \ref{dominated} and \ref{pair}, that given a graph $G$ we may extend a $k$-coloring of the resulting clean induced subgraph $G'$ of $G$ produced by \ref{cleaning} to a $k$-coloring of $G$ in time $O(|V(G)|)$.

\section{A Useful Lemma}

In this section we prove a general lemma which will be of great use when trying to apply \ref{lemma2} to clean graphs.

\begin{lemma}\label{3edge}

Let $G$ be a clean, connected $\{P_7,C_3\}$-free graph with $V(G)=P\cup Q\cup R\cup S\cup T$ such that:

\begin{itemize}

\item $P\cup T$ is anticomplete to $R\cup S$, 

\item $S$ is anticomplete to $P\cup Q$, 


\item every vertex in $R$ has a neighbor in $Q$, and

\item there exist $q_0 \in Q$ and $p_1,p_2,p_3\in P$ such that 
$p_1-p_2-p_3-q_0$ is an induced path.

\item for every $q\in Q$, there exist $p_2,p_3\in P$ and
$p_1 \in P \cup \{q_0\}$ such that $p_1-p_2-p_3-q$ is an induced path.




\end{itemize}

\noindent Then the following hold:

\begin{enumerate}

\item $S$ is empty

\item If for every $q\in Q$, there exist $p_1,p_2, p_3\in P$ such that 
$p_1-p_2-p_3-q$ is an induced path, then every component of $R$ has size at 
most two.

\item Every component of $R$ is bipartite, and if some component $X$ of
$R$ has more than two vertices, then $q_0$ is complete to one
side of the bipartition of $G[X]$.

\end{enumerate}

\end{lemma}

\begin{proof}

\noindent \textit{(1) Let $q \in Q$.   There is no   induced path $q-r_1-r_2-r_3$ such that $r_1 \in R $ and $r_2,r_3 \in S$. Moreover, 
if there exist $p_1,p_2,p_3\in P$ such that $p_1-p_2-p_3-q$ is an induced path, then there is no induced path $q-r_1-r_2-r_3$ such that $r_1,r_2,r_3\in R\cup S$.}

\bigskip

\noindent Suppose there exist $p_1,p_2,p_3\in P$ such that $p_1-p_2-p_3-q$ is an induced path, and 
$q-r_1-r_2-r_3$ is an induced path with $r_1,r_2,r_3\in R\cup S$. Then, since $P$ is anticomplete to $R\cup S$, it follows that $p_1-p_2-p_3-q-r_1-r_2-r_3$ is a $P_7$ in $G$, a contradiction. Next suppose that $r_1 \in R$ and 
$r_2,r_3 \in S$. There exist $p_1,p_2 \in P$ such 
that $q-p_1-p_1-q_0$ is an induced path. Since $q_0-r_1-r_2-r_3$ is not an 
induced
path in $G$ by  the previous argument, it follows that $q_0$ is anticomplete to 
$\{r_1,r_2,r_3\}$. But now  $q_0-p_1-p_1-q-r_1-r_2-r_3$ is an induced path in 
$G$,  a contradiction.
This proves \textit{(1)}.

\bigskip 

\noindent \textit{(2) Every vertex in $S$ has a neighbor in $R$}.

\bigskip

\noindent Partition $S=S'\cup S''$, where $S'$ is the set of vertices in $S$ with a neighbor in $R$ and $S''=S\setminus S'$. Suppose $S''$ is non-empty. Since $G$ is connected and $S''$ is anticomplete to $V(G)\setminus S$, it follows that there exists $s''\in S''$ adjacent to $s'\in S'$. By definition, there exists $r\in R$ adjacent to $s'$ and non-adjacent to $s''$. Since every vertex in $R$ has a neighbor in $Q$, there exists $q\in Q$ adjacent to $r$. Now  $q-r-s'-s''$ is an induced path, contrary to \textit{(1)}.  This proves \textit{(2)}.

\bigskip

\noindent \textit{(3) $S$ is stable}.

\bigskip

\noindent Suppose $s,s'\in S$ are adjacent. By \textit{(2)}, there exists $r\in R$ adjacent to $s$. Since $G$ is triangle-free, $s'$ is non-adjacent to $r$. Since every vertex in $R$ has a neighbor in $Q$, there exists $q\in Q$ adjacent to $r$. However, then $q-r-s-s'$ is an induced path, contradicting \textit{(1)}. This proves \textit{(3)}.

\bigskip



\noindent Now, we prove \ref{3edge}.1. Consider a vertex $s\in S$. Since $S$ is anticomplete to $P\cup Q$ and, by \textit{(3)}, $S$ is stable, it follows that $S$ is anticomplete to $V(G)\setminus R$. By \textit{(2)}, there exists $r\in R$ adjacent to $s$. Since every vertex in $R$ has a neighbor in $Q$, there exists $q\in Q$ adjacent to $r$; if possible, choose $q$ and $r$ such that there
exist $p_1,p_2,p_3 \in P$ such that $q-p_1-p_2-p_3$ is an induced path.
Since $s$ is not dominated by $q$, there exists $r'\in N(s)\setminus N(q)$. Since $G$ is triangle-free, $r'$ is non-adjacent to $r$, and it follows that $q-r-s-r'$ is an induced path. It follows from \textit{(1)} there do not exist
 $p_1,p_2,p_3 \in P$ such that $q-p_1-p_2-p_3$ is an induced path, and
in particular $q \neq q_0$. This implies that $q_0$ is anticomplete to $N(s)$,
and that there exist $p_1,p_2 \in P$ such that $q-p_1-p_2-q_0$ is an induced path. But now $q_0-p_2-p_1-q-r-s-r'$ is a $P_7$ in $G$, a contradiction.
This proves \ref{3edge}.1.

\bigskip

\noindent \textit{(4) Every component of $R$ is bipartite}.

\bigskip

\noindent Suppose $X$ is a component of $R$ which is not bipartite. Since $G$ is $\{P_7,C_3\}$-free, it follows that $G[X]$ contains either a $C_5$ or a $C_7$. Let $x_1-x_2-...-x_{2k+1}-x_1$ be either a 5-gon or a 7-gon given by $x_1,...,x_{2k+1}\in R$ with $k\in \{2,3\}$. Let $q \in Q$  be a vertex with a neighbor in $\{x_1, ..., x_{2k+1}\}$. Since $2k+1$ is odd and $G$ is triangle-free,
we may assume that $q$ is adjacent to $x_1$ and non-adjacent to $x_2,x_3$.
But now $q-x_1-x_2-x_3$ is an induced path, and so \textit{(1)} implies that
there do not exist $p_1,p_2,p_3 \in P$ such that $q-p_1-p_2-p_3$ is an induced 
path.  This implies that $q_0$ is anticomplete to $\{x_1, ..., x_{2k+1}\}$,
and that there exist $p_1,p_2 \in P$ such that $q-p_1-p_2-q_0$ is an induced path. But now $q_0-p_2-p_1-q-x_1-x_2-x_3$ is a $P_7$ in $G$, a contradiction.
This proves \textit{(4)}.

\bigskip



\noindent \textit{(5) Let $X$ be a component of $R$, and $(X_1,X_2)$ be a 
bipartition of $G[X]$. Let $q \in Q$ be such that there exist 
$p_1,p_2,p_3 \in P$ where $q-p_1-p_2-p_3$ is an induced 
path. Then $q$ is not mixed on either $X_1$ or $X_2$}.

\bigskip 

\noindent Suppose there exists a vertex $q\in Q$ adjacent to $x$ and non-adjacent to $x'$, where $x,x'\in X_1$. Since $X$ is connected, by choosing $x$ and $x'$ at minimum distance from each other in $G[X]$, we may assume that there
exists $x_2 \in X_2$ adjacent to both $x$ and $x'$. Since $G$ is triangle-free,
it follows that $q$ is non-adjacent to $x_2$.
Now $q-x-x_2-x'$ is an induced path, and so  \textit{(1)} implies that
there do not exist $p_1,p_2,p_3 \in P$ such that $q-p_1-p_2-p_3$ is an induced 
path. This proves~\textit{(5)}.

\bigskip

Let $X$ be a component of $R$, and $(X_1,X_2)$ be a bipartition of $G[X]$.
First we prove  \ref{3edge}.2.
Suppose that for every vertex of $Q$ there exist $p_1,p_2,p_3 \in P$ such 
that $q-p_1-p_2-p_3$ is an induced  path. By \textit{(5)}, no vertex of $Q$ is 
mixed on $X_1$, and similarly, no vertex of
$Q$ is mixed on  $X_2$. Since $P\cup T$ is anticomplete to $R$, it follows that $V(G)\setminus Q$ is anticomplete to $R$, and in particular to $X$. Hence, $(X_1,X_2)$ is a homogeneous pair of stable sets, and so, since $G$ is clean, it follows that $|X|\leq 2$. This proves \ref{3edge}.2.

Finally, we prove \ref{3edge}.3. Suppose that $|X|>2$.  Since $(X_1,X_2)$ is 
not a homogeneous pair of stable sets, it follows that there exists a vertex $q\in Q$ adjacent to $x$ and non-adjacent to $x'$, where $x,x'\in X_1$, say.
Since $X$ is connected, by choosing $x$ and $x'$ at minimum distance from each other in $G[X]$, we may assume that there
exists $x_2 \in X_2$ adjacent to both $x$ and $x'$. Since $G$ is triangle-free,
it follows that $q$ is non-adjacent to $x_2$. By \textit(5), $q \neq q_0$,
and there exist $p_1,p_2 \in P$ such that $q-p_1-p_2-q_0$ is a path.
Since $q_0-p_2-p_1-q-x-x_2-x'$ is not a $P_7$ in $G$, it follows that
$q_0$ has a neighbor in $\{x,x',x_2\}$, and \ref{3edge}.3 follows from
\textit(5).
\end{proof}

\section{7-gons}

In this section we show that if a $\{P_7,C_3\}$-free graph contains a 7-gon, then in polynomial time we can decide if the graph is 3-colorable, and give a coloring if one exists. Let $C$ be an $n$-gon in a graph $G$. For a vertex $v\in V(G)\setminus V(C)$, we call the neighbors of $v$ in $V(C)$ the \textit{anchors of $v$ in $C$}.

We begin with some definitions. Let $C$ be a 7-gon in a graph $G$ given by $v_0-v_1-v_2-v_3-v_4-v_5-v_6-v_0$. We say that a vertex $v\in V(G)\setminus V(C)$ is:

\begin{itemize}

\item a \textit{clone at $i$}, if $N(v)\cap V(C)=\{v_{i-1}, v_{i+1}\}$ for some $i\in\{0,1,...,6\}$, where all indices are mod 7, 

\item a \textit{propeller at $\{i,i+3\}$}, if $N(v)\cap V(C)=\{v_{i}, v_{i+3}\}$ for some $i\in\{0,1,...,6\}$, where all indices are mod 7,


\item a \textit{star at $i$}, if $N(v)\cap V(C)=\{v_{i-2},v_i, v_{i+2}\}$for some $i\in\{0,1,...,6\}$, where all indices are mod 7.

\end{itemize}

\noindent The following shows how we can partition the vertices of $G$ based on their anchors in $C$.

\begin{lemma}\label{C7}

Let $G$ be a $\{P_7,C_3\}$-free graph, and suppose $C$ is a 7-gon in $G$ given by $v_0-v_1-v_2-v_3-v_4-v_5-v_6-v_0$. If $v\in V(G)\setminus V(C)$, then for some $i\in\{0,1,...,6\}$ either:

\begin{enumerate}

\item $v$ is a clone at $i$,

\item $v$ is a propeller at $\{i,i+3\}$,

\item $v$ is a star at $i$, or

\item $v$ is anticomplete to $V(C)$.
\end{enumerate}
\end{lemma}

\begin{proof}
Consider a vertex $v\in V(G)\setminus V(C)$. If $v$ is anticomplete to $V(C)$, then \ref{C7}.4 holds. Thus, we may assume $N(v)\cap V(C)\ne \emptyset$. Since
$G$ is triangle-free, and $7$ is odd, we may assume that $v$ is adjacent to 
$v_0$, and anticomplete to $\{v_6,v_1,v_2\}$.  If $v$ is adjacent to $v_4$, 
then, 
since $G$ is triangle-free, \ref{C7}.2 holds, so we may assume not. Since
$v-v_0-v_1-v_2-v_3-v_4-v_5$ is not a $P_7$ in $G$, it follows that $v$ has 
a neighbor in $\{v_3,v_5\}$. If $v$ is complete to $\{v_3,v_5\}$, then 
\ref{C7}.3 holds; if $v$ is adjacent to $v_3$ and not to $v_5$, then 
\ref{C7}.2 holds; and if $v$ is adjacent to $v_6$ and not to $v_3$, then
\ref{C7}.1 holds. This proves~\ref{C7}.
\end{proof}


Let $G$ be a triangle-free graph, and let $C$ be a 7-gon in $G$ given by $v_0-v_1-v_2-v_3-v_4-v_5-v_6-v_0$. 
Using \ref{C7} we partition $V(G)\setminus V(C)$ as follows:

\begin{itemize}

\item Let $CL^C(i)$ be the set of clones at $i$, and define $CL^C=\bigcup\limits_{i=0}^6CL^C(i)$.

\item $P^C(i)$ be the set of propellers at $\{i,i+3\}$, and define $P^C=\bigcup\limits_{i=0}^6P^C(i)$.



\item Let $S^C(i)$ be the set of stars at $i$ and define $S^C=\bigcup\limits_{i=0}^6S^C(i)$.

\item Let $A^C$ be the set of vertices anticomplete to $V(C)$.

\end{itemize}

\bigskip

\noindent By \ref{C7}, it follows that $V(G)=V(C)\cup CL^C \cup P^C\cup S^C \cup A^C$. Furthermore, we partition $A^C=X^C\cup Y^C\cup Z^C$, where

\begin{itemize}

\item $X^C$ is the set of vertices in $A^C$ with a neighbor in $P^C$,

\item $Y^C$ is the set of vertices in $A^C\setminus X^C$ with a neighbor in $S^C$, and

\item $Z^C=A^C\setminus (X^C\cup Y^C)$.

\end{itemize}

\bigskip

\noindent And so, for a given 7-gon $C$ in time $O(|V(G)|^2)$ we obtain the partition $V(C)\cup CL^C \cup P^C\cup S^C \cup X^C\cup Y^C \cup Z^C$ of $V(G)$. Now, we establish several properties of this partition.

\begin{lemma}\label{7-1}

If $G$ is a $\{P_7,C_3\}$-free graph, then $A^C$ is anticomplete to $CL^C$ for every 7-gon $C$ in $G$.

\end{lemma}

\begin{proof} Let $C$ be a 7-gon in $G$ given by $v_0-v_1-v_2-v_3-v_4-v_5-v_6-v_0$. Suppose there exists a vertex $u\in A^C$ adjacent to $v\in CL^C$. By symmetry, we may assume $v\in CL^C(0)$. However, then $u-v-v_1-v_2-v_3-v_4-v_5$ is a $P_7$ in $G$, a contradiction. This proves \ref{7-1}.
\end{proof}

\noindent \ref{7-1} implies  that:

\begin{lemma}\label{7-2} If $G$ is a $\{P_7,C_3\}$-free graph, then for every 7-gon $C$ in $G$ the following hold:


\begin{enumerate}

\item $A^C$ is anticomplete to $V(C) \cup CL^C$.

\item Every vertex in $X^C$ has a neighbor in $P^C$ (and possibly $S^C$).

\item $Y^C\cup Z^C$ is anticomplete to $P^C$.

\item Every vertex in $Y^C$ has a neighbor in $S^C$.

\item $Z^C$ is anticomplete to $S^C$.

\end{enumerate}

\end{lemma}

\noindent For a fixed subset $X$ of $V(G)$, we say a vertex \textit{$v\in V(G)\setminus X$ is mixed on an edge of $X$}, if there exist adjacent $x,y\in X$ such that $v$ is mixed on $\{x,y\}$. We need the following two facts:

\begin{lemma}\label{7AC}

If $G$ is a $\{P_7,C_3\}$-free graph, then for every 7-gon $C$ in $G$ the following hold:

\begin{enumerate}

\item No vertex in $P^C$ is mixed on an edge of $A^C$.

\item $X^C$ is stable and anticomplete to $Y^C\cup Z^C$.

\end{enumerate}
\end{lemma}

\begin{proof}

Let $C$ be a 7-gon in $G$ given by $v_0-v_1-v_2-v_3-v_4-v_5-v_6-v_0$. Suppose for adjacent $a,a'\in A^C$, there exists $p\in P^C$ which is adjacent to $a'$ and non-adjacent to $a$. By symmetry, we may assume $p\in P^C(0)$. However, then $a-a'-p-v_3-v_4-v_5-v_6$ is a $P_7$ in $G$, a contradiction. This proves \ref{7AC}.1.

Consider a vertex $x\in X^C$. By \ref{7-2}.2, $x$ has a neighbor $p\in P^C$. If there exists $x'\in N(x)\cap A^C$, then, since $G$ is triangle-free,
it follows that  $p$ is non-adjacent to $x'$, and so $p$ is mixed on an edge of $A^C$, contradicting \ref{7AC}.1. This proves \ref{7AC}.2.

\end{proof}

\begin{lemma}\label{7comps}

Let $G$ be a clean, connected $\{P_7,C_3\}$-free graph. Then for every 7-gon $C$ in $G$ the following hold:

\begin{enumerate}

\item $Z^C$ is empty.

\item Every component of $Y^C$ is a singleton or an edge.

\end{enumerate}
\end{lemma}

\begin{proof}

Let $C$ be a 7-gon in $G$ given by $v_0-v_1-v_2-v_3-v_4-v_5-v_6-v_0$. 

\bigskip

\noindent \textit{(1) For every $s\in S^C$, there exists $t_1,t_2,t_3\in V(C)$ such that $t_1-t_2-t_3-s$ is an induced path.}

\bigskip

 \noindent Consider a vertex $s\in S^C$. By symmetry, we may assume $s\in S^C(0)$. And so, $v_4-v_3-v_2-s$ is the desired induced path. This proves \textit{(1)}.
 
\bigskip





\noindent By \ref{7-2}, \ref{7AC}.2 and \textit{(1)}, we may apply \ref{3edge} letting $P=V(C)$, $Q=S^C$, $R=Y^C$, $S=Z^C$, and $T=CL^C\cup P^C\cup X^C$. Then \ref{3edge}.1 and \ref{3edge}.2 follow immediately from \ref{3edge}. This proves \ref{7comps}.


\end{proof}

\noindent Now, we prove the main result of the section.

\begin{lemma}\label{7gon} There is an algorithm with the following specifications:

\bigskip

{\bf Input:} A clean, connected $\{P_7,C_3\}$-free graph $G$ which contains a 7-gon.

\bigskip

{\bf Output:} A $3$-coloring of $G$, or a determination that none exists.

\bigskip

{\bf Running time:} $O(|V(G)|^{10})$.

\end{lemma}

\begin{proof} 

Let $C$ be a 7-gon in $G$ given by $v_0-v_1-v_2-v_3-v_4-v_5-v_6-v_0$, and observe that $C$ can be found in time $O(|V(G)|^7)$. Since $G$ is clean, by \ref{7comps}.1, it follows that $Z^C$ is empty, and so we may partition $V(G)=V(C)\cup CL^C \cup P^C\cup S^C \cup X^C\cup Y^C$ as usual. Next, fix a $3$-coloring $c$ of $G[V(C)]$. Define the order 3 palette $L^C_c$ of $G$ as follows:

$$L^C_c(v) = \begin{cases}
 \{c(v)\} & \text{, \hspace{2ex}if $v\in V(C)$} \\
 \{1,2,3\} & \text{, \hspace{2ex}otherwise}\\
\end{cases}$$

\noindent Next, update the vertices in $CL^C\cup P^C\cup S^C$ with respect to $V(C)$. And so, $|L^C_c(v)| \leq 2$ for all $v \in V(G) \setminus (X^C\cup Y^C)$, while $|L^C_c(v)|=3$ for all $v\in X^C\cup Y^C$. Observe that, by construction, $(G,L^C_c)$ is colorable if and only the $3$-coloring $c$ of $G[V(C)]$ extends to a $3$-coloring of $G$. 

Let $A'=P^C\cup S^C$, and for every non-empty subset $X \subseteq \{1,2,3\}$, define $A'_X=\{a\in A'$ with $L^C_c(a)=X\}$.


\bigskip

\noindent \textit{(1) For every $u\in X^C\cup Y^C$ and $\{i,j,k\}=\{1,2,3\}$, $N(u) \cap A'_{\{i,j\}}$ is complete to $A'_{\{i,k\}} \setminus N(u)$.}

\bigskip

\noindent It is enough to show that for every $x \in N(u)\cap A'_{\{i,j\}}$ and $y \in A'_{\{i,k\}}\setminus N(u)$, such that $x$ is non-adjacent to $y$,  there exists an induced 6-vertex path $x-p_1-...-p_5$ with $p_1,...,p_5\in V(C) \cup CL^C\cup \{y\}$. For if such a path exists, then, since, by \ref{7-2}.1, $u$ is anticomplete to $V(C) \cup CL^C\cup \{y\}$, it follows that $u-x-p_1-...-p_5$ is a $P_7$ in $G$, a contradiction. 

Since $x \in A'_{\{i,j\}}$ and $y \in A'_{\{i,k\}}$, it follows from the definition
of updating that  all the anchors of $x$ are colored $k$, and all the anchors of $y$ are colored $j$. In particular, this implies that $x$ and $y$ have no anchors in common. 

Let $\{a,b\}=\{x,y\}$. 
First, suppose $a$ is a star. By symmetry, we may assume $a \in S^C(0)$.
Since $a$ and $b$ have no anchors in common, it follows that
$b$ is anticomplete to $\{v_0,v_2,v_5\}$. Since $b \in P^C\cup S^C$, it follows that $|N(b)\cap (V(C)\setminus \{v_0,v_2,v_5\})|\geq 2$ and so, by symmetry and \ref{C7}, we may assume that $v_1$ is an anchor of $b$, that is, that $b \in P^C(1)\cup S^C(1)\cup S^C(6)$. Further by symmetry,
we may assume that  $b \in P^C(1)\cup S^C(1)$. Suppose $a=x$. If $b\in P^C(1)$, then $x-v_0-v_1-y-v_4-v_3$ is the desired path, and if $b\in S^C(1)$, then $x-v_0-v_1-y-v_3-v_4$ is the desired path. Thus, we may assume $b=x$. If $b\in P^C(1)$, then $x-v_4-v_3-v_2-y-v_0$ is the desired path, and if $b\in S^C(1)$, then $x-v_3-v_4-v_5-y-v_0$ is the desired path. Hence, we may assume neither of $x, y$ is a  star, that is, that both $x$ and $y$ are propellers. By symmetry, we may assume $x \in P^C(0)$, and therefore  $y$ is anticomplete to $\{v_0,v_3\}$. Since $y \in P^C$, it follows that $|N(y)\cap (V(C)\setminus \{v_0,v_3\})|=2$ and so, by symmetry, we may assume that $v_1$ is an anchor of $y$, that is, that $y \in P^C(1)\cup P^C(5)$. If $y \in P^C(1)$, then $x-v_0-v_1-y-v_4-v_5$ is the desired path, and if $y \in P^C(5)$, then $x-v_3-v_2-v_1-y-v_5$ is the desired path. By our initial observation, this proves \textit{(1)}.


\bigskip

\noindent \textit{(2) For all distinct $i,j \in \{1,2,3\}$ some vertex of $V(C)\cup CL^C$ is complete to $A'_{\{i,j\}}$.}

\bigskip



\noindent If $A'_{\{i,j\}}=\emptyset$, then \textit{(2)} trivially holds. Thus, we may assume $A'_{\{i,j\}}\ne\emptyset$. Let $\{i,j,k\}=\{1,2,3\}$ and define $K$ to be the set of vertices of $V(C)$ with a neighbor in $A'_{\{i,j\}}$. Since we have updated, it follows that $c(v)=k$ for every $v\in K$. Since $G[V(C)]$ has no stable set of size 4, it follows that $|K|\leq 3$. Since $A'_{\{i,j\}}\subseteq P^C\cup S^C$, it follows that $|K|\geq 2$. If $|K|=2$, then $A_{\{i,j\}}\subseteq P^C$, and it follows that $K$ is complete to $A'_{\{i,j\}}$. Hence, we may assume $|K|=3$. By symmetry, we may assume that $K=\{v_0,v_3,v_5\}$. Since $G$ is triangle-free, it follows that $A'_{\{i,j\}}\subseteq P^C(0)\cup S^C(5)$, and so $v_3$ is complete to $A_{\{i,j\}}$. This proves \textit{(2)}.

\bigskip

\noindent By \ref{7AC}.2 and \ref{7comps}, it follows that every component of $X^C\cup Y^C$ has at most two vertices. And so, by \ref{7-2}, \textit{(1)} and \textit{(2)}, we can apply \ref{lemma2} with 

\begin{itemize}

\item $\tilde{A}=A'$,
\item $\tilde{B}=V(C) \cup CL^C$,  
\item $\tilde{C}=X^C\cup Y^C$, and 
\item $Z=\emptyset$.

\end{itemize}

\noindent Let $\mathcal{P}^C_c$ be the restriction of $(G,L^C_c, \emptyset)$, of size  $O(|V(G)|^7)$, thus obtained. By \ref{lemma2}, $\mathcal{P}^C_c$ can be computed in time $O(|V(G)|^7)$. By \ref{lemma2}\textit{(c)}, we have that $(G,L^C_c, \emptyset)$ (and, equivalently, $(G,L^C_c)$ ) is colorable if and only if $\mathcal{P}^C_c$ is colorable. Consider $(G',L',X')\in \mathcal{P}^C_c$. Since $|L^C_c(v)| \leq 2$ for all $v \in V(G) \setminus (X^C\cup Y^C)$, by \ref{lemma2}\textit{(b)}, it follows that $|L'(v)|\leq 2$ for all $v\in V(G')$. Thus, since $|X'|$ has size $O(|V(G)|)$, 
applying \ref{checkSubsets}, we can test in time
$O(|V(G)|^3)$ if $(G',L',X')$ is colorable, and extend the 
coloring to $(G,L^C_c)$.  Consequently, 
via $O(|V(G)|^7)$ applications of \ref{checkSubsets}, we can determine if $\mathcal{P}^C_c$ is colorable and extend any coloring of a colorable $(G',L',X') \in \mathcal{P}^C_c$ to a coloring of $G$.  That is, in time $O(|V(G)|^{10})$ we can determine if the $3$-coloring $c$ of $G[V(C)]$ extends to a $3$-coloring of $G$, and give an explicit $3$-coloring $c'$ of $G$ such that $c'(v)=c(v)$ for all $v\in V(C)$, if one exists. Finally, let $\mathcal{P}$ be the union of $\mathcal{P}^C_c$ taken over all $3$-colorings $c$ of $G[V(C)]$. Since there are at most $7^3$ $3$-colorings of $G[V(C)]$, it follows that we can test in time  
$O(|V(G)|^{10})$ if $\mathcal{P}$ is colorable.
Since every $3$-coloring of $G$ restricts to a $3$-coloring of $G[V(C)]$, it follows that $G$ is 3-colorable if and only if $\mathcal{P}$ is colorable. This proves \ref{7gon}.

\end{proof}

\section{Shells}

We remind the reader, that a shell in a graph $G$ is a pair $(C,p)$, 
where $C$ is a 6-gon given by $v_0-v_1-v_2-v_3-v_4-v_5-v_0$, and 
$p\in V(G)\setminus\{v_0,...,v_5\}$,  such that $N(p)\cap\{v_0,...,v_5\}=\{v_\ell,v_{\ell+3}\}$ for some $\ell\in \{0,1,2\}$.
In this section we show that if a $\{P_7,C_3,C_7\}$-free graph contains a shell, then in polynomial time we can decide if the graph is 3-colorable, and give a coloring if one exists.

We begin with some definitions. Let $C$ be a 6-gon in a graph $G$ given by $v_0-v_1-v_2-v_3-v_4-v_5-v_0$. We say that a vertex $v\in V(G)\setminus V(C)$ is:

\begin{itemize}

\item a \textit{leaf at $i$}, if $N(v)\cap V(C)=\{v_{i}\}$ for some $i\in\{0,1,...,5\}$, 

\item a \textit{clone at $i$}, if $N(v)\cap V(C)=\{v_{i-1}, v_{i+1}\}$ for some $i\in\{0,1,...,5\}$, where all indices are mod 6, 

\item a \textit{propeller at $\{i,i+3\}$}, if $N(v)\cap V(C)=\{v_{i}, v_{i+3}\}$ for some $i\in\{0,1,...,5\}$, where all indices are mod 6,

\item an \textit{even star}, if $N(v)\cap V(C)=\{v_0,v_2, v_4\}$,

\item an \textit{odd star}, if $N(v)\cap V(C)=\{v_1,v_3, v_5\}$.

\end{itemize}

\noindent The following shows how we can partition the vertices of $G$ based on their anchors in $C$.

\begin{lemma}\label{C6}

Let $G$ be a triangle-free graph, and suppose $C$ is a 6-gon in $G$ given by $v_0-v_1-v_2-v_3-v_4-v_5-v_0$. If $v\in V(G)\setminus V(C)$, then for some $i\in\{0,1,...,5\}$ either:

\begin{enumerate}

\item $v$ is a leaf at $i$,

\item $v$ is a clone at $i$,

\item $v$ is a propeller at $\{i,i+3\}$, where all indices are mod 6,

\item $v$ is an even star,

\item $v$ is an odd star, or

\item $v$ is anticomplete to $V(C)$.

\end{enumerate}
\end{lemma}

\begin{proof}

Consider a vertex $v\in V(G)\setminus V(C)$. If $v$ is anticomplete to $V(C)$, then \ref{C6}.6 holds. Thus, we may assume $N(v)\cap V(C)\ne \emptyset$. By symmetry, suppose that $v_0\in N(v)\cap V(C)$. Since $G$ is triangle-free, it follows that $v$ is anticomplete to $\{v_1,v_5\}$. Suppose $v$ is non-adjacent to $v_3$. If $v$ is anticomplete to $\{v_2,v_4\}$, then \ref{C6}.1 holds. If $v$ is mixed on $\{v_2,v_4\}$, then \ref{C6}.2 holds. If $v$ is complete to $\{v_2,v_4\}$, then \ref{C6}.4 holds. Thus, we may assume $v$ is adjacent to $v_3$. Since $G$ is triangle-free, it follows that $v$ is anticomplete to $\{v_2,v_4\}$, and so \ref{C6}.3 holds. This proves \ref{C6}.

\end{proof}

Let $G$ be a triangle-free graph and $C$ be a 6-gon in $G$ given by $v_0-v_1-v_2-v_3-v_4-v_5-v_0$. We partition $V(G)\setminus V(C)$ as follows:

\begin{itemize}

\item Let $M^C(i)$ be the set of leaves at $i$ and define $M^C=\bigcup\limits_{i=0}^5M^C(i)$.

\item Let $CL^C(i)$ be the set of clones at $i$ and define $CL^C=\bigcup\limits_{i=0}^5CL^C(i)$.

\item Let $P^C(\{i,i+3\})$ be the set of propellers at $\{i,i+3\}$ and define $P^C=\bigcup\limits_{i=0}^5P^C(\{i,i+3\})$.

\item Let $S^C_0$ be the set of even stars.

\item Let $S^C_1$ be the set of odd stars.

\item Let $S^C=S^C_0\cup S^C_1$.

\item Let $A^C$ be the set of vertices anticomplete to $V(C)$.

\end{itemize}

\bigskip

\noindent By \ref{C6}, it follows that $V(G)=V(C)\cup M^C \cup CL^C\cup P^C \cup S^C \cup A^C$. Furthermore, we partition $A^C=X^C\cup Y^C\cup Z^C$, where

\begin{itemize}
 
\item $X^C$ is the set of vertices in $A^C$ with a neighbor in $CL^C$,

\item $Y^C$ is the set of vertices in $A^C\setminus X^C$ with a neighbor in $P^C$,

\item $Z^C=A^C\setminus (X^C\cup Y^C)$.

\end{itemize}

\bigskip

\noindent And so, given a 6-gon $C$ in $G$ in time $O(|V(G)|^2)$ we obtain the partition $V(C)\cup M^C \cup CL^C\cup P^C \cup S^C \cup X^C\cup Y^C\cup Z^C$ of $V(G)$. Now, we establish several properties of this partition. The following is
immediate:

\begin{lemma}\label{6-1} 
If $G$ is a triangle-free graph, then for every 6-gon $C$ in $G$ the following hold:

\begin{enumerate}

\item Every vertex in $X^C$ has a neighbor in $CL^C$.

\item Every vertex in $Y^C$ has a neighbor in $P^C$.

\item $CL^C$ is anticomplete to $Y^C\cup Z^C$.

\end{enumerate}

\end{lemma}


\noindent Next, we show:

\begin{lemma}\label{A}

If $G$ is a $\{P_7,C_3,C_7\}$-free graph, then for every 6-gon $C$ in $G$ the following hold:

\begin{enumerate}

\item $M^C$ is anticomplete to $A^C$.

\item For every $q\in M^C\cup CL^C\cup P^C$, there exists $p_1,p_2,p_3\in V(C)$ such that $p_1-p_2-p_3-q$ is an induced path.


\item $X^C$ is stable and anticomplete to $Y^C\cup Z^C$.

\item $Z^C$ is anticomplete to $V(G)\setminus (Y^C\cup S^C)$.

\item For every $i\in\{0,...,5\}$, if $M^C(i)$ is non-empty, then $M^C(i+2)\cup M^C(i-2)$ is empty, where all indices are mod 6.

\item No vertex in $A^C$ has a neighbor in $CL^C(i)$ and $CL^C(j)$ for $i\ne j$.

\end{enumerate}
\end{lemma}

\begin{proof}

Let $C$ be a 6-gon in $G$ given by $v_0-v_1-v_2-v_3-v_4-v_5-v_0$. Suppose there exists $a\in A^C$ adjacent to $m\in M^C$. By symmetry, we may assume $m\in M^C(0)$. However, then $a-m-v_0-v_1-v_2-v_3-v_4$ is a $P_7$ in $G$, a contradiction. This proves \ref{A}.1. 

Consider a vertex $q\in M^C\cup CL^C\cup P^C$. By symmetry, we may assume $q$ is adjacent to $v_0$ and non-adjacent to $v_1,v_2$, and so $v_2-v_1-v_0-q$ is an  induced path. This proves \ref{A}.2.


Consider a vertex $x\in X^C$. By \ref{6-1}.1, there exists $c\in CL^C$ adjacent to $x$. By symmetry, we may assume $c\in CL^C(0)$. Let $C'$ be the 6-gon given by $c-v_1-v_2-v_3-v_4-v_5-c$. Suppose there exists $x'\in A^C$ adjacent to $x$. Since $G$ is triangle-free, it follows that $c$ is non-adjacent to $x'$. However, then $x\in M^{C'}$ is adjacent to $x'\in A^{C'}$, contrary to \ref{A}.1. This prove \ref{A}.3.


By \ref{6-1}.3 and \ref{A}.1, it follows that \ref{A}.4 holds. 

To prove  \ref{6-1}.5, suppose there exists $m\in M^C(0)$ and $m'\in M^C(2)$. If $m$ is non-adjacent to $m'$, then $m'-v_2-v_3-v_4-v_5-v_0-m$ is a $P_7$ in $G$, and if $m$ is adjacent to $m'$, then $m-m'-v_2-v_3-v_4-v_5-v_0-m$ is a $C_7$ in $G$, in both cases a contradiction. This proves \ref{A}.5.

We now prove \ref{A}.6. Assume $a\in A^C$ is adjacent to $c\in CL^C$. By symmetry, we may assume $c\in CL^C(0)$. Suppose there exists $c'\in CL^C\setminus CL^C(0)$ adjacent to $a$. Since $G$ is triangle-free, it follows that $c$ is non-adjacent to $c'$. By symmetry, it suffices to consider $c'\in CL^C(1)\cup CL^C(2)\cup CL^C(3)$. If $c'\in CL^C(1)$, then $a-c'-v_2-v_3-v_4-v_5-c-a$ is a $C_7$ in $G$, if $c'\in CL^C(2)$, then $v_2-v_3-c'-a-c-v_5-v_0$ is a $P_7$ in $G$, and if $c'\in CL^C(3)$, then $v_0-v_1-c-a-c'-v_4-v_3$ is a $P_7$ in $G$, in all three cases a contradiction. This proves \ref{A}.6. 
This proves \ref{A}.

\end{proof}

\bigskip


\bigskip


Let $C$ be a 6-gon in $G$ given by $v_0-v_1-v_2-v_3-v_4-v_5-v_0$. We say that a graph $G$ has a \textit{type I coloring with respect to $C$} if there exists a $3$-coloring $c$ of $G$ such that $c(v_i)=c(v_{i+3})$ for every $i\in\{0,1,2\}$.

\begin{lemma}\label{type1} There is an algorithm with the following specifications:

\bigskip

{\bf Input:} A clean, connected $\{P_7,C_3,C_7\}$-free graph $G$.

\bigskip

{\bf Output:} A type I coloring of $G$ with respect to some 6-gon in $G$, or a determination that none exists.

\bigskip

{\bf Running time:} $O(|V(G)|^{16})$.

\end{lemma}

\begin{proof}

In time $O(|V(G)|^6)$, we can enumerate all 6-gons in $G$. If $G$ is $C_6$-free, then clearly $G$ does not have a type I coloring and we may halt. Hence, we may assume there exists a 6-gon $C$ in $G$ given by $v_0-v_1-v_2-v_3-v_4-v_5-v_0$. In time $O(|V(G)|^2)$, we can partition $V(G)=V(C)\cup M^C \cup CL^C\cup P^C \cup S^C \cup X^C\cup Y^C\cup Z^C$ as usual. If $S^C$ is non-empty, then $G$ does not have a type I coloring with respect to $C$, since, by definition, the anchors of any star in $C$ receive three distinct colors in any type I coloring. Hence, we may assume $S^C$ is empty. Next, fix a $3$-coloring $c$ of $G[V(C)]$ such that $c(v_i)=c(v_{i+3})$ for every $i\in\{1,2,3\}$, where all indices are mod 6. Define the order 3 palette $L^C_c$ of $G$ as follows: For every $v\in V(G)$, set

$$L^C_c(v) = \begin{cases}
 \{c(v)\} & \text{, \hspace{2ex}if $v\in V(C)$} \\
 \{1,2,3\} & \text{, \hspace{2ex}otherwise}\\
\end{cases}$$

\noindent Next, update the vertices in $M^C\cup CL^C\cup P^C$ with respect to $V(C)$. And so, $|L^C_c(v)| \leq 2$ for all $v \in V(G) \setminus (X^C\cup Y^C\cup Z^C)$, while $|L^C_c(v)|=3$ for all $v\in X^C\cup Y^C\cup Z^C$. Additionally, $|L^C_c(v)|=2$ if and only if $v\in M^C\cup P^C$. Observe that, by construction, $(G,L^C_c)$ is colorable if and only the $3$-coloring $c$ of $G[V(C)]$ extends to a type I coloring of $G$. 

By \ref{6-1}, \ref{A}.1, \ref{A}.2 and \ref{A}.4, we may apply \ref{3edge} letting $P=V(C)\cup M^C$, $Q=CL^C\cup P^C$, $R=X^C\cup Y^C$, $S=Z^C$, and $T=\emptyset$. It follows that $Z^C$ is empty and that every component of $X^C\cup Y^C$ has size at most two. Let $A'=CL^C\cup P^C$, and for every non-empty subset $X \subseteq \{1,2,3\}$, define $A'_X=\{a\in A'$ with $L^C_c(a)=X\}$.
Since for $v \in A'$,  $|L^C_c(v)|=2$ if and only if 
$v \in P^C$, it follows that if $|X|=2$, then $A_X=P^C(\{i,i+3\})$ for some 
$i \in \{0,1,2\}$.

\bigskip

\noindent \textit{(1) For every $c_1,c_2 \in X^C\cup Y^C$ and $\{i,j,k\}= \{1,2,3\}$, $(N(c_1)\cap A'_{\{i,j\}})\setminus N(c_2)$ is complete to $(N(c_2)\cap A'_{\{i,k\}})\setminus N(c_1)$.}

\bigskip

\noindent We may assume $A'_{\{i,j\}}=P^C(\{0,3\})$ and $A'_{\{i,k\}}=P^C(\{1,4\})$. Suppose there exists $p_1\in P^C(\{0,3\})\setminus N(c_2)$ adjacent to $c_1$ and $p_2\in P^C(\{1,4\})\setminus N(c_1)$ adjacent to $c_2$ such that $p_1$ is non-adjacent to $p_2$. If $c_1$ is non-adjacent to $c_2$, then $c_2-p_2-v_1-v_2-v_3-p_1-c_1$ is a $P_7$ in $G$, and if $c_1$ is adjacent to $c_2$, then $c_1-c_2-p_2-v_1-v_2-v_3-p_1-c_1$ is a $C_7$ in $G$, in both cases a contradiction. Hence, it follows that $p_1$ is adjacent to $p_2$. This proves \textit{(1)}.


\bigskip

\noindent \textit{(2) For all distinct $i,j \in \{1,2,3\}$ some vertex of $V(C)\cup M^C$ is complete to $A'_{\{i,j\}}$.}

\bigskip

\noindent After updating, it follows that $|L^C_c(v)|=2$ if and only if $v\in M^C\cup P^C$. By symmetry, we may assume $A'_{\{i,j\}}=P^C(\{0,3\})$. Hence, $\{v_0,v_3\}$ is complete to $\tilde{A}_{\{i,j\}}$. This proves \textit{(2)}.

\bigskip

\noindent By \textit{(1)} and \textit{(2)}, we can apply \ref{lemma2} with 

\begin{itemize}

\item $\tilde{A}=CL^C\cup P^C$, 
\item $\tilde{B}=V(C) \cup M^C$, 
\item $\tilde{C}=X^C\cup Y^C$, and
\item $Z=\emptyset$.

\end{itemize}

\noindent Let $\mathcal{P}^C_c$ be the restriction of  $(G,L^C_c, \emptyset)$ of size
$O(|V(G)|^7)$ thus obtained. By \ref{lemma2}, $\mathcal{P}^C_c$ can be computed in time $O(|V(G)|^7)$. By \ref{lemma2}\textit{(c)}, we have that $(G,L^C_c, \emptyset)$  (and equivalently $(G,L^C_c)$) is colorable if and only if $\mathcal{P}^C_c$ is colorable. Consider 
$(G',L',X')\in \mathcal{P}^C_c$. Since $|L^C_c(v)| \leq 2$ for all $v \in V(G) \setminus (X^C\cup Y^C)$, by \ref{lemma2}\textit{(b)}, it follows that $|L'(v)|\leq 2$ for all $v\in V(G')$. Since $|X'|$ has size $O(|V(G)|)$,
applying \ref{checkSubsets}, it follows that 
we can determine in time $O(|V(G)|^3)$ if $(G',L',X')$ is colorable, and 
if it is, extend the coloring to a coloring of $G$. Therefore,
via $O(|V(G)|^7)$ applications of \ref{checkSubsets}, we can determine if $\mathcal{P}^C_c$ is colorable. That is, in time $O(|V(G)|^{10})$ we can determine if the $3$-coloring $c$ of $G[V(C)]$ extends to a type I coloring of $G$, and give an explicit type I coloring $c'$ of $G$ such that $c'(v)=c(v)$ for all $v\in V(C)$, if one exists. Finally, let $\mathcal{P}$ be the union of $\mathcal{P}^C_c$ taken over all 6-gons $C$ in $G$ given by $v_0-v_1-v_2-v_3-v_4-v_5-v_0$ with $S^C$ empty and all $3$-colorings $c$ of $G[V(C)]$ such that $c(v_i)=c(v_{i+3})$ for every $i\in\{1,2,3\}$, where all indices are mod 6.  Since every type I coloring of $G$ restricts to such a coloring of $G[V(C)]$, it follows that $G$ has a type I coloring if and only if $\mathcal{P}$ is colorable.
Since there are $O(|V(G)|^6)$ 6-gons in $G$ and $3!$ such colorings of $G[V(C)]$, it follows that $\mathcal{P}$  consists of $O(|V(G)|^6)$ restrictions $\mathcal{P}^C_c$, and so by the previous argument, we can determine in time $O(|V(G)|^{16})$ if $G$ admits a type I coloring, and construct such a coloring if one exists. This proves \ref{type1}.
\end{proof}

Let $C$ be a 6-gon in $G$ given by $v_0-v_1-v_2-v_3-v_4-v_5-v_0$. We say that a graph $G$ has a \textit{type II coloring with respect to $C$} (or just a 
{\em type II coloring} when the details are not important) if there exist a $3$-coloring $c$ of $G$ such that  $c(p_1) \neq c(p_2)$ for some  $p_1,p_2 \in P^C(0,3)$.

\begin{lemma}\label{type2} There is an algorithm with the following specifications:

\bigskip

{\bf Input:} A clean, connected $\{P_7,C_3,C_7\}$-free graph $G$, that does not admit a type I coloring.

\bigskip

{\bf Output:} A type II coloring of $G$ with respect to some 6-gon in $G$, or a determination that none exists.

\bigskip

{\bf Running time:} $O(|V(G)|^{18})$.

\end{lemma}

\begin{proof}

In time $O(|V(G)|^8)$, we can enumerate all triples $(C,p_1,p_2)$ in $G$, 
where $C$ is a 6-gon given by $v_0-v_1-v_2-v_3-v_4-v_5-v_0$, and 
$p_1, p_2 \in P^C(0,3)$. If $G$ has no such triple, then clearly $G$ does not have a type II coloring and we may halt. Hence, we may assume there exists a 6-gon $C$ in $G$ given by $v_0-v_1-v_2-v_3-v_4-v_5-v_0$, and $p_1, p_2 \in P^C(0,3)$. In time $O(|V(G)|^2)$, we can partition $V(G)=V(C)\cup M^C \cup CL^C\cup P^C \cup S^C \cup X^C\cup Y^C\cup Z^C$ as usual. Write $D=V(C) \cup \{p_1,p_2\}$. 
Fix a $3$-coloring $c$ of $G[D]$ such that $c(p_1) \neq c(p_2)$. Since $G$ does not admit a type I coloring,
we may assume that $c(v_0)=c(v_3)=1, c(v_1)=c(v_5)=2, c(v_2)=c(v_4)=3, c(p_1)=2$,
and $c(p_2)=3$.  Define the order 3 palette $L^C_c$ of $G$ as follows: For every $v\in V(G)$, set

$$L^C_c(v) = \begin{cases}
 \{c(v)\} & \text{, \hspace{2ex}if $v\in D$} \\
 \{1,2,3\} & \text{, \hspace{2ex}otherwise} \\
\end{cases}$$

\noindent 
Next, update the vertices in $M^C\cup CL^C\cup (P^C \setminus \{p_1,p_2\})  \cup S^C$ with respect to 
$D$. And so, $|L^C_c(v)| \leq 2$ if and only if  $v \in V(G) \setminus (X^C\cup Y^C\cup Z^C)$. Moreover, for $v \in V(G) \setminus (X^C\cup Y^C\cup Z^C)$, 
$|L(v)|=2$ only if $v \in M^C \cup CL^C(0) \cup CL^C(3) \cup P^C(0,3)$.
By construction, $(G,L^C_c)$ is colorable if and only the $3$-coloring $c$ of $G[D]$ extends to a type II coloring of $G$.

Observe that for every $v \in S^C(0)$ and $i \in \{1,2\}$, $v-v_2-v_3-p_i$
is an induced path in $G$, and for every $v \in S^C(1)$
and  $i \in \{1,2\}$, $v-v_1-v_0-p_i$ is an induced path in $G$. 
Let $W^C$ be the vertices of $Z^C$ with a neighbor in $S^C$. Now by 
\ref{A}.1 and \ref{A}.2, we may apply \ref{3edge} letting 
$P=V(C)$, $Q=M^C \cup CL^C \cup P^C \cup S^C$, $R=X^C \cup Y^C \cup W^C$, 
$S=Z^C \setminus W^C$, $T=\emptyset$, and $q_0=p_1$. 
It follows that  $Z^C \setminus W^C=\emptyset$, that every component of $R$ is bipartite, 
and if some component of $R$ has more than two vertices, then $p_1$ is
complete to at least one side of the bipartition.  Symmetrically,
 if some component of $R$ has more than two vertices, then $p_2$ is
complete to at least one side of the bipartition. 

Let $R_1$ be the union of the components of $R$ that contain a vertex 
complete to $\{p_1,p_2\}$. For every $v \in R_1$ that is 
complete to $\{p_1,p_2\}$, set $L^C_c(v)=\{1\}$. 
For every component $X$ of $R_1$ with $|X|>1$, proceed as follows. Let 
$(A,B)$ be a  bipartition of $G[X]$. By \ref{3edge} and the definition of
$R_1$, it follows that one of $A,B$ is complete to $\{p_1,p_2\}$.
We may assume that  $A$ is complete to $\{p_1,p_2\}$. Therefore
$L^C_c(a)=\{1\}$ for every $a \in A$. Now set
$L^C_c(b)= \{2,3\}$ for every $b \in B$. Note that this does not change the 
colorability of $(G,L^C_c)$. Observe that at this stage $|L^C_c(v)| < 3$ for
every $v \in R_1$, and $|L^C_c(v)|=3$ for every $v \in R \setminus R_1$.

Let $R_2$ be the union of all components  $Y$ of $R \setminus R_1$ such that
$Y=\{y\}$ and  $y$ has a neighbor in $\{p_1,p_2\}$.
For every $v \in R_2$, update the list of $v$ with respect to
$\{p_1,p_2\}$. 

\bigskip

\noindent \textit{(1) Let $v \in CL^C(0) \cup CL^C(3)$ be adjacent to 
$y \in R$. Then each of $p_1, p_2$ has a neighbor in $\{v,y\}$.}

\bigskip

\noindent Suppose not. We may assume that $v \in CL^C(0)$. If
$p_1$ is anticomplete to $\{v,y\}$, then 
$y-v-v_1-v_0-p_1-v_3-v_4$ is a $P_7$ in $G$. This proves that
either $v$ or $y$ is adjacent to $p_1$.  Similarly, either $v$ or $y$ is 
adjacent to $p_2$. This proves \textit{(1)}.
\bigskip

Let $C'=R  \setminus (R_1 \cup R_2)$. Then $|L_c^C(y)|=3$ for every $y \in C'$.
Moreover, no vertex of $C'$ is complete to $\{p_1,p_2\}$, and if $Y$ is a 
component of $C'$ with $|Y|=1$, then $Y$ is anticomplete to $\{p_1,p_2\}$.
Let $A'$ be the set of vertices of $M^C \cup CL^c \cup P^C \cup S^C$ with a 
neighbor in $C'$.
For every non-empty subset $X \subseteq \{1,2,3\}$, define $A'_X=\{a\in A'$ with $L^C_c(a)=X\}$. 

Suppose that  $v\in A' \cap (CL^C(0) \cup CL^C(3))$ has a neighbor $y$  in $C'$.
By~\ref{A}.3, $\{y\}$ is a component of $R$. It follows from the
definition of $C'$ that $y$ is anticomplete to $\{p_1,p_2\}$. Now
(1) implies that  $v$
is complete to $\{p_1,p_2\}$, and, in particular, $L^C_c(v)=\{1\}$. 
Consequently, for $v \in A'$,  $|L^C_c(v)|=2$ if and only if 
$v \in P^C(0,3)$ and $L^C_c(v)=\{2,3\}$.  Thus 
$A'_{\{1,2\}}=A'_{\{1,2\}}=\emptyset$, and $v_0$ is complete to $A_{\{2,3\}}$.

We apply \ref{lemma2} with 

\begin{itemize}

\item $\tilde{A}=A'$, 
\item $\tilde{B}=V(G) \setminus (A' \cup C')$, 
\item $\tilde{C}=C'$, and
\item $Z = \emptyset$.

\end{itemize}

\noindent Let $\mathcal{P}^{C,p_1,p_2}_c$ be the restriction of   $(G,L^C_c, \emptyset)$ of size 
$O(|V(G)|^7)$ thus obtained. By \ref{lemma2}, $\mathcal{P}^{C,p_1,p_2}_c$ can be computed in time $O(|V(G)|^7)$. By \ref{lemma2}\textit{(c)}, we have that $(G,L^C_c, \emptyset)$ (and equivalently $(G,L^C_c)$)  is colorable if and only if $\mathcal{P}^{C,p_1,p_2}_c$ is colorable. Consider 
$(G',L',X') \in \mathcal{P}^{C,p_1,p_2}_c$. Since $|L^C_c(v)| \leq 2$ for all 
$v \in V(G) \setminus V(C')$, by \ref{lemma2}\textit{(b)}, it follows that $|L'(v)|\leq 2$ for all $v\in V(G')$. Since $X'$ has size $O(|V(G)|)$, applying \ref{checkSubsets}, it follows that we can test in time $O(|V(G)|^3)$ if  $(G',L',X')$ is colorable, and extend the coloring to $(G, L^C_c)$ if it is. Therefore,
via $O(|V(G)|^7)$ applications of \ref{checkSubsets}, we can determine 
in time $O(|V(G)|^{10})$ if the $3$-coloring $c$ of $G[D]$ extends to a type II coloring of $G$, and give an explicit type 
II coloring $c'$ of $G$ such that $c'(v)=c(v)$ for all $v\in D$, if one exists. Finally, let $\mathcal{P}$ be the union of $\mathcal{P}^{C,p_1,p_2}_c$ taken over all 
triples $(C,p_1,p_2)$ where $C$ is a 
6-gon given by $v_0-v_1-v_2-v_3-v_4-v_5-v_0$, and $p_1,p_2 \in P^C(0,3)$, and
 all $3$-colorings $c$ of $G[V(C) \cup \{p_1,p_2\}]$ such that $c(p_1)\neq c(p_2)$. Since every type II coloring of $G$ restricts to such a coloring of 
$G[V(C) \cup \{p_1,p_2\}]$ for some $C,p_1,p_2$ , it follows that $G$ has a type II coloring if and only if $\mathcal{P}$ is colorable.
Since there are $O(|V(G)|^8)$ such triples $(C,p_1,p_2)$ in $G$ and $2$ such colorings of 
$G[V(C) \cup \{p_1,p_2\}]$ for each $(C,p_1,p_2)$,   the restriction $\mathcal{P}$ is the union  of $O(|V(G)|^8)$ restrictions $\mathcal{P}^{C,p_1,p_2}_c$. Therefore by the previous argument, we can determine in time $O(|V(G)|^{18})$ if $G$ admits a type II coloring, and construct such a coloring if one exists. This proves \ref{type2}.
\end{proof}

\noindent Now, suppose $(C,p)$ is a shell in $G$. 
We partition $V(G)\setminus (V(C)\cup\{p\})$ as follows:



\begin{itemize}

\item Let $Q^C_p$ be the set of vertices in $V(G)\setminus (V(C) \cup \{p\})$ with a neighbor in $V(C)\cup \{p\}$.

\item Let $R^C_p$ be the set of vertices in $V(G)\setminus (V(C)\cup\{p\}\cup Q^C_p)$ with a neighbor in $Q^C_p$.

\item Let $S^C_p=V(G)\setminus (V(C)\cup\{p\}\cup Q^C_p\cup R^C_p)$.

\item Let $PL^C_p$ be the set of vertices in $V(G)\setminus V(C)$ adjacent to $p$ and anticomplete to $V(C)$. Note, that $PL^C_p$ is a subset of $Q^C_p$.

\end{itemize}

\begin{lemma}\label{B}

Let $G$ be a clean, connected $\{P_7,C_3,C_7\}$-free graph. Then for every shell $(C,p)$ in $G$ the following hold:

\begin{enumerate}

\item $M^C\cup CL^C\cup P^C\cup S^C\cup PL^C_p$ gives a partition of $Q^C_p$.

\item Every vertex of $Q^C_p$ with a neighbor in $R^C_p$ either belongs to $PL^C_p$ or has at least two neighbors in $V(C)$.

\item $S^C_p$ is empty.

\item Every component of $R^C_p$ has size at most two.

\end{enumerate}
\end{lemma}

\begin{proof}

%
%
%
%
%

Let $(C,p)$ be a shell in $G$, where $C$ is the 6-gon in $G$ given by $v_0-v_1-v_2-v_3-v_4-v_5-v_0$ and $p\in P^C$. 
By \ref{C6}, it follows that \ref{B}.1 holds. Since $R^C_p$ is a subset of $A^C$, by \ref{A}.1, it follows that $M^C$ is anticomplete to $R^C_p$. And so, by definition, \ref{C6}, and \ref{B}.1, it follows that \ref{B}.2 holds.

\bigskip

\noindent \textit{(1) For every $s\in S^C$, there exists $p_1,p_2,p_3\in V(C)\cup \{p\}$ such that $p_1-p_2-p_3-s$ is an induced path.}

\bigskip

 \noindent Consider a vertex $s\in S^C$. By symmetry, we may assume both $s$ and $p$ are adjacent to $v_0$, that is, that $s\in S^C_0$ and $p\in P^C(\{0,3\})$. Since $G$ is triangle-free, it follows that $s$ is non-adjacent to $p$. Then $v_3-p-v_0-s$ is the desired induced path. This proves \textit{(1)}.
 
\bigskip

\noindent By definition, $S^C_p$ is anticomplete to $V(C)\cup\{p\}\cup Q^C_p$. Since $G$ is clean and connected, by \ref{A}.2 and \textit{(1)}, we may apply \ref{3edge} letting $P=V(C)\cup \{p\}$, $Q=Q^C_p$, $R=R^C_p$, $S=S^C_p$ and $T=\emptyset$. It follows that \ref{B}.3 and \ref{B}.4 hold. This proves \ref{B}.

\end{proof}

\begin{lemma}\label{shell} There is an algorithm with the following specifications:

\bigskip

{\bf Input:} A clean, connected $\{P_7,C_3,C_7\}$-free graph $G$ which contains a shell.

\bigskip

{\bf Output:} A $3$-coloring of $G$, or a determination that none exists.

\bigskip

{\bf Running time:} $O(|V(G)|^{18})$.

\end{lemma}

\begin{proof}



By \ref{type1} and \ref{type2} in time $O(|V(G)|^{18})$ we can produce a type I or a type  II coloring of $G$, if one exists. Hence, we may assume there does not exists a type I or a type II coloring of $G$. Let $C$ be the 6-gon in $G$ given by $v_0-v_1-v_2-v_3-v_4-v_5-v_0$, and suppose $(C,p)$ is a shell in $G$. Observe that such an induced subgraph can be found in time $O(|V(G)|^7)$. Since $G$ is clean, by \ref{B}.3, it follows that $S^C_p$ is empty, and so we may partition $V(G)=V(C)\cup \{p\} \cup Q^C_p\cup R^C_p$ as usual. Next, fix a $3$-coloring $c$ of $G[V(C)\cup\{p\}]$, that is not a type I coloring with respect to
$C$. Define the order 3 palette $L^C_c$ of $G$ as follows: For every $v\in V(G)$, set

$$L^C_c(v) = \begin{cases}
 \{c(v)\} & \text{, \hspace{2ex}if $v\in V(C)\cup\{p\}$} \\
 \{c(p)\} & \text{, \hspace{2ex}if $v \not \in V(C)\cup\{p\}$, and $v$ has the 
same anchors as $p$ in $C$}\\
\{1,2,3\} & \text{, \hspace{2ex}otherwise}\\
\end{cases}$$

\noindent Next, update the vertices in $Q^C_p$ with respect to $V(C)\cup \{p\}$. And so, $|L^C_c(v)| \leq 2$ for all $v \in V(G) \setminus R^C_p$, while $|L^C_c(v)|=3$ for all $v\in R^C_p$. Observe that, 
since $G$ does not have a type II coloring, $(G,L^C_c)$ is colorable if and only if the $3$-coloring $c$ of $G[V(C)\cup\{p\}]$ extends to a $3$-coloring of $G$. 

Let $A'$ be the set of vertices in $Q^C_p$ with a neighbor in $R^C_p$, and for every non-empty subset $X \subseteq \{1,2,3\}$, define $A'_X=\{a\in A'$ with $L^C_c(a)=X\}$.

\bigskip

\noindent \textit{(1) Let $\{i,j,k\}= \{1,2,3\}$. If $q_1\in A'_{\{i,j\}}$, $q_2\in A'_{\{j,k\}}$, and $r_1\in R^C_p\cap (N(q_1)\setminus N(q_2))$ and $r_2\in R^C_p\cap (N(q_2)\setminus N(q_1))$, then $q_1$ is adjacent to $q_2$.}

\bigskip

\noindent Suppose $q_1$ is non-adjacent to $q_2$. Since $L(q_1)=\{i,j\}$ and
$L(q_2)=\{j,k\}$, it follows that $N(q_1)\cap N(q_2)\cap (V(C)\cup\{p\})$ is empty. First, assume $q_1\in S^C$. By symmetry, we may assume $q_1\in S^C_0$
and  $p\in P^C(\{0,3\})$. Since $G$ is triangle-free, it follows that $q_1$ is non-adjacent to $p$. Suppose $q_2\in PL^C_p$.  However, if $r_1$ is non-adjacent to $r_2$, then $r_1-q_1-v_2-v_3-p-q_2-r_2$ is a $P_7$ in $G$, and if $r_1$ is adjacent to $r_2$, then $r_1-q_1-v_2-v_3-p-q_2-r_2-r_1$ is a $C_7$ in $G$, in both case a contradiction. Hence, $q_2\notin PL^C_p$. By \ref{B}.2, it follows that $|N(q_2)\cap \{v_1,v_3,v_5\}|\geq 2$. Suppose $q_2$ is adjacent to $v_3$. Since $G$ is triangle-free, it follows that $p$ is non-adjacent to $q_2$. However, if $r_1$ is non-adjacent to $r_2$, then $r_1-q_1-v_0-p-v_3-q_2-r_2$ is a $P_7$ in $G$, and if $r_1$ is adjacent to $r_2$, then $r_1-q_1-v_0-p-v_3-q_2-r_2-r_1$ is a $C_7$ in $G$, in both case a contradiction. Hence, $q_2$ is non-adjacent to $v_3$, and so $q_2$ is complete to $\{v_1,v_5\}$. If $p$ is non-adjacent to $q_2$, then $v_4-v_3-p-v_0-v_1-q_2-r_2$ is a $P_7$ in $G$, a contradiction. Hence, $p$ is adjacent to $q_2$. However, if $r_1$ is non-adjacent to $r_2$, then $r_1-q_1-v_2-v_3-p-q_2-r_2$ is a $P_7$ in $G$, and if $r_1$ is adjacent to $r_2$, then $r_1-q_1-v_2-v_3-p-q_2-r_2-r_1$ is a $C_7$ in $G$, in both case a contradiction. By symmetry, this proves that neither $q_1$ nor $q_2$ belongs to $S^C$. 

Next, suppose $q_1\in P^C$. By symmetry, we may assume $q_1\in P^C(\{0,3\})$.
Suppose first that $q_2$ is adjacent to $v_1$. Since $G$ is triangle-free, it follows that $q_2$ is non-adjacent to $v_2$. However, if $r_1$ is non-adjacent to $r_2$, then $r_1-q_1-v_3-v_2-v_1-q_2-r_2$ is a $P_7$ in $G$, and if $r_1$ is adjacent to $r_2$, then $r_1-q_1-v_3-v_2-v_1-q_2-r_2-r_1$ is a $C_7$ in $G$, in both case a contradiction. By symmetry, it follows that $q_2$ is anticomplete to
$\{v_1,v_2,v_4,v_5\}$. Since  $q_1\in A'_{\{i,j\}}$ and $q_2\in A'_{\{j,k\}}$,
it follows that $q_2 \in PL^C_p$, $c(p) \neq c(v_0)$, and $p$ is
non-adjacent to $q_1$. Since  $|L^C_c(q_1)|=2$, it follows that 
$p  \not \in P^C(\{0,3\})$, and so we may assume that $p \in P^C(\{1,4\})$.
Now, if $r_1$ is non-adjacent to $r_2$, then
$r_2-q_2-p-v_1-v_0-q_1-r_1$ is a $P_7$ in $G$, and if  if $r_1$ is adjacent to 
$r_2$, then $r_2-q_2-p-v_1-v_0-q_1-r_1-r_2$ is a $C_7$ in $G$, in both cases a 
contradiction.
By symmetry, this proves that neither $q_1$ nor $q_2$ belongs to $P^C$.

Since not both $q_1$ and $q_2$ are adjacent to $p$, by \ref{B}.1 and \ref{B}.2, we may assume $q_1\in CL^C$ is non-adjacent to $p$. By symmetry, we may assume $q_1\in CL^C(1)$. Since $r_1-q_1-v_0-v_1-p-v_4-v_3$ is not a $P_7$ in $G$, it follows that $p\notin P^C(\{1,4\})$. And so, we may assume $p\in P^C(\{0,3\})$. Suppose $q_2\in CL^C$ also. Since $N(q_1)\cap N(q_2)\cap V(C)$ is empty, we may assume that $q_2\in CL^C(0)\cup CL^C(4)$. Suppose $q_2\in CL^C(0)$. Let $C'$ be the 6-gon in $G$ given by $q_2-v_1-v_2-v_3-v_4-v_5-q_2$. Then $r_2\in M^{C'}(0)$ and $q_1\in M^{C'}(2)$, contrary to \ref{B}.2. Hence, $q_2\in CL^C(4)$. However, if $r_1$ is non-adjacent to $r_2$, then $r_1-q_1-v_0-p-v_3-q_2-r_2$ is a $P_7$ in $G$, and if $r_1$ is adjacent to $r_2$, then $r_1-q_1-v_0-p-v_3-q_2-r_2-r_1$ is a $C_7$ in $G$, in both case a contradiction. 
Hence, $q_2\notin CL^C$. By \ref{B}.1 and \ref{B}.2, it follows $q_2\in PL^C_p$. 
However, if $r_1$ is non-adjacent to $r_2$, then $r_1-q_1-v_2-v_3-p-q_2-r_2$ is a $P_7$ in $G$, and if $r_1$ is adjacent to $r_2$, then $r_1-q_1-v_2-v_3-p-q_2-r_2-r_1$ is a $C_7$ in $G$, in both case a contradiction. This proves \textit{(1)}.

\bigskip

\noindent \textit{(2) For all distinct $i,j \in \{1,2,3\}$ some vertex of $V(C)\cup \{p\}$ is complete to $A'_{\{i,j\}}$.}

\bigskip

\noindent If $A'_{\{i,j\}}=\emptyset$, then \textit{(2)} trivially holds. Thus, we may assume $A'_{\{i,j\}}\ne\emptyset$. Let $\{i,j,k\}=\{1,2,3\}$ and define $K$ to be the set of vertices of $V(C)\cup \{p\}$ with a neighbor in $A'_{\{i,j\}}$. Since we have updated, it follows that $c(v)=k$ for every $v\in K$. Since $G[V(C)\cup \{p\}]$ has no stable set of size 4, it follows that $|K|\leq 3$. If $|K|=1$, then, by definition, the unique vertex of $K$ is complete to $A'_{\{i,j\}}$. Hence, we may assume that $|K|\geq 2$. If $|K|=2$, then, by \ref{B}.2, either $p\in K$ is complete to $A'_{\{i,j\}}$, or $p\notin K$ and then $K$ is complete to $A'_{\{i,j\}}$. In either case, \textit{(2)} holds. And so we may assume $|K|=3$. By \ref{B}.1 and \ref{B}.2, it follows that $A'_{\{i,j\}}\subseteq CL^C\cup P^C\cup S^C\cup PL^C_p$. First, suppose $p\notin K$. It follows that $PL^C_p\cap A'_{\{i,j\}}$ is empty. By symmetry, we may assume $K=\{v_0,v_2,v_4\}$. Suppose further that $v_0$ and $v_2$ are not complete to $A'_{\{i,j\}}$. By \textit{(1)}, it follows that there exists $c_3\in A'_{\{i,j\}}\cap CL^C(3)$, and $c_5\in A'_{\{i,j\}}\cap CL^C(5)$. Since $G$ is triangle-free, it follows that $c_3$ is non-adjacent to $c_5$. By definition, there exists $r_3,r_5\in R^C_p$ such that $c_3$ is adjacent to $r_3$, and $c_5$ is adjacent to $r_5$. By \ref{A}.6, it follows that $c_3$ is non-adjacent to $r_5$, and $c_5$ is non-adjacent to $r_3$. Let $C'$ be the 6-gon in $G$ given by $v_0-v_1-v_2-c_3-v_4-c_5-v_0$. Then $r_3\in M^{C'}(3)$ and $r_5\in M^{C'}(5)$, contrary to \ref{A}.5. And so, it follows that $p\in K$. 

Without loss of generality, we may assume $p\in P^C(\{0,3\})$. Let $\{a,b\}=K\cap V(C)$. By symmetry, we may assume $a=v_1$ and $b\in \{v_4,v_5\}$. We may also assume $p$ is not complete to $A'_{\{i,j\}}$, as otherwise \textit{(2)} holds immediately. By \ref{B}.1 and \ref{B}.2, it follows that there exists $q\in A'_{\{i,j\}}$ complete to $\{a,b\}$ and non-adjacent to $p$. We may also assume that $a$ is not complete to $A'_{\{i,j\}}$, as otherwise \textit{(2)} holds immediately. And so, by \ref{B}.2, there exists $q'\in PL^C_p\cap A'_{\{i,j\}}$. By definition, there exists $r,r'\in R^C_p$ such that $r$ is adjacent to $q$, and $r'$ is adjacent to $q'$. Suppose $r$ is adjacent to $q'$. Since $G$ is triangle-free, it follows that $q$ is non-adjacent to $q'$. But now
$v_2-v_1-q-r-q'-p-v_3-v_2$ is a $C_7$ in $G$, a contradiction.
Hence, it follows that $r$ is non-adjacent to $q'$, and, by symmetry, that $r'$ is non-adjacent to $q$. Suppose $q$ is non-adjacent to $q'$. If $r$ is non-adjacent to $r'$, then $r-q-v_1-v_0-p-q'-r'$ is a $P_7$ in $G$, and if $r$ is adjacent to $r'$, then $r-q-v_1-v_0-p-q'-r'-r$ is a $C_7$ in $G$, in both cases a contradiction. Hence, $q$ is adjacent to $q'$. Let $C'''$ be the 6-gon in $G$ given by $v_1-v_2-v_3-p-q'-q-v_1$. If $q$ is adjacent to $v_5$ (and therefore not to $v_4$), then $r'-q'-q-v_1-v_2-v_3-v_4$ is a $P_7$ in $G$, a contradiction.
Hence, it follows that $q$ is non-adjacent to $v_5$, that is, that $b=v_4$. 
Since $c$ is not a type I coloring with respect to $C$, and 
since $c(v_1)=c(v_4)=c(p)$, it follows
that $c(v_0)=c(v_2)$, and $c(v_3)=c(v_5)$. Applying the fact that $G$ admits
no type I coloring to the 6-gon
$v_1-v_2-v_3-p-q'-q-v_1$, we deduce that in every coloring $c'$ of $(G,L^C_c)$,
$c'(q)=c(v_2)$  and  $c'(q')=c(v_3)$. However, applying the fact that
$G$ admits no type I coloring with respect to the 6-gon $v_4-v_5-v_0-p-q'-q-v_4$, 
we deduce that in every coloring $c'$ of $(G,L^C_c)$,
$c'(q)=c(v_5)$  and  $c'(q')=c(v_0)$. But this implies that 
$c(v_0)=c(v_2)=c(v_5)$, a contradiction. This proves \textit{(2)}.

\bigskip

\noindent By \ref{B}.4, it follows that every component of $R^C_p$ has at most two vertices. And so, by \ref{B}.1, \ref{B}.2, \textit{(1)} and \textit{(2)}, we can apply \ref{lemma2} with 

\begin{itemize}

\item $\tilde{A}=A'$ 
\item $\tilde{B}=V(C) \cup \{p\}\cup (Q^C_p\setminus A')$, 
\item $\tilde{C}=R^C_p$, and
\item $Z=\emptyset$.
\end{itemize}

\noindent Let $\mathcal{P}_c$ be the restriction of  $(G,L^C_c)$ of size
$O(|V(G)|^7)$ thus obtained, and let  $\mathcal{P}$ be the union of $\mathcal{P}_c$ taken over all $3$-colorings $c$ of $G[V(C) \cup \{p\}]$ that are not type I colorings. By \ref{lemma2}, and since there are at most $7^3$ $3$-colorings of $G[V(C)\cup\{p\}]$, it follows that $\mathcal{P}$ can be computed in time $O(|V(G)|^7)$. By \ref{lemma2}\textit{(c)}, we have that $(G,L^C_c)$ is colorable if and only if $\mathcal{P}_c$ is colorable. 
Since every $3$-coloring of $G$ restricts to a $3$-coloring of $G[V(C) \cup \{p\}]$, it follows that $G$ is 3-colorable if and only if $\mathcal{P}$ is colorable.

Consider  $(G',L',X')\in \mathcal{P}$. Then $(G',L',X')\in \mathcal{P}_c$ 
for some coloring $c$ of $G[V(C)\cup\{p\}]$, and $c$ is not a type I or a type 
II coloring.
Since $|L^C_c(v)| \leq 2$ for all $v \in V(G) \setminus R^C_p$, by \ref{lemma2}\textit{(b)}, it follows that $|L'(v)|\leq 2$ for all $v\in V(G')$. Thus, 
since by \ref{lemma2} $X'$ has size $O(|V(G)|)$, applying \ref{checkSubsets},  
we can test in time $O(|V(G)|^3)$ if $(G',L',X')$ is colorable.
Therefore, via $O(|V(G)|^7)$ applications of \ref{checkSubsets}, we can determine if $\mathcal{P}$ is colorable and extend any coloring of a colorable $(G',L', X') \in \mathcal{P}$ to a coloring of $G$ in linear time. 
Consequently,  in time $O(|V(G)|^{10})$ we can determine if $\mathcal{P}$ is 
colorable. This proves \ref{shell}.

\end{proof}

\section{5-gons}

In this section we show that if a $\{P_7,C_3,C_7,shell\}$-free graph contains a 5-gon, then in polynomial time we can decide if the graph is 3-colorable, and give a coloring if one exists.

Let $C$ be a 5-gon in a graph $G$ given by $v_0-v_1-v_2-v_3-v_4-v_0$. We say that a vertex $v\in V(G)\setminus V(C)$ is:

\begin{itemize}

\item a \textit{leaf at $i$}, if $N(v)\cap V(C)=\{v_{i}\}$ for some $i\in\{0,1,...,4\}$,

\item a \textit{clone at $i$}, if $N(v)\cap V(C)=\{v_{i-1}, v_{i+1}\}$ for some $i\in\{0,1,...,4\}$, where all indices are mod 5.

\end{itemize}

\noindent The following shows how we can partition the vertices of $G$ based on their anchors in $C$.

\begin{lemma}\label{C5}

Let $G$ be a triangle-free graph, and suppose $C$ is a 5-gon in $G$ given by $v_0-v_1-v_2-v_3-v_4-v_0$. If $v\in V(G)\setminus V(C)$, then for some $i\in\{0,1,...,4\}$, either:

\begin{enumerate}
\item $v$ is a leaf at $i$,

\item $v$ is a clone at $i$, or

\item $v$ is anticomplete to $V(C)$.

\end{enumerate}
\end{lemma}

\begin{proof}
Consider a vertex $v\in V(G)\setminus V(C)$. If $v$ is anticomplete to $V(C)$, then \ref{C5}.3 holds. Thus, we may assume $N(v)\cap V(C)\ne \emptyset$, and, by symmetry, suppose that $v_0\in N(v)\cap V(C)$. If $|N(v)\cap V(C)|=1$, then \ref{C5}.1 holds, and so we may assume $|N(v)\cap V(C)|\geq 2$. Since $G$ is triangle-free, it follows that $v$ is anticomplete to $\{v_1,v_4\}$. Since $G$ is triangle-free, it follows that $v$ is mixed on $\{v_2,v_3\}$ and so \ref{C5}.2 holds. This proves \ref{C5}.
\end{proof}


Let $G$ be a triangle-free graph. Suppose $C$ is a 5-gon in $G$ given by $v_0-v_1-v_2-v_3-v_4-v_0$. Using \ref{C5} we partition $V(G)\setminus V(C)$ as follows:

\begin{itemize}

\item Let $M^C(i)$ be the set of leaves at $i$ and define $M^C=\bigcup\limits_{i=0}^4M^C(i)$.

\item Let $CL^C(i)$ be the set of clones at $i$ and define $CL^C=\bigcup\limits_{i=0}^4CL^C(i)$.

\item Let $A^C$ be the set of vertices anticomplete to $V(C)$.

\end{itemize}

\bigskip

\noindent By \ref{C5}, it follows that $V(G)=V(C)\cup M^C \cup CL^C\cup A^C$. Furthermore, we partition $A^C=X^C\cup Y^C\cup Z^C$, where

\begin{itemize}

\item $X^C$ is the set of vertices in $A^C$ with a neighbor in $M^C$,

\item $Y^C$ is the set of vertices in $A^C\setminus X^C$ with a neighbor in $CL^C$, and

\item $Z^C=A^C\setminus (X^C\cup Y^C)$.

\end{itemize}

\noindent Finally, we define subsets of $X^C$, $Y^C$ and $M^C$, for every $i\in\{0,...,4\}$ as follows:

\begin{itemize}

\item Let $X^C(i)$ be the set of vertices of $X^C$ with a neighbor in $M^C(i)$. 

\item Let $Y^C(i)$ be the set of vertices of $Y^C$ with a neighbor in $CL^C(i)$.

\item Let $M^C_i$ be the set of vertices of $M^C$ with a neighbor in $X^C(i)$. 

\end{itemize}

\bigskip

\noindent And so, for a given 5-gon $C$ in time $O(|V(G)|^2)$ we obtain the partition $V(C)\cup M^C\cup CL^C \cup X^C\cup Y^C \cup Z^C$ of $V(G)$. Now, we establish several properties of this partition. By definition and \ref{C5}, it follows that:

\begin{lemma}\label{5-1} If $G$ is a triangle-free graph, then for every 5-gon $C$ in $G$ the following hold:

\begin{enumerate}

\item Every vertex in $X^C$ has a neighbor in $M^C$. 

\item Every vertex in $Y^C$ has a neighbor in $CL^C$.

\item $Y^C$ is anticomplete to $M^C$.

\end{enumerate}

\end{lemma}

\noindent Recall that  for a fixed subset $X$ of $V(G)$, we say a vertex \textit{$v\in V(G)\setminus X$ is mixed on an edge of $X$}, if there exist adjacent $x,y\in X$ such that $v$ is mixed on $\{x,y\}$. 

\begin{lemma}\label{mixedAC5}

If $G$ is a $\{P_7,C_3\}$-free graph, then for every 5-gon $C$ in $G$ the following hold:

\begin{enumerate}

\item No vertex in $M^C$ is mixed on an edge of $A^C$.

\item $X^C$ is stable and anticomplete to $Y^C\cup Z^C$.

\item Both $M^C(i)$ and $CL^C(i)$ are stable for every $i\in \{0,...,4\}$.

\end{enumerate}
\end{lemma}

\begin{proof}

Let $C$ be a 5-gon in $G$ given by $v_0-v_1-v_2-v_3-v_4-v_0$. Suppose for adjacent $a,a'\in A^C$, there exists $m\in M^C$ which is adjacent to $a'$ and non-adjacent to $a$. By symmetry, we may assume $m\in M^C(0)$. However, then $a-a'-m-v_0-v_1-v_2-v_3$ is a $P_7$ in $G$, a contradiction. This proves \ref{mixedAC5}.1.

Consider a vertex $x\in X^C$. By \ref{5-1}.1, there exists $m\in M^C$ adjacent to $x$. If there exists $x'\in N(x)\cap A^C$, then, since $G$ is triangle-free $m$ is non-adjacent to $x'$, and it follows that $m$ is mixed on an edge of $A^C$, contradicting \ref{mixedAC5}.1. This proves \ref{mixedAC5}.2. 

For every $i\in \{0,...,4\}$, by definition $v_i$ is complete to $M^C(i)$ and $v_{i+1}$ is complete to $CL^C(i)$, where all indices are mod 5. Since $G$ is triangle-free, it follows that \ref{mixedAC5}.3 holds. This proves \ref{mixedAC5}.

\end{proof}

\begin{lemma}\label{5X}
Let $G$ be a clean, connected $\{P_7,C_3,C_7,shell\}$-free graph. Then for every 5-gon $C$ in $G$ and $i\in \{0,...,4\}$ the following hold:

\begin{enumerate}

\item $X^C(i)$ is anticomplete to $M^C\setminus M^C(i)$; in other
words, $M^C_i \subseteq M^C(i)$.

\item $X^C(0)\cup ...\cup X^C(4)$ gives a partition of $X^C$.

\item Every vertex in $X^C(i)$ has a neighbor in $CL^C(i)$.

\item $X^C(i)$ is anticomplete to $CL^C\setminus CL^C(i)$. 


\item $Y^C(i)$ is anticomplete to $CL^C(i+1)\cup CL^C(i-1)$, where all indices are mod 5.

\item $M^C_i$ is anticomplete to $V(G)\setminus (M^C_i\cup CL^C(i)\cup X^C(i))$.

\end{enumerate}
\end{lemma}

\begin{proof} Let $C$ be a 5-gon in $G$ given by $v_0-v_1-v_2-v_3-v_4-v_5-v_0$. It is enough to prove this statement for $i=0$. Let $x\in X^C(0)$, and let  $m\in M^C(0)$ be  adjacent to $x$. Then $x \in M^C_0$.
By \ref{mixedAC5}.2, $X^C$ is stable and anticomplete to $Y^C\cup Z^C$, and so it follows that $N(x)\subseteq M^C\cup CL^C$. 
Suppose there exists $m'\in N(x)\cap (M^C\setminus M^C_0)$. Since $G$ is triangle-free, $m$ is non-adjacent to $m'$. By symmetry, we may assume $m'\in M^C(3)\cup M^C(4)$. However, if $m'\in M^C(4)$, then $m'-x-m-v_0-v_1-v_2-v_3$ is a $P_7$ in $G$, and if $m'\in M^C(3)$, then $m-x-m'-v_3-v_2-v_1-v_0-m$ is a $C_7$ in $G$, in both cases, a contradiction. Hence, $x$ is anticomplete to $M^C\setminus M^C_0$. This proves \ref{5X}.1, which, by \ref{5-1}.1, immediately implies \ref{5X}.2. Since $v_0$ is complete to $M^C(0)$ and $G$ has no dominated vertices, it follows that there exists $c\in CL^C\setminus (CL^C(1)\cup CL^C(4))$ adjacent to $x$. Since $G$ is triangle-free, $c$ is non-adjacent to $m$. Suppose $c\notin CL^C(0)$. By symmetry, we may assume $c\in CL^C(2)$. However, then $v_0-m-x-c-v_3-v_4-v_0$ with $v_1$ is a shell in $G$, a contradiction. Hence, $x$ has a neighbor in $CL^C(0)$. This proves \ref{5X}.3. Now, we prove \ref{5X}.4 and \ref{5X}.5. We have already shown that $X^C(0)$ is anticomplete to $CL^C(2)\cup CL^C(3)$. Let $c\in CL^C(0)$ be adjacent to $z\in X^C(0)\cup Y^C(0)$. Suppose there exists $c'\in CL^C(1)\cup CL^C(4)$ adjacent to $z$. By symmetry, we may assume $c'\in CL^C(1)$. Since $G$ is triangle-free, $c'$ is non-adjacent to $c$. However, then $c-z-c'-v_2-v_3-v_4-c$ with $v_1$ is a shell in $G$, a contradiction. Hence, $X^C(0)$ is anticomplete to $CL^C\setminus CL^C(0)$, and $Y^C(0)$ is anticomplete to $CL^C(1)\cup CL^C(4)$. This proves \ref{5X}.4 and \ref{5X}.5.

Next we prove \ref{5X}.6. Recall $m$ is an arbitrary vertex of $M^C(0)$, and 
that $x \in X^C(0)$ is adjacent to  $m$ and to $c \in CL^C(0)$.
By definition, \ref{5-1}.3 and \ref{5X}.1, it follows that $M^C_0$ anticomplete to $(X^C\setminus X^C(0))\cup Y^C\cup Z^C$. Suppose there exists $m'\in M^C\setminus M^C_0$ adjacent to $m$. By \ref{mixedAC5}.3, it follows that $M^C(0)$, and thus $M^C_0$, is stable, and so it follows that $m'\in M^C\setminus M^C(0)$. Since $G$ is triangle-free, $x$ is non-adjacent to $m'$. By symmetry, we may assume $m'\in M^C(1)\cup M^C(2)$. However, if $m'\in M^C(1)$, then $x-m-m'-v_1-v_2-v_3-v_4$ is a $P_7$ in $G$, and if $m'\in M^C(2)$, then $m-m'-v_2-v_3-v_4-v_0-m$ with $v_1$ is a shell in $G$, in both cases, a contradiction. Hence, $M^C_0$ is anticomplete to $M^C\setminus M^C_0$. Finally, we show that $M^C_0$ is anticomplete to $CL^C\setminus CL^C(0)$. Since $v_0$ is complete to $M^C(0)$ and $G$ is triangle-free, it follows that $M^C_0$ is anticomplete to $CL^C(1)\cup CL^C(4)$. Suppose there exists $c''\in CL^C(2)\cup CL^C(3)$ adjacent to $m$. Since $G$ is triangle-free, $c''$ is anticomplete to $\{c,x\}$. By symmetry, we may assume $c''\in CL^C(2)$. However, then $x-m-c''-v_3-v_4-c-x$ with $v_0$ is a shell in $G$, a contradiction. Hence, $M^C_0$ is anticomplete to $CL^C\setminus CL^C(0)$. This proves \ref{5X}.6.


\end{proof}

\begin{lemma}\label{5comps}

Let $G$ be a clean, connected $\{P_7,C_3,C_7,S_7\}$-free graph. Then for every 5-gon $C$ in $G$ the following hold:

\begin{enumerate}

\item $Z^C$ is empty.

\item Every component of $Y^C$ has size two.

\item $Y^C(0),...,Y^C(4)$ are pairwise disjoint and anticomplete to 
each other.



\item Every component of $M^C_i\cup X^C(i)\cup Y^C(i)$ has size two for every $i\in\{0,...,4\}$.

\end{enumerate}
\end{lemma}

\begin{proof}

Let $C$ be a 5-gon in $G$ given by $v_0-v_1-v_2-v_3-v_4-v_5-v_0$. 

\bigskip

\noindent \textit{(1) For every $c\in CL^C$, there exists $b_1,b_2,b_3\in V(C)$ such that $b_1-b_2-b_3-c$ is an induced path.}

\bigskip

 \noindent Consider a vertex $c\in CL^C$. By symmetry, we may assume $c\in CL^C(0)$. And so, $v_3-v_2-v_1-c$ is the desired induced path. This proves \textit{(1)}.
 
\bigskip






\noindent By \ref{5-1}, \ref{mixedAC5}.2 and \textit{(1)}, we may apply \ref{3edge} letting $P=V(C)$, $Q=CL^C$, $R=Y^C$, $S=Z^C$, and $T=M^C\cup X^C$. It thus follows that \ref{5comps}.1 holds and that every component of $Y^C$ has size at most two.
Now, suppose $y \in Y^C$ is a singleton component of $Y^C$. By \ref{5-1}.2, there exists $c\in CL^C$ adjacent to $y$. By symmetry, we may assume $c\in CL^C(0)$, and so $y\in Y^C(0)$. By \ref{5X}.5, $y$ is anticomplete to $CL^C(1)\cup CL^C(4)$, and so, by \ref{5-1}.3, it follows that $y$ is anticomplete to $V(G)\setminus (CL^C(0)\cup CL^C(2)\cup CL^C(3))$. Since $v_{1}$ does not dominate $y$, it follows that $y$ has a neighbor in $CL^C(3)$. Since $v_{4}$ does not dominate $y$, it follows that $y$ has a neighbor in $CL^C(2)$. However, then $y$ has a neighbor in $CL^C(2)$ and in $CL^C(3)$, contrary to \ref{5X}.5. This proves \ref{5comps}.2.

\bigskip

\noindent \textit{(2) If $\{y,y'\}$ is the vertex set of a component of $Y^C$, then there exists a unique $i\in \{0,...,4\}$ such that vertices of $CL^C(i)$ have neighbors in $\{y,y'\}$.}

\bigskip

\noindent Suppose $\{y,y'\}$ is the vertex set of a component of $Y^C$. By \ref{5-1}.2, there exists $c\in CL^C$ adjacent to $y$. Since $G$ is triangle-free, $y'$ is non-adjacent to $c$. By symmetry, we may assume $c\in CL^C(0)$. Let $C'$ be the 5-gon given by $c-v_1-v_2-v_3-v_4-c$. It follows that $y\in M^{C'}(0)$, $y'\in X^{C'}(0)$ and, since $G$ is triangle-free, $CL^C(j)=CL^{C'}(j)$ for $j=0,2,3$. And so, by \ref{5X}.4 applied to $C'$, it follows that $y'$ has a neighbor in $CL^{C'}(0)=CL^C(0)$ and is anticomplete to $CL^C(2)\cup CL^C(3)$. In particular,
$y' \in Y^C(0)$. By \ref{5X}.5 applied to $C$, it follows that $y'$ is anticomplete to $CL^C(1)\cup CL^C(4)$. And so, it follows that $y'$ is anticomplete to $CL^C\setminus CL^C(0)$. But now reversing the roles of $y$ and $y'$, it follows that $y$ is anticomplete to $CL^C \setminus CL^C(0)$. This proves \textit{(2)}.

\bigskip 

\noindent By \ref{5-1}.2, every vertex $y\in Y^C$ has a neighbor in $CL^C$ and so \textit{(2)} implies that $Y^C(0), ...,Y^C(4)$ are pairwise disjoint and 
anticomplete to each other. This proves \ref{5comps}.3.

\bigskip

\noindent \textit{(3) Every component of $M^C_i\cup X^C(i)$ has size two for every $i\in\{0,...,4\}$.}

\bigskip

\noindent We may assume $i=0$. By \ref{mixedAC5}.3, it follows that $M^C(0)$, and thus $M^C_0$, is stable. 
By definition, \ref{5X}.3 and \ref{5X}.4, it follows that every vertex in $M^C_0\cup X^C(0)$ has a neighbor in $CL^C(0)\cup\{v_0\}$. By \ref{mixedAC5}.2, \ref{5X}.1, \ref{5X}.2 and \ref{5X}.4, it follows that $X^C(0)$ is anticomplete to $V(G)\setminus (X^C(0) \cup M^C_0\cup CL^C(0))$. By \ref{5X}.6, it follows that $M^C_0$ is anticomplete to $V(G)\setminus (M^C_0\cup CL^C(0)\cup X^C(0))$. Since $v_0-v_1-v_2-v_3$ is an induced path, by~\textit{(1)}, we may apply \ref{3edge} letting $P=V(C)\setminus \{v_0\}$, $Q=CL^C(0)\cup \{v_0\}$, $R=M^C_0\cup X^C(0)$, $S=\emptyset$, and $T=V(G)\setminus (V(C)\cup M^C_0\cup CL^C(0)\cup X^C(0))$. Hence, every component of $M^C_0\cup X^C$ has size at most two. However, since every vertex in 
$M^C_0$ has a neighbor in $X^C(0)$ and every vertex in $X^C(0)$ has a neighbor in $M^C_0$, it follows that \textit{(3)} holds.

\bigskip

\noindent By \ref{5-1}.3 and \ref{mixedAC5}.2, it follows that $M^C\cup X^C$ is anticomplete to $Y^C$. Hence, together \ref{5comps}.2 and \textit{(3)} imply that \ref{5comps}.4 holds. This proves \ref{5comps}.












\end{proof}

\noindent Now, we prove the main result of the section.

\begin{lemma}\label{5gon} There is an algorithm with the following specifications:

\bigskip

{\bf Input:} A clean, connected $\{P_7,C_3,C_7,shell\}$-free graph $G$ which contains a 5-gon.

\bigskip

{\bf Output:} A $3$-coloring of $G$, or a determination that none exists.

\bigskip

{\bf Running time:} $O(|V(G)|^6)$.

\end{lemma}

\begin{proof}

Let $C$ be a 5-gon in $G$ given by $v_0-v_1-v_2-v_3-v_4-v_0$; clearly $C$ can be found in time $O(|V(G)|^5)$. In time $O(|V(G)|^2)$, we can partition $V(G)=V(C)\cup CL^C \cup M^C\cup X^C \cup Y^C\cup Z^C$ as well as determine $X^C(i),Y^C(i)$ and $M^C_i$ for every $i\in \{0,...,4\}$. Since $G$ is clean, by \ref{5comps}.1, it follows that $Z^C$ is empty and, by \ref{5X}.2 and \ref{5comps}.3, we obtain the partitions $X^C(0)\cup ...\cup X^C(4)$ of $X^C$ and $Y^C(0)\cup ...\cup Y^C(4)$ of $Y^C$. 
Next, fix a $3$-coloring $c$ of $G[V(C)]$. By symmetry, we may assume $c(v_1)=c(v_3)$ and $c(v_2)=c(v_4)$. Define the order 3 palette $L^C_c$ of $G$ as follows:

$$L^C_c(v) = \begin{cases}
 \{c(v)\} & \text{, \hspace{2ex}if $v\in V(C)$} \\
 \{1,2,3\} & \text{, \hspace{2ex}otherwise}\\
\end{cases}$$

\noindent  Next, update the vertices in $CL^C\cup M^C$ with respect to $V(C)$. And so, $|L^C_c(v)| \leq 2$ for all $v \in V(G) \setminus (X^C\cup Y^C)$. Furthermore, $|L^C_c(v)|=2$ if and only if $v\in M^C\cup CL^C(2)\cup CL^C(3)$. Now, update the vertices in $X^C\cup Y^C$ with respect to $CL^C\cup M^C$. By \ref{5X}.3, it follows that $|L^C_c(v)|=3$ if and only if $v\in X^C(2)\cup Y^C(2)\cup X^C(3)\cup Y^C(3)$. 
For every $j\in\{2,3\}$, by \ref{5X}.3, every vertex of $M_j\cup X^C(j)\cup Y^C(j)$ has a neighbor in $CL^C(j)\cup \{v_j\}$ and, by \ref{5comps}.4, every component of $M_j\cup X^C(j)\cup Y^C(j)$ has size 2. Let $\mathcal{L}_1$ be the set of $O(|V(G)|^2)$ subpalettes of $L^C_c$ obtained from \ref{lemma1} applied with






\begin{itemize}

\item $x=v_1$, 
\item $S=CL^C(2)\cup\{v_2\}$, 
\item $\hat{A}\cup \hat{B}=M^C_2\cup X^C(2)\cup Y^C(2)$, 
\item $Y=V(G)\setminus (\{v_1,v_2\}\cup CL^C(2)\cup M_2\cup X^C(2)\cup Y^C(2))$,
and
\item $X= \emptyset$.

\end{itemize}

\noindent Next, for every $L\in \mathcal{L}_1$, let $\mathcal{L}(L)$ be the set of $O(|V(G)|^2)$ subpalettes of $L$ obtained from \ref{lemma1} applied with

\begin{itemize}
\item $x=v_4$, 
\item $S=CL^C(3)\cup\{v_3\}$, 
\item $\hat{A}\cup \hat{B}=M^C_3\cup X^C(3)\cup Y^C(3)$, and 
\item $Y=V(G)\setminus (\{v_3,v_4\}\cup CL^C(3)\cup M_3\cup X^C(3)\cup Y^C(3))$,
and 
\item $X=\emptyset$.
\end{itemize}

\noindent Finally, let $\mathcal{L}_c=\{\mathcal{L}(L):L\in\mathcal{L}_1\}$ be the set of $O(|V(G)|^4)$ subpalettes of $L^C_c$ thus obtained. By \ref{lemma1}, $\mathcal{L}_c$ can be computed in time $O(|V(G)|^6)$. Since
$X=\emptyset$, by \ref{lemma1}\textit{(b)}, we have that $(G,L^C_c)$ is colorable if and only if $(G,\mathcal{L}_c)$ is colorable. Let $\mathcal{L}$ be the union of the sets $\mathcal{L}_c$ taken over all $3$-colorings $c$ of $G[V(C)]$. Then $G$ is $3$-colorable if and only if
$(G, \mathcal{L})$ is $3$-colorable. 
Since $|L^C_c(v)| \leq 2$ for all $v \in V(G) \setminus (X^C(2)\cup Y^C(2)\cup X^C(3)\cup Y^C(3))$, by \ref{lemma1}\textit{(a)}, it follows that $|L(v)|\leq 2$ for all $L\in \mathcal{L}$ and $v\in V(G)$.  Thus, by~\ref{check}  we can test in time $O(|V(G)|^2)$ if $(G,L)$ is colorable for every $L\in\mathcal{L}$. Since there are at most $5^3$ $3$-colorings of $G[V(C)]$,  it follows that $\mathcal{L}$ consists of $O(|V(G)|^4)$ subpalettes of $L^C_c$, and so, 
via $O(|V(G)|^4)$ applications of \ref{check}, we can determine if $(G,\mathcal{L})$ is colorable. That is, in time $O(|V(G)|^6)$ we can determine if there exists a  $3$-coloring $c$ of $G[V(C)]$ that extends to a $3$-coloring of $G$, and give an explicit $3$-coloring $c'$ of $G$ such that $c'(v)=c(v)$ for all $v\in V(C)$, if one exists.  Since every $3$-coloring of $G$ restricts to a 
$3$-coloring of $G[V(C)]$, this proves \ref{5gon}.

\end{proof}
\section{Main Result}

In this section we prove the main result of this paper \ref{half1}, which we restate:

\begin{lemma}\label{main} There is an algorithm with the following specifications:

\bigskip

{\bf Input:} A $\{P_7,C_3\}$-free graph $G$.

\bigskip

{\bf Output:} A $3$-coloring of $G$, or a determination that none exists.

\bigskip

{\bf Running time:} $O(|V(G)|^{18})$.

\end{lemma}

\begin{proof}

By \ref{cleaning}, at the expense of carrying out a time $O(|V(G)|^5)$ procedure we may assume $G$ is clean. Via breadth-first search in time $O(|V(G)|^2)$ we can determine the components of $G$, and so we may also assume $G$ is connected. By \ref{7gon}, if $G$ contains a 7-gon, then in time $O(|V(G)|^{10})$ we can either produce a $3$-coloring of $G$, or determine that none exists. Hence, we may assume $G$ is a $\{P_7,C_3,C_7\}$-free graph. By \ref{shell}, if $G$ contains a shell, then in time $O(|V(G)|^{18})$ we can either produce a $3$-coloring of $G$, or determine that none exists. Hence, we may assume $G$ is a $\{P_7,C_3,C_7,shell\}$-free graph. By \ref{5gon}, if $G$ contains a 5-gon, then in time $O(|V(G)|^6)$ we can either produce a $3$-coloring of $G$, or determine that none exists. Hence, we may assume $G$ is a $\{P_7,C_3,C_5,C_7,shell\}$-free graph. Since $G$ is $P_7$-free, it follows that $G$ is $C_k$-free for all $k>7$. And so, $G$ is bipartite and in time $O(|V(G)|^3)$ we can produce a 2-coloring of $G$. This proves \ref{main}.

\end{proof}

\section{Acknowledgment}
We are grateful to Juraj Stacho for telling us about \ref{checkSubsets}, that
simplified  many of our arguments. We would also like to thank Paul Seymour for 
listening to (most of) the details of the proof of~\ref{lemma1}. Finally, we thank Alex Scott for telling us about the problem that this paper solves.

\end{document}